\documentclass{amsart}
\usepackage{a4wide,graphicx}

\newtheorem{theorem}{Theorem}[section]
\newtheorem{lemma}[theorem]{Lemma}
\newtheorem{proposition}[theorem]{Proposition}

\newtheorem{conjecture}[theorem]{Conjecture}

\numberwithin{equation}{section}
\numberwithin{figure}{section}

\title{Periodicity of certain piecewise affine planar maps}

\author[S.~Akiyama]{Shigeki~Akiyama}

\author[H.~Brunotte]{Horst~Brunotte}

\author[A.~Peth\H{o}]{Attila~Peth\H{o}}

\author[W.~Steiner]{Wolfgang~Steiner}

\date{\today}

\begin{document}

\begin{abstract}
We determine periodic and aperiodic points of certain piecewise affine maps in
the Euclidean plane. 
Using these maps, we prove for 
$\lambda\in\{\frac{\pm1\pm\sqrt5}2,\pm\sqrt2,\pm\sqrt3\}$ that all integer 
sequences $(a_k)_{k\in\mathbb Z}$ satisfying 
$0\le a_{k-1}+\lambda a_k +a_{k+1}<1$ are periodic.
\end{abstract}

\maketitle

\section{introduction}

In the past few decades, discontinuous piecewise affine maps have found 
considerable interest in the theory of dynamical systems. 
For an overview, we refer the reader to \cite{AKT,A,Goe00,Goe01,K,KLV}, for 
particular instances to \cite{VS,GS,T} (polygonal dual billiards), \cite{GH} 
(polygonal exchange transformations), \cite{CL88,WC,D,ACP} (digital filters) 
and \cite{LHV,LV98,LV} (propagation of round-off errors in linear systems). 
The present paper deals with a conjecture on the periodicity of a certain kind 
of these maps:

\begin{conjecture}\cite{ABPT,V}\label{cj}
For every real $\lambda$ with $|\lambda|<2$, all integer sequences
$(a_k)_{k\in\mathbb Z}$ satisfying
\begin{equation}\label{e1app}
0\le a_{k-1}+\lambda a_k +a_{k+1}<1
\end{equation}
for all $k\in\mathbb Z$ are periodic.
\end{conjecture}

This conjecture originated on the one hand from a discretization process in a 
rounding-off scheme occurring in computer simulation of dynamical systems (we 
refer the reader to \cite{LHV,V} and the literature quoted there), and on the 
other hand in the study of shift radix systems (see \cite{ABPT,ABBPT} for 
details). 
Extensive numerical evidence on the periodicity of integer sequences satisfying
(\ref{e1app}) was first observed in \cite{V94}.

We summarize the situation of the Conjecture \ref{cj}. 
Since we have approximately 
$$
\begin{pmatrix}a_k \\ a_{k+1}\end{pmatrix} \approx
\begin{pmatrix}0 & 1 \\ -1 & -\lambda\end{pmatrix} 
\begin{pmatrix}a_{k-1} \\ a_k\end{pmatrix}
$$
and the eigenvalues of the matrix are $\exp(\pm\theta\pi i)$ with 
$\theta\in [0,1]$, the sequence may be viewed as a discretized rotation on 
$\mathbb{Z}^2$, and it is natural to parametrize $-\lambda=2\cos(\theta\pi)$.
There are five different classes of $\lambda$ of apparently increasing 
difficulty:
\begin{enumerate}
\item $\theta$ is rational and $\lambda$ is rational.
\item $\theta$ is rational and $\lambda$ is quadratic.
\item $\theta$ is rational and $\lambda$ is cubic or of higher degree.
\item $\theta$ is irrational and $\lambda$ is rational.
\item None of the above.
\end{enumerate}

The first case consists of the three values $\lambda=-1,0,1$, where the
conjecture is trivially true.
Already in case (2) the problem is far from trivial.
A computer assisted proof for $-\lambda=\frac{\sqrt5-1}2$  was given by
Lowenstein, Hatjispyros and Vivaldi~\cite{LHV}.\footnote{Indeed, they showed 
that all trajectories of the map 
$(x,y)\mapsto(\lfloor(-\lambda)x\rfloor-y,x)$ on $\mathbb Z^2$ are periodic.}
A short proof (without use of computers) of the golden mean case 
$\lambda=\frac{1+\sqrt5}2$ was given by the authors~\cite{ABPS}.
The main goal of this paper is to settle the conjecture for all the cases of 
(2), i.e., the quadratic parameters
$$
\lambda=\frac{\pm1\pm\sqrt5}2,\,\pm\sqrt2,\,\pm\sqrt3.
$$
The proofs are sensitive to the choice of $\lambda$, and we have to work 
tirelessly in computation and drawings, especially in the last case 
$\pm\sqrt3$. 
However, an important feature of our proof is that it can basically be checked 
by hand. 
The (easiest) case $\frac{1+\sqrt{5}}{2}$ in Section~\ref{sectgolden} gives a 
prototype of our discussion and should help the reader to understand the idea 
for the remaining values.

For case (3), it is possible that Conjecture~\ref{cj} can be proved using the
same method, which involves a map on $[0,1)^{2d-2}$, where $d$ denotes the 
degree of $\lambda$.
However, it seems to be difficult in case $d\ge3$
to find self inducing structures, which are essential for this method.
In~\cite{LV}, a similar embedding into a higher dimensional torus is used for 
efficient orbit computations. 
Goetz \cite{Goe00,Goe01,Goe02} found a piecewise $\pi/7$ rotation on an 
isosceles triangle in a cubic case having a self inducing structure, but we do 
not see a direct connection to our problem.

The problem currently seems hopeless for cases (4) and (5).
However, a nice observation on rational values of $\lambda$ with prime-power 
denominator $p^n$ is exhibited in~\cite{BV}. 
The authors show that the dynamical system given by (\ref{e1app}) can be 
embedded into a $p$-adic rotation dynamics, by multiplying a $p$-adic unit. 
These investigations were extended in \cite{VV}. 
Furthermore, in \cite{V} the case $\lambda=q/p$ with $p$ prime was related to 
the concept of minimal modules, the lattices of minimal complexity which 
support periodic orbits.

Now we come back to the content of the present paper. 
The proof in~\cite{LHV} is based on a discontinuous non-ergodic piecewise 
affine map on the unit square, which dates back to Adler, Kitchens and 
Tresser~\cite{AKT}.
Let $\lambda^2=b\lambda+c$ with $b,c\in\mathbb Z$. 
Set $x=\{\lambda a_{k-1}\}$ and $y=\{\lambda a_k\}$, where 
$\{z\}=z-\lfloor z\rfloor$ denotes the fractional part of $z$. 
Then we have $a_{k+1}=-a_{k-1}-\lambda a_k+y$ and
$$
\{\lambda a_{k+1}\} = \{-\lambda a_{k-1}-\lambda^2 a_k+\lambda y\}=
\{-x+(\lambda-b) y\} = \{-x+cy/\lambda\} = \{-x-\lambda'y\},
$$
where $\lambda'$ is the algebraic conjugate of $\lambda$. 
Therefore we are interested in the map $T:[0,1)^2\to[0,1)^2$ given by 
$T(x,y)=(y,\{-x-\lambda'y\})$.
Obviously, it suffices to study the periodicity of $(T^k(z))_{k\in\mathbb Z}$ 
for points $z=(x,y)\in(\mathbb Z[\lambda]\cap[0,1))^2$ in order to prove the
conjecture.
Using this map, Kouptsov, Lowenstein and Vivaldi~\cite{KLV} showed for all 
quadratic $\lambda$ corresponding to rational rotations 
$\lambda=\frac{\pm1\pm\sqrt5}2,\pm\sqrt2,\pm\sqrt3$ that the trajectories of
almost all points are periodic, by heavy use of computers.
Of course, such metric results do not settle Conjecture~\ref{cj}, which
deals with countably many points in $[0,1)^2$, which may be exceptional.
The main goal of this article is to show that no point with aperiodic
trajectory has coordinates in $\mathbb Z[\lambda]$, which proves 
Conjecture~\ref{cj} for these eight values of $\lambda$.

This number theoretical problem is solved by introducing a map $S$, which is 
the composition of the first hitting map to the image of a suitably chosen self 
inducing domain under a (contracting) scaling map and the inverse of the 
scaling map.
A crucial fact is that the inverse of the scaling constant is a Pisot unit in 
the quadratic number field $\mathbb Q(\lambda)$.
This number theoretical argument greatly reduces the classification problem of 
periodic orbits, see e.g. Theorem~\ref{thmper}.
All possible period lengths can be determined explicitly and one can even 
construct concrete aperiodic points in $(\mathbb Q(\lambda)\cap[0,1))^2$.
We can associate to each aperiodic orbit a kind of $\beta$-expansion with 
respect to the scaling constant.
Note that the set of aperiodic points can be constructed similarly to 
a Cantor set, and that it is an open question of Mahler~\cite{M} 
whether there exist algebraic points in the triadic Cantor set.

The paper is organized as follows. 
In Section 2, we reprove the conjecture for the simplest non-trivial case, 
i.e., where $\lambda$ equals the golden mean. 
An exposition of our domain exchange method is given in Section~3, where the 
ideas of Section~2 are extended to a general setting.
In the subsequent seven sections we prove the conjecture for the cases 
$\lambda=-\gamma,\pm1/\gamma,\pm\sqrt2,\pm\sqrt3$.
Some parts of the proofs for $\lambda=\pm\sqrt3$ are put into the Appendix. 
We conclude this paper by an observation relating the famous Thue-Morse 
sequence to the trajectory of points for $\lambda=\pm\gamma,\pm1/\gamma,\sqrt3$.

\section{The case $\lambda=\gamma=\frac{1+\sqrt5}2=-2\cos\frac{4\pi}5$}
\label{sectgolden}

We consider first the golden mean $\lambda=\gamma=\frac{1+\sqrt5}2$,
$\lambda^2=\lambda+1$.
Note that $T$ is given by
\begin{equation}\label{eqTA}
T(x,y)=(x,y)A+(0,\lceil x-y/\gamma\rceil)\ \mbox{ with }\
A=\begin{pmatrix}0 & -1 \\ 1 & 1/\gamma\end{pmatrix}.
\end{equation}
Therefore, we have $T(x,y)=(x,y)A$ if $y\ge\gamma x$ and $T(z)=zA+(0,1)$
for the other points $z\in[0,1)^2$, see Figure~\ref{figT1}.
A particular role is played by the set
$$
\mathcal R=\{(x,y)\in[0,1)^2:\,y<\gamma x,\,x+y>1,\,x<y\gamma\}\cup\{(0,0)\}.
$$
If $z\in\mathcal R$, $z\ne(0,0)$, then we have $T^{k+1}(z)=T^k(z)A+(0,1)$ for
all $k\in\{0,1,2,3,4\}$, hence
$$
T^5(z) = zA^5+(0,1)(A^4+A^3+A^2+A^1+A^0) = z+(0,1)(A^5-A^0)(A-A^0)^{-1} = z
$$
since $A^5=A^0$.
It can be easily verified that the minimal period length is 5 for all
$z\in\mathcal R$ except
$(\frac{\gamma^2}{\gamma^2+1},\frac{\gamma^2}{\gamma^2+1})$ and $(0,0)$, which
are fixed points of $T$.
Therefore, it is sufficient to consider the domain 
$\mathcal D=[0,1)^2\setminus\mathcal R$ in the following.
According to the action of $T$, we partition $\mathcal D$ into two sets $D_0$
and $D_1$, with $D_0=\{(x,y)\in[0,1)^2:y\ge\gamma x\}\setminus\{(0,0)\}$, 

\begin{figure}
\includegraphics{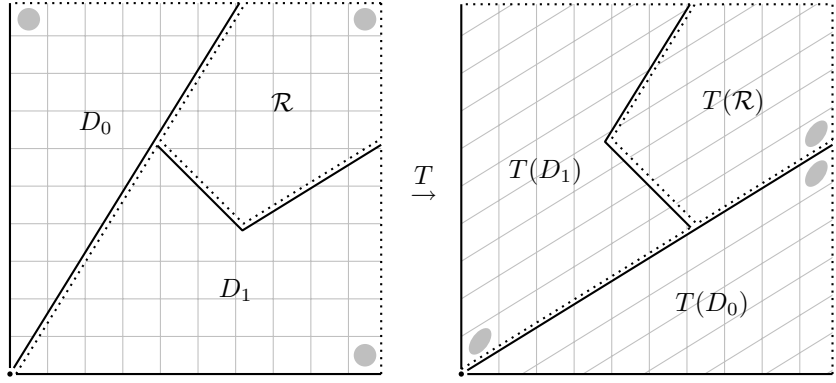}
\caption{The piecewise affine map $T$ and the set $\mathcal R$,
$\lambda=\gamma=\frac{1+\sqrt5}2$.} \label{figT1}
\end{figure}

\begin{figure}
\includegraphics{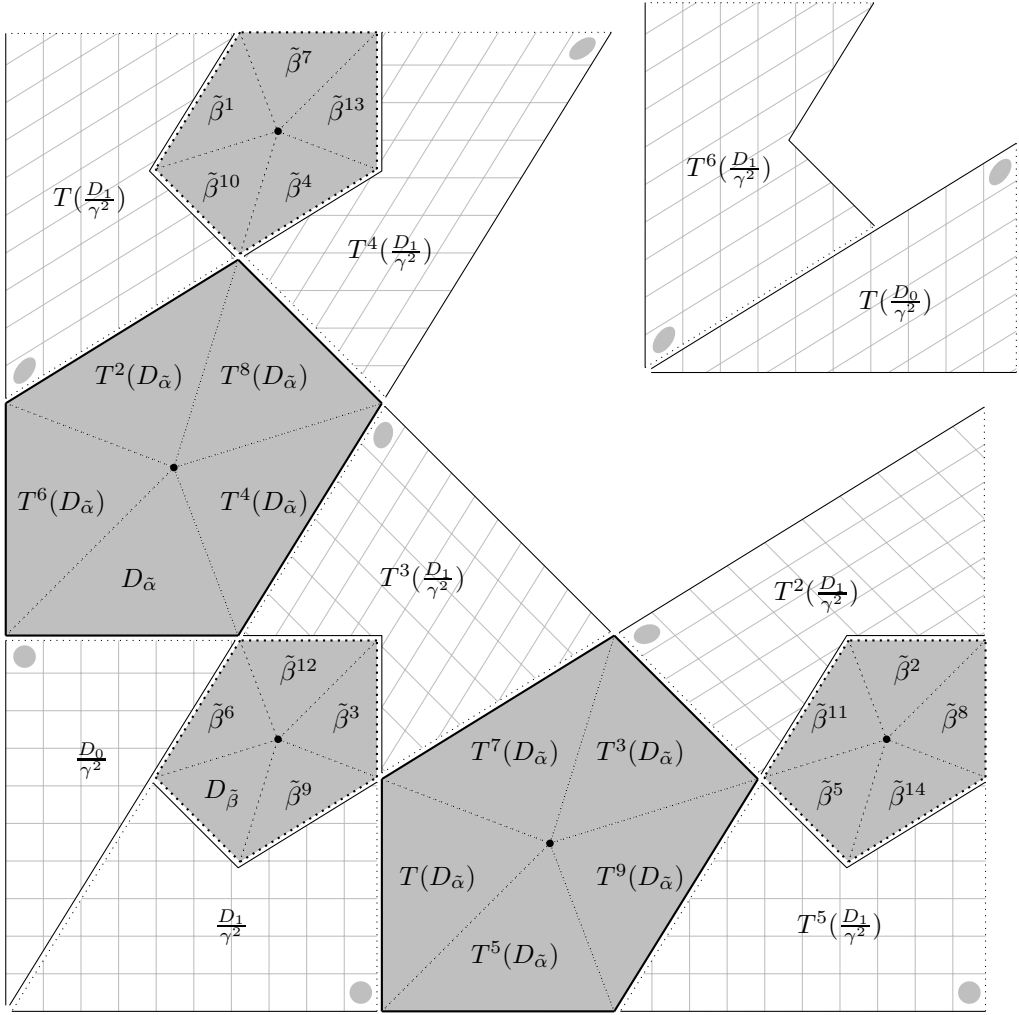}
\caption{The trajectory of the scaled domains and the (gray) set $\mathcal P$,
$\lambda=\gamma$. ($\tilde\beta^k$ stands for $T^k(D_{\tilde\beta})$.)}
\label{figP1}
\end{figure}

In Figure~\ref{figP1}, we scale $D_0$ and $D_1$ by the factor $1/\gamma^2$ and 
follow their $T$-trajectory until the return to $\mathcal D/\gamma^2$.
Let $\mathcal P$ be the set of points in $\mathcal D$ which are not
eventually mapped to $\mathcal D/\gamma^2$, i.e.,
$$
\mathcal P=D_\alpha\cup T(D_\alpha)\cup D_\beta\cup T(D_\beta)\cup T^2(D_\beta),
$$
where $D_\alpha$ is the closed pentagon
$\{(x,y)\in D_0:y\ge1/\gamma^2,x+y\le1,y\le(1+x)/\gamma\}$ and $D_\beta$ is the
open pentagon $\mathcal R/\gamma^2\setminus\{(0,0)\}$.
(In Figure~\ref{figP1}, $D_\alpha$ is split up into
$\{T^k(D_{\tilde\alpha}):k\in\{0,2,4,6,8\}\}$, and $D_\beta$ is split up
into $\{T^k(D_{\tilde\beta}):k\in\{0,3,6,9,12\}\}$.)
All points in $\mathcal P$ are periodic (with minimal period lengths $2,3,10$
or $15$).
Figures~\ref{figT1} and~\ref{figP1} show that the action of the first return
map on $\mathcal D/\gamma^2$ is similar to the action of $T$ on $\mathcal D$,
more precisely,
\begin{equation}\label{eqselfsimilar}
\frac{T(z)}{\gamma^2}=\left\{\begin{array}{cl}T(z/\gamma^2)&\mbox{if }z\in D_0,
\vspace{1mm} \\ T^6(z/\gamma^2)&\mbox{if }z\in D_1.\end{array}\right.
\end{equation}
For $z\in\mathcal D\setminus\mathcal P$, let
$s(z)=\min\{m\ge0:T^m(z)\in\mathcal D/\gamma^2\}$.
(Figure~\ref{figP1} shows $s(z)\le 5$.)
By the map
$$
S:\ \mathcal D\setminus\mathcal P\to\mathcal D,\quad
z\mapsto\gamma^2T^{s(z)}(z),
$$
we can completely characterize the periodic points.
For $z\in [0,1)^2$, denote by $\pi(z)$ the minimal period length if
$(T^k(z))_{k\in\mathbb Z}$ is periodic and set $\pi(z)=\infty$ else.

\begin{theorem}\label{thmper}
$(T^k(z))_{k\in\mathbb Z}$ is periodic if and only if $z\in\mathcal R$ or
$S^n(z)\in\mathcal P$ for some $n\ge0$.
\end{theorem}

We postpone the proof to Section~\ref{sectgeneral}, where the more general
Proposition~\ref{propperiodic} and Theorem~\ref{thmperiodic} are proved (with
$U(z)=z/\gamma^2$, $R(z)=z$, $\hat T(z)=T(z)$, $\hat\pi(z)=\pi(z)$,
and $z\in D_1$ or $T(z)\in D_1$ for all $z\in\mathcal D$,
$|\sigma^n(1)|\to\infty$, see below).

(\ref{eqselfsimilar}) and Figure~\ref{figP1} suggest to define a
substitution (or morphism) $\sigma$ on the alphabet $\mathcal A=\{0,1\}$, i.e.,
a map $\sigma:\mathcal A\to\mathcal A^*$ (where $\mathcal A^*$ denotes
the set of words with letters in $\mathcal A$), by
$$
\sigma:\quad 0\mapsto 0\qquad 1\mapsto 101101
$$
in order to code the trajectory of the scaled domains until their return to
$\mathcal D/\gamma^2$: We have
$T^{k-1}(D_\ell/\gamma^2)\subseteq D_{\sigma(\ell)[k]}$ and
$T^{|\sigma(\ell)|}(z/\gamma^2)=T(z)/\gamma^2$ for all $z\in
D_\ell$, where $w[k]$ denotes the $k$-th letter of the word $w$ and
$|w|$ denotes its length.
Furthermore, we have $T^k(D_\ell/\gamma^2)\cap\mathcal D/\gamma^2=\emptyset$
for $1\le k<|\sigma(\ell)|$.
Extend the definition of $\sigma$ naturally to words in $\mathcal A^*$ by
setting $\sigma(vw)=\sigma(v)\sigma(w)$, where $vw$ denotes the concatenation
of $v$ and $w$.
Then we get the following lemma, which resembles Proposition~1 by 
Poggiaspalla~\cite{P}.

\begin{lemma}
For every integer $n\ge0$ and every $\ell\in\{0,1\}$, we have
\begin{itemize}
\item
$T^{|\sigma^n(\ell)|}(z/\gamma^{2n})=T(z)/\gamma^{2n}$ for all $z\in D_\ell$,
\item
$T^{k-1}(D_\ell/\gamma^{2n})\subseteq  D_{\sigma^n(\ell)[k]}$ for all $k$,
$1\le k\le|\sigma^n(\ell)|$
\item
$T^k(D_\ell/\gamma^{2n})\cap\mathcal D/\gamma^{2n}=\emptyset$ for all $k$,
$1\le k<|\sigma^n(\ell)|$.
\end{itemize}
\end{lemma}

The proof is again postponed to Section~\ref{sectgeneral},
Lemma~\ref{lemsubstitution}.
This lemma allows to determine the minimal period lengths:
If $z\in D_\alpha$, then
$$
T^{|\sigma^n(0101010101)|}(z/\gamma^{2n})=
T^{|\sigma^n(101010101)|}(T(z)/\gamma^{2n})=
\cdots=T^{10}(z)/\gamma^{2n}=z/\gamma^{2n}
$$
for all $n\ge0$.
The only points of the form $T^k(z/\gamma^{2n})$, $1\le k\le5|\sigma^n(01)|$, 
which lie in $\mathcal D/\gamma^{2n}$ are the points $T^m(z)/\gamma^{2n}$, 
$1\le m\le9$, which are all different from $z/\gamma^{2n}$ if $\pi(z)=10$.
Therefore, we obtain $\pi(z/\gamma^{2n})=5|\sigma^n(01)|$ in this case.
A point $\tilde z$ lies in the trajectory of $z/\gamma^{2n}$ if and only if
$S^n(\tilde z)=T^m(z)$ for some $m\in\mathbb Z$, see Lemma~\ref{lemTm}.
This implies $\pi(\tilde z)=5|\sigma^n(01)|$ for these $\tilde z$ as well.
The period lengths of all points are given by the following theorem.

\begin{theorem}\label{thperiods1}
If $\lambda=\gamma$, then the minimal period lengths $\pi(z)$ of
$(T^k(z))_{k\in\mathbb Z}$ are \\
\centerline{ \begin{tabular}{cl}
$1$ & if $z=(0,0)$ or
$z=(\frac{\gamma^2}{\gamma^2+1},\frac{\gamma^2}{\gamma^2+1})$ \\
$5$ & if $z\in \mathcal R\setminus\{(0,0),
(\frac{\gamma^2}{\gamma^2+1},\frac{\gamma^2}{\gamma^2+1})\}$ \\
$(5\cdot 4^n+1)/3$ & if
$S^n(z)=T^m(\frac{1/\gamma}{\gamma^2+1},\frac2{\gamma^2+1})$ for some $n\ge0$,
$m\in\{0,1\}$ \\
$5(5\cdot 4^n+1)/3$ & if $S^n(z)\in T^m\big(D_\alpha\setminus
\{(\frac{1/\gamma}{\gamma^2+1},\frac2{\gamma^2+1})\}\big)$ for some $n\ge0$,
$m\in\{0,1\}$ \\
$(10\cdot 4^n-1)/3$ & if $S^n(z)=T^m(\frac1{\gamma^2+1},\frac1{\gamma^2+1})$
for some $n\ge0$, $m\in\{0,1,2\}$ \\
$5(10\cdot 4^n-1)/3$ & if $S^n(z)\in T^m\big(D_\beta\setminus
\{(\frac1{\gamma^2+1},\frac1{\gamma^2+1})\}\big)$ for some $n\ge0$,
$m\in\{0,1,2\}$ \\
$\infty$ & if $S^n(z)\in\mathcal D\setminus\mathcal P$ for all $n\ge0$.
\end{tabular} } \\
The minimal period length of $(a_k)_{k\in\mathbb Z}$ is
$\pi(\{\gamma a_{k-1}\},\{\gamma a_k\})$ (which does not depend on $k$).
\end{theorem}

\begin{proof}
By Theorem~\ref{thmper}, Proposition~\ref{propperiodic} and the remarks
preceding the theorem, it suffices to calculate $|\sigma^n(0)|$ and
$|\sigma^n(1)|$.
Clearly, we have $|\sigma^n(0)|=1$ for all $n\ge0$ and thus
$$
|\sigma^n(1)| = |\sigma^{n-1}(101101)| = 4|\sigma^{n-1}(1)|+2 =
4(5\cdot 4^{n-1}-2)/3+2 = (5\cdot 4^n-2)/3.
$$
If $S^n(z)\in T^m(D_\alpha)$, then $\pi(z)=|\sigma^n(01)|$ and
$\pi(z)=5|\sigma^n(01)|$ respectively.
If $S^n(z)\in T^m(D_\beta)$, then $\pi(z)=|\sigma^n(101)|$ and
$\pi(z)=5|\sigma^n(101)|$ respectively.
\end{proof}

Now consider aperiodic points $z\in[0,1)^2$, i.e.,
$S^n(z)\in\mathcal D\setminus\mathcal P$ for all $n\ge0$.
We can write
$$
S(z) = \gamma^2T^{s(z)}(z) = \gamma^2\big(zA^{s(z)}+t(z)\big)
$$
for some $t(z)$ by using (\ref{eqTA}).
Note that $T(z)=zA$ for $z\in D_0$ and $T(z)=zA+(0,1)$ for $z\in D_1$.
For $z\in\mathcal D/\gamma^2$, we have $s(z)=0$ and $t(z)=0$.
For $z\in T^k(D_1/\gamma^2)$, $1\le k\le 5$, we have $s(z)=6-k$,
$$
t(z)=\left\{\begin{array}{ll}
(0,1)&\mbox{if }s(z)\in\{1,2\}, \\
(0,1)A^2+(0,1)=(1/\gamma,1/\gamma^2)&\mbox{if }s(z)=3, \\
(0,1)A^3+(0,1)A^2+(0,1)=(0,-1/\gamma)&\mbox{if }s(z)\in\{4,5\}.
\end{array}\right.
$$
We obtain inductively
$$
S^n(z) = \gamma^{2n}zA^{s(z)+s(S(z))+\ldots+s(S^{n-1}(z))}+
\sum_{k=0}^{n-1}\gamma^{2(n-k)}t(S^k(z))A^{s(S^{k+1}(z))+\cdots+s(S^{n-1}(z))}.
$$
If $z\in\mathbb Q(\gamma)^2$, then we have 
\begin{gather*}
(S^n(z))' = 
\frac{\left(zA^{s(z)+s(S(z))+\cdots+s(S^{n-1}(z))}\right)'}{\gamma^{2n}}+
\sum_{k=0}^{n-1} \frac{\left(t(S^k(z))A^{s(S^{k+1}(z))+\cdots+s(S^{n-1}(z))}
\right)'}{\gamma^{2(n-k)}}\, \\
\left\|(S^n(z))'\right\|_\infty \le
\frac{\max_{h\in\mathbb Z}\|(zA^h)'\|_\infty}{\gamma^{2n}} +
\sum_{k=0}^{n-1}\frac{\max_{h\in\mathbb Z,w\in\mathcal D\setminus\mathcal P}
\|(t(w)A^h)'\|_\infty}{\gamma^{2{n-k}}}\,,
\end{gather*}
where $z'=(x',y')$ if $z=(x,y)$ and $x',y'$ are the algebraic conjugates of 
$x,y$.
Since
\begin{multline*}
t(z)A^h\in\big\{(0,0),\
(0,1),(1,1/\gamma),(1/\gamma,-1/\gamma),(-1/\gamma,-1),(-1,0), \\
(1/\gamma,1/\gamma^2),(1/\gamma^2,-1/\gamma^2),(-1/\gamma^2,-1/\gamma),
(-1/\gamma,0),(0,1/\gamma),\\
(0,-1/\gamma),(-1/\gamma,-1/\gamma^2),(-1/\gamma^2,1/\gamma^2),
(1/\gamma^2,1/\gamma),(1/\gamma,0)\big\}
\end{multline*}
and $zA^h$ takes only the values $z$, $zA$, $zA^2$, $zA^3$ and $zA^4$, we obtain
$$
\left\|(S^n(z))'\right\|_\infty\le
\frac{\max_{h\in\mathbb Z}\|(zA^h)'\|_\infty}{\gamma^{2n}} +
\sum_{k=0}^{n-1}\frac{\gamma^2}{\gamma^{2(n-k)}}<\frac{C(z)}{\gamma^{2n}}+\gamma
$$
for some constant $C(z)$.
If $z\in(\frac1Q\mathbb Z[\gamma])^2$ for some integer $Q\ge1$, then
$S^n(z)\in(\frac1Q\mathbb Z[\gamma])^2$.
Since there exist only finitely many points
$w\in(\frac1Q\mathbb Z[\gamma]\cap[0,1))^2$ with $\|w'\|_\infty<C(z)+\gamma$,
we must have $\|(S^n(z))'\|_\infty\le\gamma$ for some $n\ge0$, which proves the
following proposition.

\begin{proposition}\label{propfinite}
Let $z\in(\frac1Q\mathbb Z[\gamma]\cap[0,1))^2$ be an aperiodic point.
Then there exists an aperiodic point
$\tilde z\in(\frac1Q\mathbb Z[\gamma])^2\cap\mathcal D$ with
$\|\tilde z'\|_\infty\le\gamma$.
\end{proposition}

For every denominator $Q\ge1$, it is therefore sufficient to check the
periodicity of the (finite set of) points
$z\in(\frac1Q\mathbb Z[\gamma])^2\cap\mathcal D$ with $\|z'\|_\infty\le\gamma$
in order to determine if all points in $(\frac1Q\mathbb Z[\gamma]\cap[0,1))^2$
are periodic.

For $Q=1$, we have to consider $z=(x,y)\in\mathcal D$ with 
$x,y\in\mathbb Z[\gamma]\cap[0,1)$ and $|x'|,|y'|\le\gamma$, hence 
$(x,y)\in\{0,1/\gamma\}^2$.
Since $(0,0)$ and $(1/\gamma,1/\gamma)$ are in $\mathcal R$, it only remains to
check the periodicity of $(0,1/\gamma)$ and $(1/\gamma,0)$.
These two points lie in $\mathcal P$, thus Conjecture~\ref{cj} is proved for
$\lambda=\gamma$.

For $Q=2$, a short inspection shows that all points 
$z\in(\frac12\mathbb Z[\gamma]\cap[0,1))^2$ are periodic as well.
The situation is completely different for $Q=3$, and we have
\begin{align*}
S(0,1/3)&=(0,\gamma^2/3),\qquad
S(0,\gamma^2/3)=\gamma^2\big((0,\gamma^2/3)A^5+(0,-1/\gamma)\big)=(0,2/3), \\
S(0,2/3)&=\gamma^2\big((0,2/3)A^5+(0,-1/\gamma)\big)=\big(0,1/(3\gamma^2)\big),
\quad S^4(0,1/3)=S\big(0,1/(3\gamma^2)\big)=(0,1/3).
\end{align*}
This implies $S^n(0,1/3)\in\mathcal D\setminus\mathcal P$ for all $n\ge0$
and $\pi(0,1/3)=\infty$ by Theorem~\ref{thperiods1}.

\begin{theorem}
$\pi(z)$ is finite for all points $z\in(\mathbb Z[\gamma]\cap[0,1))^2$,
but $(T^k(0,1/3))_{k\in\mathbb Z}$ is aperiodic.
\end{theorem}

\begin{figure}
\begin{minipage}{7.3cm}
\centerline{\includegraphics[scale=.8]{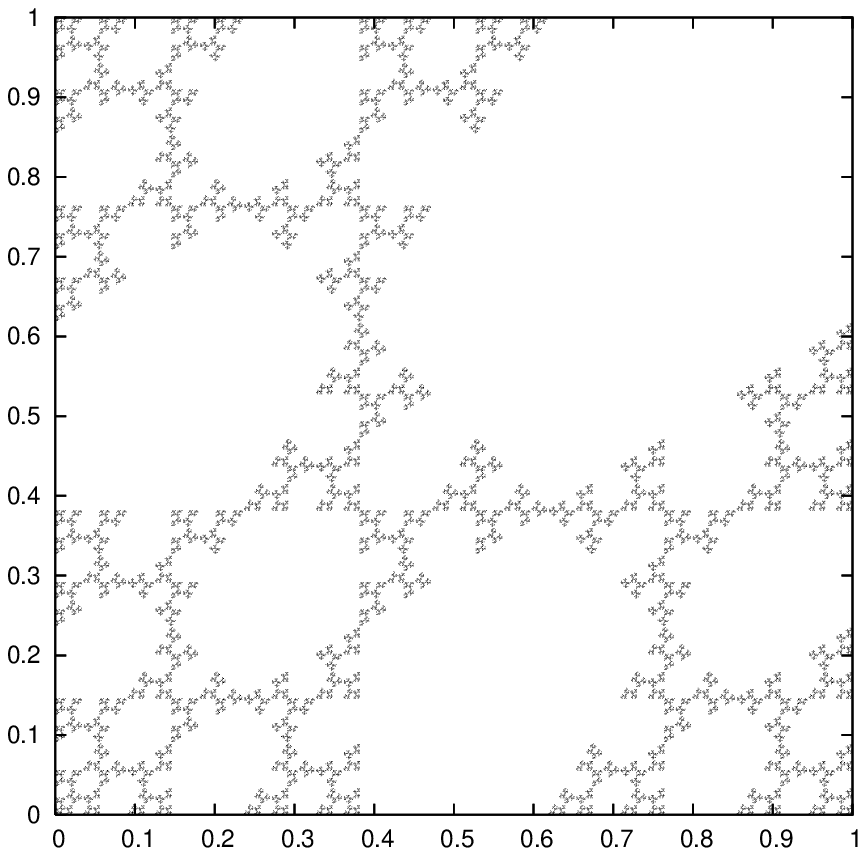}}
\caption{\mbox{Aperiodic points, $\lambda=\gamma$.}}\label{fig1a}
\end{minipage}
\begin{minipage}{7.3cm}
\centerline{\includegraphics[scale=.8]{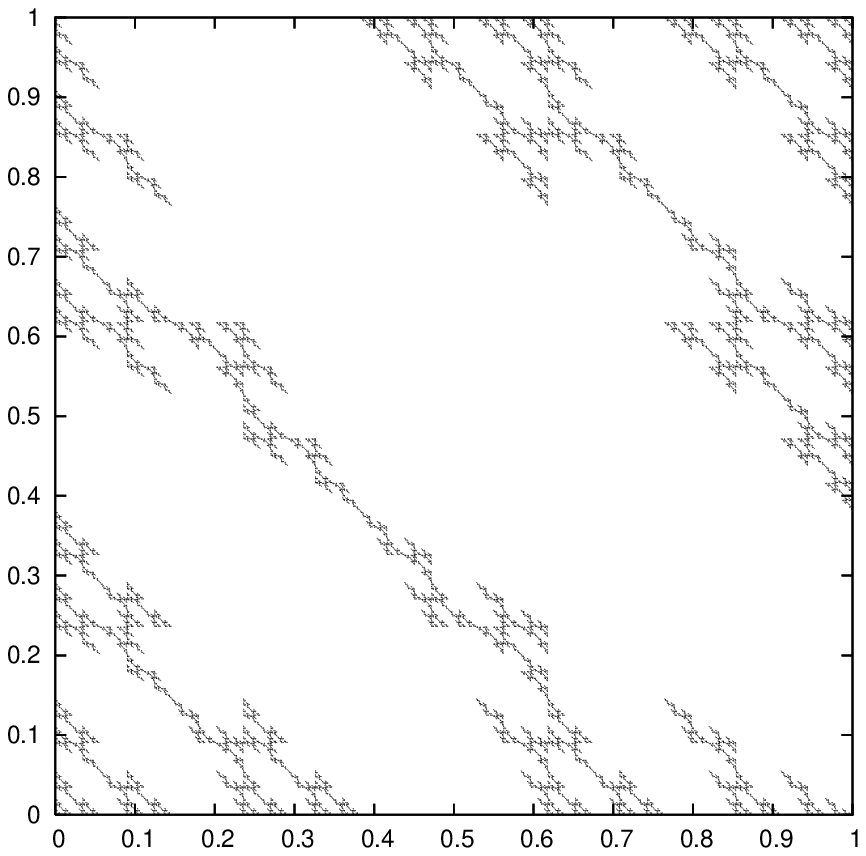}}
\caption{\mbox{Aperiodic points, $\lambda=-1/\gamma$.}}\label{fig2a}
\end{minipage}
\end{figure}

\section{General description of the method} \label{sectgeneral}
In this section, we generalize the method presented in Section~\ref{sectgolden}
in order to make it applicable for 
$\lambda=-\gamma,\pm1/\gamma,\pm\sqrt2,\pm\sqrt3$.

For the moment, we only need that $T:X\to X$ is a bijective map on a set $X$.
Fix $\mathcal D\subseteq X$, let
$$
\mathcal R= \{z\in X:\ T^m(z)\not\in\mathcal D\mbox{ for all }m\ge0\}
$$
set $r(z)=\min\{m\ge0:T^m(z)\in\mathcal D\}$ for $z\in X\setminus\mathcal R$,
and
$$
R:\ X\setminus\mathcal R\to\mathcal D,\qquad R(z)=T^{r(z)}(z).
$$
Let $\hat T$ be the first return map (of the iterates by $T$) on $\mathcal D$,
i.e.,
$$
\hat T:\ \mathcal D\to\mathcal D,\qquad \hat T(z)=RT(z)=T^{r(T(z))+1}(z),
$$
in particular $\hat T(z)=T(z)$ if $T(z)\in\mathcal D$.
Let $\mathcal A$ be a finite set, $\{D_\ell:\ell\in\mathcal A\}$ a partition
of $\mathcal D$ and define a coding map
$\iota:\mathcal D\to\mathcal A^{\mathbb Z}$ by
$\iota(z)=(\iota_k(z))_{k\in\mathbb Z}$ such that
$\hat T^k(z)\in D_{\iota_k(z)}$ for all $k\in\mathbb Z$.
Let $U:\mathcal D\to\mathcal D$, $\varepsilon\in\{-1,1\}$ and $\sigma$ a
substitution on $\mathcal A$ such that, for every $\ell\in\mathcal A$ and
$z\in D_\ell$,
$$
U\hat T(z)=\hat T^{\varepsilon|\sigma(\ell)|}U(z),
$$
$\hat T^{\varepsilon k}U(z)\not\in U(\mathcal D)$ for all $k$,
$1\le k<|\sigma(\ell)|$, and
$$
\sigma(\ell)=\left\{\begin{array}{ll} \iota_0(U(z))\,\iota_1(U(z))\,\cdots\,
\iota_{|\sigma(\ell)|-1}(U(z)) & \mbox{if }\varepsilon=1, \\
\iota_{-|\sigma(\ell)|}(U(z))\,\cdots\,\iota_{-2}(U(z))\,\iota_{-1}(U(z)) &
\mbox{if }\varepsilon=-1. \end{array}\right.
$$
Then the following lemma holds.

\begin{lemma}\label{lemsubstitution}
For every integer $n\ge0$, every $\ell\in\mathcal A$ and $z\in D_\ell$, we have
$$
U^n\hat T(z)=\hat T^{\varepsilon^n|\sigma^n(\ell)|}U^n(z),
$$
$\hat T^{\varepsilon^nk}U^n(z)\not\in U^n(\mathcal D)$ for all $k$,
$1\le k<|\sigma^n(\ell)|$, and
$$
\begin{array}{cl}
\iota_0(U^n(z))\,\iota_1(U^n(z))\, \cdots\,\iota_{|\sigma^n(\ell)|-1}(U^n(z)) =
\sigma^n(\ell) & \mbox{if }\varepsilon=1, \\
\iota_0(U^n(z))\,\iota_1(U^n(z))\, \cdots\,\iota_{|\sigma^n(\ell)|-1}(U^n(z)) =
(\sigma\bar\sigma)^{n/2}(\ell) & \mbox{if }\varepsilon=-1,\varepsilon^n=1, \\
\iota_{-|\sigma^n(\ell)|}(U^n(z))\,\cdots\,\iota_{-2}(U^n(z))\,
\iota_{-1}(U^n(z))=(\sigma\bar\sigma)^{(n-1)/2}\sigma(\ell) &
\mbox{if }\varepsilon=-1,\varepsilon^n=-1,
\end{array}
$$
where $\bar\sigma(\ell)=\ell_m\cdots\ell_2\ell_1$ if
$\sigma(\ell)=\ell_1\ell_2\cdots\ell_m$.
\end{lemma}

\begin{proof}
The lemma is trivially true for $n=0$, and for $n=1$ by the assumptions on
$\sigma$.
If we suppose inductively that it is true for $n-1$, then let
$\sigma(\ell)=\ell_1\ell_2\cdots\ell_m$ if $\varepsilon=1$,
$\sigma(\ell)=\ell_m\cdots\ell_2\ell_1$ if $\varepsilon=-1$, and we obtain
(by another induction) for all $j$, $1\le j\le m$,
\begin{equation}\label{eqj}
\hat T^{\varepsilon^n|\sigma^{n-1}(\ell_1\cdots\ell_{j-1}\ell_j)|}U^n(z) =
\hat T^{\varepsilon^n|\sigma^{n-1}(\ell_j)|}U^{n-1}\hat T^{\varepsilon(j-1)}U(z)
= U^{n-1}\hat T^{\varepsilon j}U(z).
\end{equation}
If $\varepsilon=1$, then this follows immediately from the induction hypothesis;
if $\varepsilon=-1$, then this follows by setting $k=|\sigma^{n-1}(\ell_j)|$ in
\begin{equation}\label{eqk}
\hat T^{(-1)^nk}U^{n-1}\hat T\big(\hat T^{-j}U(z)\big) =
\hat T^{(-1)^n(k-|\sigma^{n-1}(\ell_j)|)}U^{n-1}\hat T^{-j}U(z).
\end{equation}
Therefore, we have
$$
\hat T^{\varepsilon^n|\sigma^n(\ell)|}U^n(z) =
\hat T^{\varepsilon^n|\sigma^{n-1}(\ell_1\cdots\ell_{m-1}\ell_m)|}U^n(z) =
U^{n-1}\hat T^{\varepsilon m}U(z) =
U^{n-1}\hat T^{\varepsilon|\sigma(\ell)|}U(z) = U^n\hat T(z).
$$

If $\varepsilon=1$, then (\ref{eqj}) implies that
\begin{multline*}
\iota_0(U^n(z))\cdots\iota_{|\sigma^n(\ell)|-1}(U^n(z)) =
\big(\iota_0(U^{n-1}U(z))\cdots
\iota_{|\sigma^{n-1}(\ell_1)|-1}(U^{n-1}U(z))\big)\cdots \\
\big(\iota_0(U^{n-1}\hat T^{m-1}U(z))\cdots
\iota_{|\sigma^{n-1}(\ell_m)|-1}(U^{n-1}\hat T^{m-1}U(z))\big) =
\sigma^{n-1}(\ell_1)\cdots\sigma^{n-1}(\ell_m)=\sigma^n(\ell);
\end{multline*}
if $\varepsilon=-1$ and $\varepsilon^n=1$, then (\ref{eqj}) and (\ref{eqk})
provide
\begin{align*}
\iota_0(U^n(z))\cdots\iota_{|\sigma^n(\ell)|-1}(U^n(z)) & =
\big(\iota_{-|\sigma^{n-1}(\ell_1)|}(U^{n-1}T^{-1}U(z))\cdots
\iota_{-1}(U^{n-1}T^{-1}U(z))\big) \\
& \quad\cdots\big(\iota_{-|\sigma^{n-1}(\ell_m)|}(U^{n-1}\hat T^{-m}U(z))\cdots
\iota_{-1}(U^{n-1}\hat T^{-m}U(z))\big) \\
& = (\sigma\bar\sigma)^{(n-2)/2}\sigma(\ell_1)\cdots
(\sigma\bar\sigma)^{(n-2)/2}\sigma(\ell_m) = (\sigma\bar\sigma)^{n/2}(\ell);
\end{align*}
if $\varepsilon=-1$ and $\varepsilon^n=-1$, then
\begin{align*}
\iota_{-|\sigma^n(\ell)|}(U^n(z))\cdots\iota_{-1}(U^n(z)) & =
\big(\iota_0(U^{n-1}T^{-m}U(z))\cdots
\iota_{|\sigma^{n-1}(\ell_m)|-1}(U^{n-1}T^{-m}U(z))\big) \\
& \quad\cdots\big(\iota_0(U^{n-1}\hat T^{-1}U(z))\cdots
\iota_{|\sigma^{n-1}(\ell_1)|}(U^{n-1}\hat T^{-1}U(z))\big) \\
& = (\sigma\bar\sigma)^{(n-1)/2}(\ell_m)\cdots
(\sigma\bar\sigma)^{(n-1)/2}(\ell_1) = (\sigma\bar\sigma)^{(n-1)/2}\sigma(\ell).
\end{align*}

By (\ref{eqj}), (\ref{eqk}) and the induction hypothesis, the only points in
$(\hat T^{\varepsilon^nk}U^n(z))_{1\le k<|\sigma^n(\ell)|}$ lying in
$U^{n-1}(\mathcal D)$ are $U^n\hat T^{\varepsilon j}(z)$,
$1\le j<|\sigma(\ell)|$.
Since $\hat T^{\varepsilon j}(z)\not\in U(\mathcal D)$ for these $j$,
the lemma is proved.
\end{proof}

\noindent{\em Remark.}
If $\tilde z=\hat T^{-1}(z)\in D_\ell$, then
$U^n\hat T(\tilde z)=\hat T^{\varepsilon^n|\sigma^n(\ell)|}U^n(\tilde z)$,
thus $U^n\hat T^{-1}(z)=T^{-\varepsilon^n|\sigma^n(\ell)|}U^n(z)$.

\medskip
As in Section~\ref{sectgolden}, a key role will be played by the map $S$.
Assume that $U$ is injective, let
$$
\mathcal P=\{z\in\mathcal D:\ \hat T^m(z)\not\in U(\mathcal D)\mbox{ for all }
m\in\mathbb Z\},
$$
fix $\hat s(z)=\min\{m\ge0:\hat T^m(z)\in U(\mathcal D)\}$ or
$\hat s(z)=\max\{m\le0:\hat T^m(z)\in U(\mathcal D)\}$ for every
$z\in\mathcal D\setminus\mathcal P$, let $s(z)\in\mathbb Z$ be such that
$\hat T^{\hat s(z)}(z)=T^{s(z)}(z)$, and define
$$
S:\ \mathcal D\setminus\mathcal P \to\mathcal D,\qquad
z\mapsto U^{-1}\hat T^{\hat s(z)}(z)=U^{-1}T^{s(z)}(z).
$$

\noindent{\em Remark.}
Allowing $s(z)$ and $\hat s(z)$ to be negative decreases the $\delta$ in
Proposition~\ref{propaper} in some cases.

\begin{lemma}\label{lemTm}
If $S^nR(z)$ exists, then we have some $m\ge0$ such that $U^nS^nR(z)=T^m(z)$,
and
$$
\tilde z=T^m(z)\mbox{ for some }m\in\mathbb Z\mbox{ if and only if }
S^nR(\tilde z)=\hat T^kS^nR(z)\mbox{ for some }k\in\mathbb Z.
$$
\end{lemma}

\begin{proof}
Suppose that $S^nR(z)$ exists.
Then we have
$$
U^n\!S^n\!R(z)=U^{n-1}\hat T^{\hat s(S^{n-1}R(z))}S^{n-1}R(z)=
\hat T^{m_1}U^{n-1}S^{n-1}R(z)=\cdots=\hat T^{m_1+\cdots+m_n}R(z)=T^m(z)
$$
for some $m_1,\ldots,m_n,m\ge0$.

If $S^nR(\tilde z)=\hat T^kS^nR(z)$ for some $k\in\mathbb Z$, then let
$m_1,m_2\ge0$ be such that $U^nS^nR(z)=T^{m_1}(z)$,
$U^nS^nR(\tilde z)=T^{m_2}(\tilde z)$, and we have
$$
T^{m_2}(\tilde z) = U^nS^nR(\tilde z) = U^n\hat T^kS^nR(z) =
\hat T^{k_1}U^nS^nR(z) = T^{k_2+m_1}(z)
$$
for some $k_1,k_2\in\mathbb Z$, hence $\tilde z=T^m(z)$ with $m=k_2+m_1-m_2$.

If $\tilde z=T^m(z)$ for some $m\in\mathbb Z$ and $n=0$, then we have
$S^nR(\tilde z)=\hat T^{k_n}S^nR(z)$ for some $k_n\in\mathbb Z$.
If we suppose inductively that this is true for $n-1$, then
$$
S^nR(\tilde z) = S\hat T^{k_{n-1}}S^{n-1}R(z) =
S\hat T^{k_{n-1}-\hat s(S^{n-1}R(z))}US^nR(z) = SU\hat T^{k_n}S^nR(z) =
\hat T^{k_n}S^nR(z)
$$
for some $k_{n-1},k_n\in\mathbb Z$, and the statement is proved.
\end{proof}

If $rT$ is constant on every $D_\ell$, $\ell\in\mathcal A$, then we can define
$\tau:\mathcal A\to\mathbb N$ by $\tau(\ell)=r(T(z))+1$ for $z\in D_\ell$
(cf. the definition of $\hat T$) and extend $\tau$ naturally to words
$w\in\mathcal A^*$ by $\tau(w)=\sum_{\ell\in\mathcal A}|w|_\ell\tau(\ell)$.

Let $\pi(z)$, $\hat\pi(z)$ be the minimal period lengths of
$(T^k(z))_{k\in\mathbb Z}$ and $(\hat T^k(z))_{k\in\mathbb Z}$ respectively,
with $\pi(z)=\infty$, $\hat\pi(z)=\infty$ if the sequences are aperiodic.
Then the following proposition holds.

\begin{proposition}\label{propperiodic}
If $\hat\pi(S^nR(z))=p$ and
$\ell_1\cdots\ell_p=\iota_0(S^nR(z))\cdots\iota_{p-1}(S^nR(z))$, then we have
$$
\hat\pi(R(z))=|\sigma^n(\ell_1\ell_2\cdots\ell_p)|\quad\mbox{and}\quad \pi(z)=
\tau(\sigma^n(\ell_1\ell_2\cdots\ell_p))\ \mbox{(if $\tau$ is well defined)}.
$$
\end{proposition}

\begin{proof}
Since $U^nS^nR(z)=T^m(z)=\hat T^{\hat m}R(z)$ for some $m,\hat m\in\mathbb Z$,
and
$$
T^{\tau(\sigma^n(\ell_1\ell_2\cdots\ell_p))}U^nS^nR(z)=
\hat T^{|\sigma^n(\ell_1\ell_2\cdots\ell_p)|}U^nS^nR(z)=
U^n\hat T^pS^nR(z)=U^nS^nR(z),
$$
we have $\hat\pi(R(z))\le|\sigma^n(\ell_1\cdots\ell_p)|$ and
$\pi(z)\le\tau(\sigma^n(\ell_1\cdots\ell_p))$ (if $\tau$ exists).
Since $p$ is minimal, we can show similarly to the proof of
Lemma~\ref{lemsubstitution} that these period lengths are minimal.
\end{proof}

We obtain the following characterization of periodic points
$z\not\in\mathcal R$.
Note that all points in $\mathcal P\cup\mathcal R$ are periodic in our cases,
hence the characterization is complete.

\begin{theorem}\label{thmperiodic}
Let $R,S,T,\mathcal D,\mathcal P,\mathcal R,\sigma$ be as in the preceding 
paragraphs of this section. 
Assume that $\hat\pi(z)$ is finite for all $z\in\mathcal P$, and that for every 
$z\in\mathcal D\setminus\mathcal P$ there exist $m\in\mathbb Z$, 
$\ell\in\mathcal A$, such that $\hat T^m(z)\in D_\ell$ and
$|\sigma^n(\ell)|\to\infty$ for $n\to\infty$. 
Then we have for $z\not\in\mathcal R$: 
$$
(T^k(z))_{k\in\mathbb Z}\mbox{ is periodic if and only if } 
S^nR(z)\in\mathcal P\mbox{ for some }n\ge0.
$$
\end{theorem}

\begin{proof}
If $S^nR(z)\in\mathcal P$, then we have
$\hat\pi(R(z))=\hat\pi(S^nR(z))<\infty$, which implies $\pi(z)<\infty$.

Suppose now that $S^nR(z)\in\mathcal D\setminus\mathcal P$ for all $n\ge0$.
Then we have $m_n\in\mathbb Z$ and $\ell_n\in\mathcal A$ such that
$\hat T^{m_n}S^nR(z)\in D_{\ell_n}$ and $|\sigma^n(\ell_n)|\to\infty$ for
$n\to\infty$ (because $\mathcal A$ is finite).
We have
$U^n\hat T^{m_n}S^nR(z)=\hat T^{\tilde m_n}U^nS^nR(z)\in U^n(D_{\ell_n})$ for
some $\tilde m_n\in\mathbb Z$, hence
$\hat T^{\tilde m_n+k}U^nS^nR(z)\not\in U^n(\mathcal D)$ for all $k$,
$1\le k<|\sigma^n(\ell_n)|$, which implies
$\pi(z)\ge\hat\pi(R(z))=\hat\pi(U^nS^nR(z))\ge|\sigma^n(\ell_n)|$ for all
$n\ge0$, thus $\pi(z)=\infty$.
\end{proof}

Assume now $\lambda\in\{\pm\sqrt2,\frac{\pm1\pm\sqrt5}2,\pm\sqrt3\}$, let
$\lambda'$ be its algebraic conjugate, $T:[0,1)^2\to[0,1)^2$,
\begin{gather}
T(x,y)=(x,y)A+(0,\lceil x+\lambda'y\rceil)\ \mbox{ with }\
A=\begin{pmatrix}0 & -1 \\ 1 & -\lambda'\end{pmatrix}, \label{eqTAgen} \\
U(z)=V^{-1}(\kappa V(z)) \nonumber
\end{gather}
with $0<\kappa<1$, $\kappa\in\mathbb Z[\lambda]$, $|\kappa\kappa'|=1$, and
$V(z)=\pm\kappa^n(z-v)$ some 
$v\in\mathbb Z[\lambda]^2$, $n\in\mathbb Z$.
Let
$$
t(z)=V\big(T^{s(z)}(z)\big)-V(z)A^{s(z)}
$$
for $z\in\mathcal D\setminus\mathcal P$.
Since $U^{-1}(z)=V^{-1}(V(z)/\kappa)$, we have
$$
S(z)=U^{-1}T^{s(z)}(z)=V^{-1}\left(\frac{V(z)A^{s(z)}+t(z)}\kappa\right).
$$
Note that $A^h=A^0$ for some $h\in\{5,8,10,12\}$,
$$
T^{-1}(x,y)=(x,y)A^{-1}+(\lceil\lambda'x+y\rceil,0)\ \mbox{ with }\
A^{-1}=\begin{pmatrix}-\lambda' & 1 \\ -1 & 0\end{pmatrix},
$$
and $T^{-1}(x,y)=(\tilde x,\tilde y)$ with $(\tilde y,\tilde x)=T(y,x)$.
Since $|\hat s(z)|<\max_{\ell\in\mathcal A}|\sigma(\ell)|$, there exists only a
finite number of values for $t(z)$, and we obtain the following proposition.

\begin{proposition}\label{propaper}
Let $T,V,\kappa$ be as above and the assumptions of Theorem~\ref{thmperiodic}
be satisfied.
Suppose that $\pi(z)=\infty$ for some
$z\in(\frac1Q\mathbb Z[\lambda]\cap[0,1))^2\setminus\mathcal R$, where $Q$ is a
positive integer.
Then there exists an aperiodic point
$\tilde z\in(\frac1Q\mathbb Z[\lambda])^2\cap\mathcal D$ with
$$
\|V(\tilde z)'\|_\infty\le\delta,\quad \mbox{where }
\delta=\frac{\max\{\|(t(z)A^h)'\|_\infty:\,z\in\mathcal D\setminus\mathcal P,\,
\pi(z)=\infty,\, h\in\mathbb Z\}}{|\kappa'|-1}\,.
$$
\end{proposition}

\begin{proof}
First note that $\delta$ exists since $t(z)$ and $A^h$ take only finitely many
values.
If $\pi(z)=\infty$ for some
$z\in(\frac1Q\mathbb Z[\lambda]\cap[0,1))^2\setminus\mathcal R$, then
$S^nR(z)\in\mathcal D\setminus\mathcal P$ for all $n\ge0$ by
Theorem~\ref{thmperiodic}.
In particular, $S^nR(z)$ is aperiodic as well.
We use the abbreviations $s_n=s(S^nR(z))$ and $t_n=t(S^nR(z))$.
Then we obtain inductively, for $n\ge1$,
$$
VS^nR(z)=\frac{VS^{n-1}R(z)A^{s_{n-1}}+t_{n-1}}\kappa=
\frac{VR(z)A^{s_0+s_1+\cdots+s_{n-1}}}{\kappa^n}+
\sum_{k=0}^{n-1}\frac{t_kA^{s_{k+1}+\cdots+s_{n-1}}}{\kappa^{n-k}}.
$$
If we look at the algebraic conjugates, then note that $|\kappa'|>1$, and we
obtain
$$
\left\|(VS^nR(z))'\right\|_\infty <
\frac{\left\|\big(VR(z)A^{s_0+s_1+\cdots+s_{n-1}}\big)'\right\|_\infty}
{|\kappa'|^n}+\delta,
$$
thus $\left\|(VS^nR(z))'\right\|_\infty\le\delta$ for some $n\ge0$ (as in 
Section~\ref{sectgolden}), and we can choose $\tilde z=S^nR(z)$.
\end{proof}

\noindent{\em Remarks.}
\begin{itemize}
\item
The last proof shows that, for every
$z\in(\mathbb Q(\lambda)\cap[0,1))^2\setminus\mathcal R$ with $\pi(z)=\infty$,
there are only finitely many possibilities for $VS^nR(z)$, hence
$(S^nR(z))_{n\ge0}$ is eventually periodic.
\item
For every $z\in\mathcal D$ with $\pi(z)=\infty$, we have
$$ \qquad\quad
V(z)=\Big(VS^n(z)\kappa^n-\sum_{k=0}^{n-1}t_kA^{s_{k+1}+\cdots+s_{n-1}}
\kappa^k\Big)A^{-s_0-\cdots-s_{n-1}}=
-\sum_{k=0}^\infty t_kA^{-\sum_{j=0}^ks(S^j(z))}\kappa^k,
$$
which is a $\kappa$-expansion ($\kappa<1$) of $V(z)$ with
(two-dimensional) ``digits'' $-t_kA^{-s_0-s_1-\cdots-s_k}$.
\item
As a consequence of Lemma~\ref{lemTm} and the definition of $U$, for every
aperiodic point $z\in[0,1)^2\setminus\mathcal R$ and every $c>0$, there exists
some $m\in\mathbb Z$ such that $\|T^m(z)-v\|_\infty<c$.
\item
In all our cases, we have $\varepsilon=\kappa\kappa'$.
\end{itemize}

\section{The case $\lambda=-1/\gamma=\frac{1-\sqrt5}2=-2\cos\frac{2\pi}5$}
Now we apply the method of Section~\ref{sectgeneral} for $\lambda=-1/\gamma$,
i.e., $\lambda'=\gamma$.
To this end, set
$$
\mathcal D=\{(x,y)\in[0,1)^2:\,x+y\ge3-\gamma\}=D_0\cup D_1
$$
with $D_0=\{(x,y)\in\mathcal D:x+\gamma y>2\}$,
$D_1=\{(x,y)\in\mathcal D:x+\gamma y\le2\}$.
Figure~\ref{figR2} shows that $\hat T$ is given by
$\hat T(z)=T^{\tau(\ell)}(z)$ if $z\in D_\ell$, $\ell\in\mathcal A=\{0,1\}$,
with $\tau(0)=1$ and $\tau(1)=4$.
The set which is left out by the iterates of $D_0$ and $D_1$ is 
$\mathcal R=\{(0,0)\}\cup D_A\cup D_B$, with
\begin{gather*}
D_A=\{z\in[0,1)^2:T^{k+1}(z)=T^k(z)A+(0,1)\mbox{ for all }k\ge0\}, \\
D_B=\{z\in[0,1)^2:T^{k+1}(z)=T^k(z)A+(0,2)\mbox{ for all }k\ge0\}.
\end{gather*}
As in Section~\ref{sectgolden}, we have $T^5(z)=z$ for all $z\in\mathcal R$.
If we set
$$
U(z)=\frac z{\gamma^2}+\Big(\frac1\gamma,\frac1\gamma\Big)=
(1,1)-\frac{(1,1)-z}{\gamma^2},
$$
$V(z)=(1,1)-z$, $\kappa=1/\gamma^2$, $\varepsilon=1$, and
$$
\sigma:\ 0\mapsto 010\qquad 1\mapsto 01110
$$
then Figure~\ref{figP2} shows that $\sigma$ satisfies the conditions in
Section~\ref{sectgeneral}, and 
$\mathcal P=D_\alpha\cup D_\beta$ with $D_\alpha=U(D_A)$, $D_\beta=U(D_B)$.
All points in $\mathcal P$ are periodic and $|\sigma^n(\ell)|\to\infty$ as
$n\to\infty$ for all $\ell\in\mathcal A$.
Therefore, all conditions of Proposition~\ref{propperiodic} and
Theorem~\ref{thmperiodic} are satisfied, and we obtain the following theorem.

\begin{figure}
\includegraphics{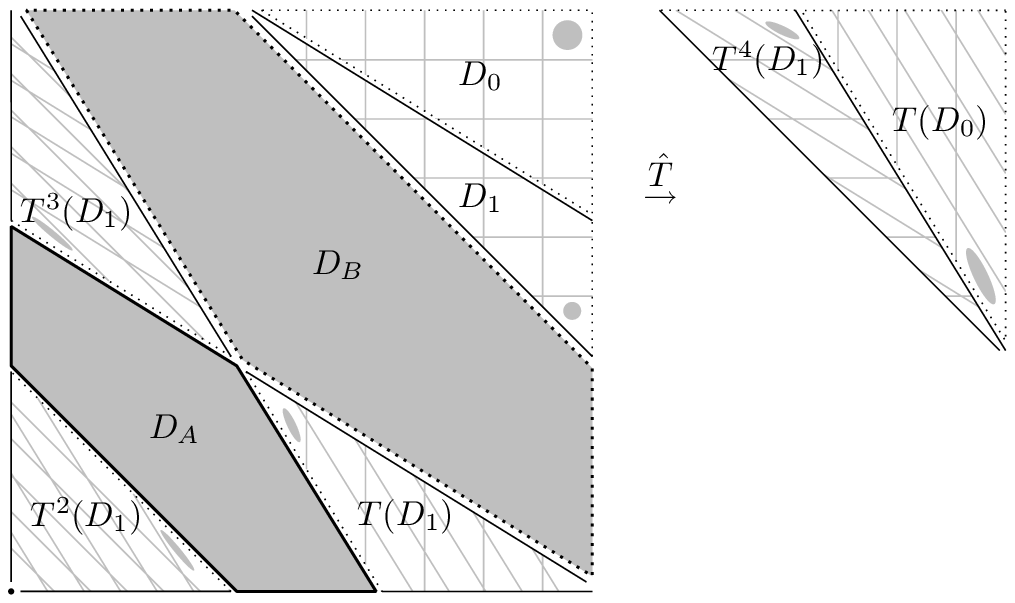}
\caption{\mbox{The map $\hat T$, $\hat T(D_0)=T(D_0)$, $\hat T(D_1)=T^4(D_1)$,
and the (gray) set $\mathcal R$, $\lambda=-1/\gamma$.}} \label{figR2}
\end{figure}

\begin{figure}
\includegraphics{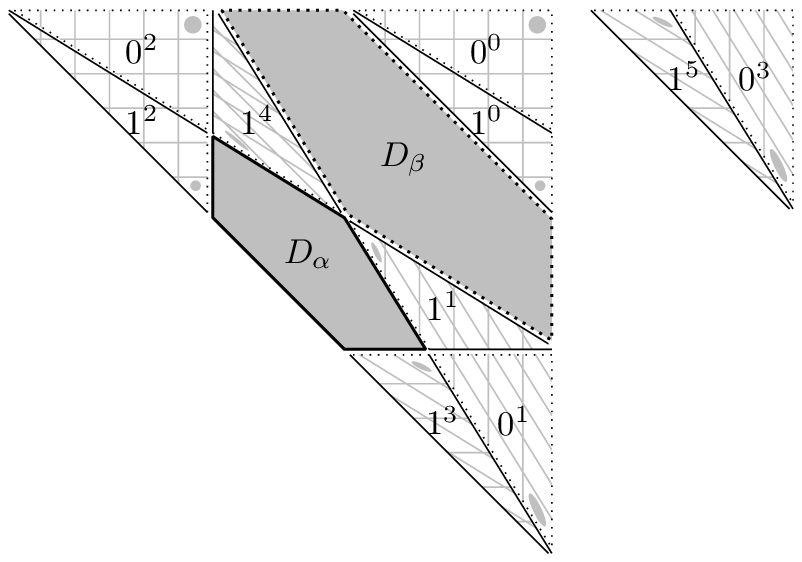}
\caption{\mbox{The trajectory of the scaled domains and $\mathcal P$,
$\lambda=-1/\gamma$. ($\ell^k$ stands for $\hat T^kU(D_\ell)$.)}} \label{figP2}
\end{figure}

\begin{theorem}\label{thperiods2}
If $\lambda=-1/\gamma$, then the period lengths $\pi(z)$ are \\
\centerline{ \begin{tabular}{cl} $1$ & if
$z\in\{(0,0),(\frac1{\gamma^2+1},\frac1{\gamma^2+1}),
(\frac2{\gamma^2+1},\frac2{\gamma^2+1})\}$ \\
$5$ & for the other points of the pentagons $D_A$ and $D_B$ \\
$2(5\cdot 4^n+1)/3$ & if
$S^nR(z)=(\frac{\gamma^2}{\gamma^2+1},\frac{\gamma^2}{\gamma^2+1})$ for some
$n\ge0$ \\
$10(5\cdot 4^n+1)/3$ & for the other points with $S^nR(z)\in D_\alpha$ for
some $n\ge0$ \\
$(5\cdot 4^n-2)/3$ & if $S^nR(z)=(\frac3{\gamma^2+1},\frac3{\gamma^2+1})$ for
some $n\ge0$ \\
$5(5\cdot 4^n-2)/3$ & for the other points with $S^nR(z)\in D_\beta$ for
some $n\ge0$ \\
$\infty$ & if $S^nR(z)\in\mathcal D\setminus\mathcal P$ for all $n\ge0$.
\end{tabular} }
\end{theorem}

\begin{proof}
We easily calculate
\begin{gather*}
\begin{pmatrix}|\sigma^n(0)|_0\\|\sigma^n(0)|_1\end{pmatrix}=
4^n\begin{pmatrix}1/3\\1/3\end{pmatrix}+\begin{pmatrix}2/3\\-1/3\end{pmatrix},
\qquad \begin{pmatrix}|\sigma^n(1)|_0\\|\sigma^n(1)|_1\end{pmatrix}=
4^n\begin{pmatrix}2/3\\2/3\end{pmatrix}+\begin{pmatrix}-2/3\\1/3\end{pmatrix},
\end{gather*}
hence $\tau(\sigma^n(0))=\frac53 4^n-\frac23$,
$\tau(\sigma^n(1))=\frac{10}3 4^n+\frac23$.
If $S^nR(z)\in D_\alpha$, then $\pi(z)=\tau(\sigma^n(1))$ and
$\pi(z)=\tau(\sigma^n(11111))$ respectively; if $S^nR(z)\in D_\beta$, then
$\pi(z)=\tau(\sigma^n(0))$ and $\pi(z)=5\tau(\sigma^n(0))$ respectively.
\end{proof}

For $z\in U(\mathcal D)$, we have $\hat s(z)=s(z)=0$ and $t(z)=(0,0)$.
For the other $z\in\mathcal D\setminus\mathcal P$, we choose $\hat s(z)$ as 
follows and obtain the following $s(z),t(z)$:
\begin{align*}
z\in\hat T^2U(D_0)\cup\hat T^2U(D_1):\ & \hat s(z)=-2,\,s(z)=-5,\
t(z)=V(\hat T^{-2}(z))-V(z)=(-1/\gamma^2,0) \\
z\in\hat TU(D_1):\ & \hat s(z)=-1,\,s(z)=-1,\
t(z)=V(\hat T^{-1}(z))-V(z)A^{-1}=(1/\gamma,0) \\
z\in\hat T^4U(D_1):\ & \hat s(z)=1,\,s(z)=1,\ 
t(z)=V(\hat T(z))-V(z)A=(0,1/\gamma) \\ 
z\in\hat TU(D_0)\cup\hat T^3U(D_1):\ & \hat s(z)=2,\,s(z)=5,\
t(z)=V(\hat T^2(z))-V(z)=(0,-1/\gamma^2)
\end{align*}
Observe the symmetry between positive and negative $\hat s(z)$ which is due to
the symmetry of $T(x,y)$ and $T^{-1}(y,x)$ and the symmetry of $\mathcal D$.
With
$$
\{(1/\gamma,0)A^h:\,h\in\mathbb Z\}=
\{(1/\gamma,0),\,(0,-1/\gamma),\,(-1/\gamma,1),\,(1,-1),\,(-1,1/\gamma)\},
$$
we obtain 
$\delta\le\max\{\|(t(z)A^h)'\|_\infty:\,z\in\mathcal D\setminus\mathcal P,\,
h\in\mathbb Z\}/\gamma=(1/\gamma^2)'/\gamma=\gamma$, as in 
Section~\ref{sectgolden}.
The following theorem shows that aperiodic points with $t(z)=(-1/\gamma^2,0)$
exist, hence $\delta=\gamma$.

\begin{theorem}
$\pi(z)$ is finite for all $z\in(\mathbb Z[\gamma]\cap[0,1))^2$, but
$\pi\big(1-1/(3\gamma),1-2/(3\gamma)\big)=\infty$.
\end{theorem}

\begin{proof}
By Proposition~\ref{propaper}, we have to show that all
$z\in\mathbb Z[\gamma]^2\cap\mathcal D$ with
$\|V(z)'\|_\infty\le\gamma$ are periodic.
Since $V(\mathcal D)=\{(x,y):x>0,y>0,x+y\le1/\gamma\}$, we have to consider 
$x,y\in\mathbb Z[\gamma]\cap(0,1/\gamma)$ with $|x'|,|y'|\le\gamma$.
No such $x,y$ exist, hence the conjecture is proved for $\lambda=-1/\gamma$.
Note that $\pi(z)$ is finite for all 
$z\in(\frac12\mathbb Z[\gamma]\cap[0,1))^2$ as well.
If $V(z)=\big(1/(3\gamma),2/(3\gamma)\big)$, then we have
\begin{align*}
VS(z) & = \gamma^2\big(V(z)A^5+(0,-1/\gamma^2)\big) =
\big(\gamma/3,1/(3\gamma^3)\big) \\
VS^2(z) & = \gamma^2\big(VS(z)A^{-5}+(-1/\gamma^2,0)\big) =
\big(2/(3\gamma),1/(3\gamma)\big) \\
VS^3(z) & = \gamma^2\big(VS^2(z)A^{-5}+(0,-1/\gamma^2)\big) =
\big(1/(3\gamma^3),\gamma/3)\big) \\
VS^4(z) & = \gamma^2\big(VS^3(z)A^5+(0,-1/\gamma^2)\big) =
\big(1/(3\gamma),2/(3\gamma)\big) = V(z),
\end{align*}
hence $S^n(z)\in\mathcal D\setminus\mathcal P$ for all $n\ge0$ and
$\pi(z)=\infty$ by Theorem~\ref{thperiods2}.
\end{proof}

\section{The case $\lambda=\sqrt2=-2\cos\frac{3\pi}4$}

Let $\lambda=\sqrt2$ ($\lambda'=-\sqrt2$) and set
\begin{gather*}
\mathcal D=\{(x,y)\in[0,1)^2:\sqrt2-2<x-\sqrt2y<0,\ 0<\sqrt2x-y<\sqrt2-2\}=
\bigcup\nolimits_{\ell\in\mathcal A=\{0,1,2,3\}}D_\ell, \\
D_0=\{(x,y)\in\mathcal D:x<\sqrt2-1\},\quad
D_1=\{(x,y)\in\mathcal D:x>\sqrt2-1,y\le\sqrt2-1\}, \\
D_2=\{(x,y)\in\mathcal D:x>\sqrt2-1,y>\sqrt2-1\},\quad
D_3=\{(x,y)\in\mathcal D:x=\sqrt2-1\}.
\end{gather*}
Figure~\ref{figR3} shows that $\hat T(z)=T^{\tau(\ell)}(z)$ if $z\in D_\ell$,
with $\tau(0)=5$, $\tau(1)=9$, $\tau(2)=3$, $\tau(3)=11$, and $\mathcal R=
\{(0,0)\}\cup\bigcup_{k=0}^3T^k(D_A)\cup\bigcup_{k=0}^5T^k(D_B)$ with
$D_A=\{(0,y):1-1/\sqrt2<y<1/\sqrt2\}$, $D_B=\{(0,1/\sqrt2)\}$.
If we set $U(z)=(\sqrt2-1)z$, $V(z)=z$, $\kappa=\sqrt2-1$, $\varepsilon=-1$, 
and
$$
\sigma:\ 0\mapsto 010\qquad 1\mapsto 000\qquad 2\mapsto 0\qquad
3\mapsto 030,
$$
then Figure~\ref{figP3} shows that $\sigma$ satisfies the conditions in
Section~\ref{sectgeneral}, and
$$
\mathcal P=\{(x,y)\in\mathcal D:x,y\ge\sqrt2-1\}=
D_\alpha\cup D_\beta\cup\hat T(D_\beta)\cup D_\zeta
$$
with $D_\alpha=D_2$, $D_\beta=\{(x,\sqrt2-1):\sqrt2-1<x<2-\sqrt2\}$ and
$D_\zeta=\{(\sqrt2-1,\sqrt2-1)\}$.
All points in $\mathcal P$ are periodic and $|\sigma^n(\ell)|\to\infty$ as
$n\to\infty$ for all $\ell\in\mathcal A$.
Therefore, all conditions of Proposition~\ref{propperiodic} and
Theorem~\ref{thmperiodic} are satisfied, and we obtain the following theorem.

\begin{figure}
\includegraphics{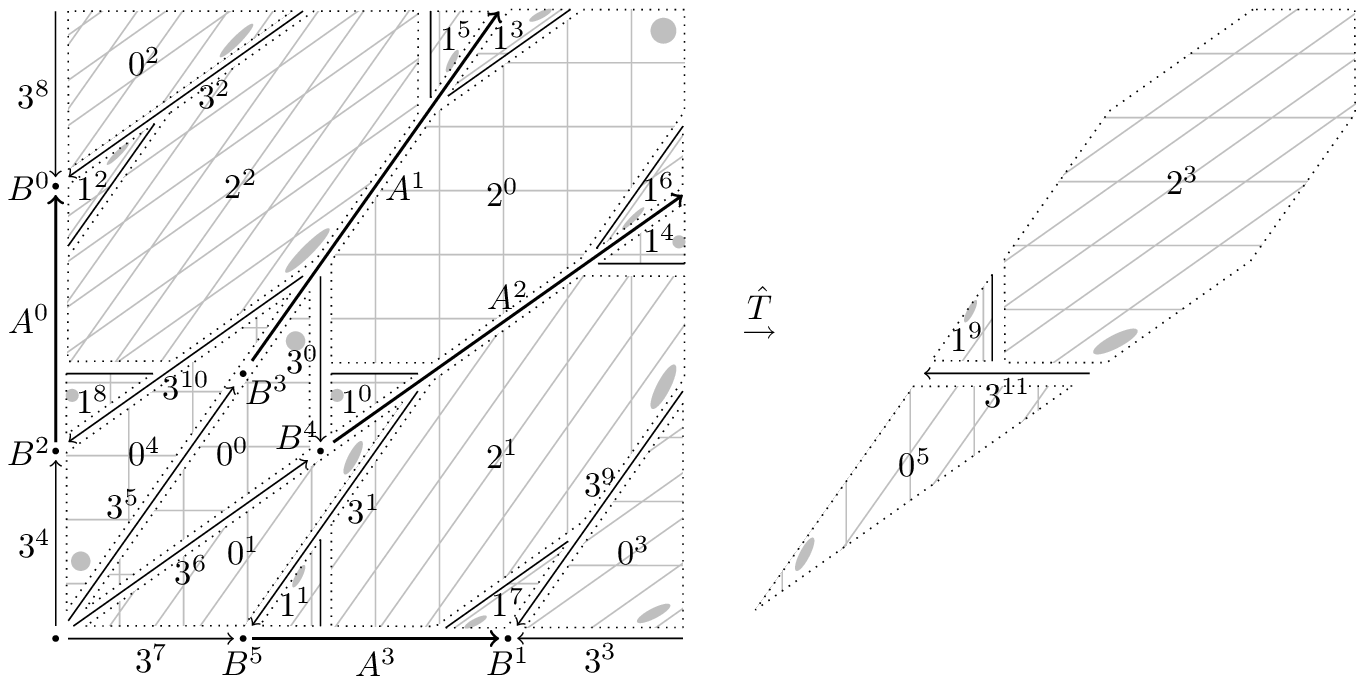}
\caption{The map $\hat T$ and the set $\mathcal R$,
$\lambda=\sqrt2$. ($\ell^k$ stands for $T^k(D_\ell)$.)} \label{figR3}
\end{figure}

\begin{figure}
\includegraphics{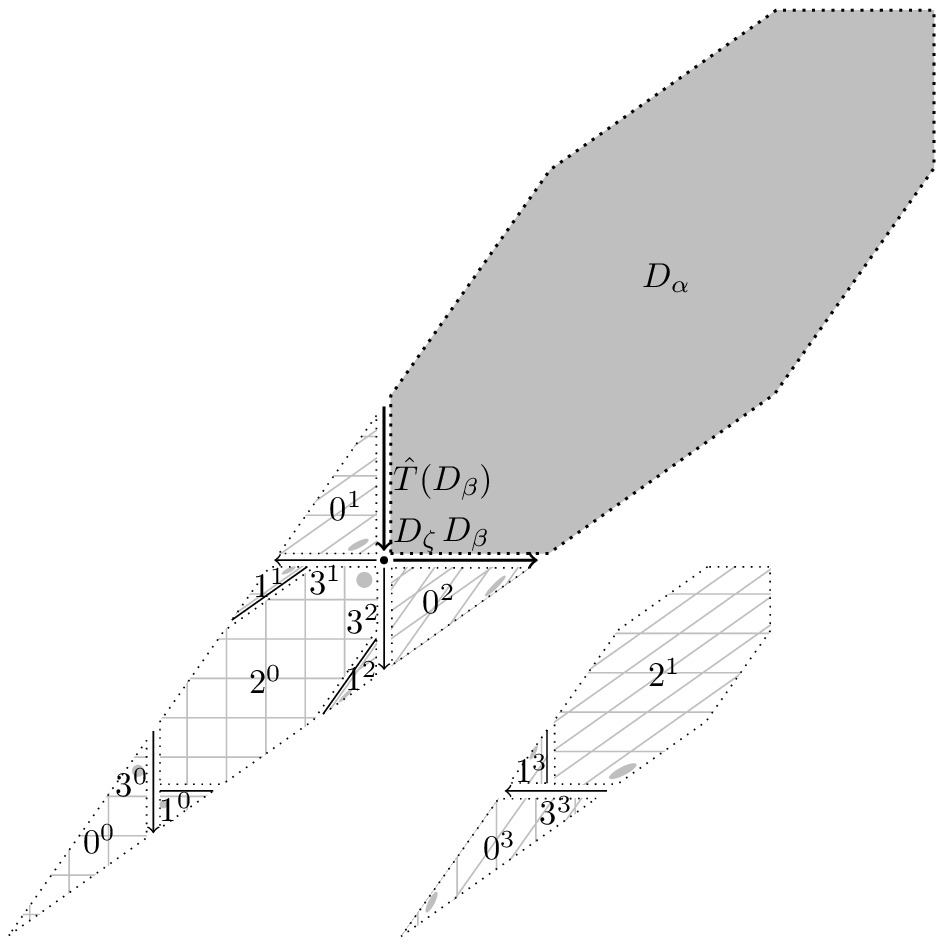}
\caption{The trajectory of the scaled domains and $\mathcal P$,
$\lambda=\sqrt2$. ($\ell^k$ stands for $\hat T^{-k}U(D_\ell)$.)} \label{figP3}
\end{figure}

\begin{theorem}\label{thperiods3}
If $\lambda=\sqrt2$, then the minimal period length $\pi(z)$ is \\
\centerline{\begin{tabular}{cl}
$1$ & if $z=(0,0)$ \\
$4 $& if $z=T^m(0,1/2)$, $0\le m\le 3$ \\
$8$ & for the other points of $T^m(D_A)$, $0\le m\le 3$ \\
$6$ & if $z=T^m(0,1/\sqrt2)$, $0\le m\le 5$ \\
$2\cdot3^n+(-1)^n$ & if $S^nR(z)=(1/\sqrt2,1/\sqrt2)$, $n\ge0$ \\
$8(2\cdot3^n+(-1)^n)$ & for the other points with $S^nR(z)\in D_\alpha$ \\
$4(3^{n+1}+1+(-1)^n)$ & if $S^nR(z)\in\{(1/2,\sqrt2-1),(\sqrt2-1,1/2)\}$, 
$n\ge0$ \\
$8(3^{n+1}+1+(-1)^n)$ & for the other points with
$S^nR(z)\in D_\beta\cup\hat T(D_\beta)$ \\
$2\cdot 3^{n+1}+4+(-1)^n$ & if $S^nR(z)=(\sqrt2-1,\sqrt2-1)$, $n\ge0$ \\
$\infty$ & if $S^nR(z)\in\mathcal D\setminus\mathcal P$ for all $n\ge0$.
\end{tabular}}
\end{theorem}

\begin{proof}
We easily calculate
\begin{gather*}
\begin{pmatrix}|\sigma^n(0)|_0\\|\sigma^n(0)|_1\end{pmatrix}=
3^n\begin{pmatrix}3/4\\1/4\end{pmatrix}+
(-1)^n\begin{pmatrix}1/4\\-1/4\end{pmatrix},\qquad
\begin{pmatrix}|\sigma^n(1)|_0\\|\sigma^n(1)|_1\end{pmatrix}=
3^n\begin{pmatrix}3/4\\1/4\end{pmatrix}+
(-1)^n\begin{pmatrix}-3/4\\3/4\end{pmatrix}
\end{gather*}
and obtain $\tau(\sigma^n(0))=2\cdot3^{n+1}-(-1)^n$,
$\tau(\sigma^n(3))=\tau(\sigma^{n-1}(030))=2\cdot3^{n+1}+4+(-1)^n$.
If $S^nR(z)\in D_\alpha$ and $n\ge1$, then
$\pi(z)=\tau(\sigma^n(2))=\tau(\sigma^{n-1}(0))$ and
$\pi(z)=8\tau(\sigma^{n-1}(0))$ respectively; if $S^nR(z)\in D_\beta$, then
$\pi(z)=\tau(\sigma^n(13))=\tau(\sigma^{n-1}(000030))$ and
$\pi(z)=2\tau(\sigma^{n-1}(000030))$ respectively; if
$S^nR(z)=(\sqrt2-1,\sqrt2-1)$, then $\pi(z)=\tau(\sigma^n(3))$.
The given $\pi(z)$ hold for $n=0$ as well.
\end{proof}

For $z\in\mathcal D\setminus(U(\mathcal D)\cup\mathcal P)$, we choose 
$\hat s(z)$ as follows and obtain the following $s(z),t(z)$:
\begin{align*}
z\in\hat T^{-2}U(D_0\cup D_1\cup D_3):\ & \hat s(z)=-1,\,s(z)=-5,\
t(z)=\hat T^{-1}(z)-zA^{-5}=(\sqrt2-1,2-\sqrt2) \\
z\in\hat T^{-1}U(D_0\cup D_1\cup D_3):\ & \hat s(z)=1,\,s(z)=5,\
t(z)=\hat T(z)-zA^5=(2-\sqrt2,\sqrt2-1)
\end{align*}
This gives $\delta=(2+\sqrt2)/\sqrt2=\sqrt2+1$ since
$$
\{t(z)A^h:z\in\mathcal D\setminus\mathcal P,h\in\mathbb Z\}=\pm\{(0,0),\
(2-\sqrt2,\sqrt2-1),(\sqrt2-1,0),(0,1-\sqrt2),(1-\sqrt2,\sqrt2-2)\}.
$$

\begin{theorem}
$\pi(z)$ is finite for all $z\in(\mathbb Z[\sqrt2]\cap[0,1))^2$, but
$(T^k(\frac{3-\sqrt2}4,\frac{2\sqrt2-1}4))_{k\in\mathbb Z}$ is aperiodic.
\end{theorem}

\begin{proof}
We have to consider $z\in\mathbb Z[\sqrt2]^2\cap\mathcal D$ with
$\|z'\|_\infty\le\delta=\sqrt2+1$.
The only such point is $(\sqrt2-1,\sqrt2-1)=D_\zeta$, hence
Conjecture~\ref{cj} holds for $\lambda=\sqrt2$.
It can be shown that all points in $(\frac12\mathbb Z[\sqrt2]\cap[0,1))^2$
and $(\frac13\mathbb Z[\sqrt2]\cap[0,1))^2$ are periodic as well.
For $z=(\frac{3-\sqrt2}4,\frac{2\sqrt2-1}4)$, we have
\begin{gather*}
S(z)=\big(zA^5+(2-\sqrt2,\sqrt2-1)\big)/\kappa=
(\sqrt2+1)\Big(\frac{9-6\sqrt2}4,\sqrt2-\frac54\Big)=
\Big(\frac{3\sqrt2-3}4,\frac{3-\sqrt2}4\Big), \\
S^2(z)=\big(S(z)A^5+(2-\sqrt2,\sqrt2-1)\big)/\kappa=
(\sqrt2+1)\Big(\frac{5-3\sqrt2}4,\sqrt2-\frac54\Big)=
\Big(\frac{2\sqrt2-1}4,\frac{3-\sqrt2}4\Big),
\end{gather*}
$S^3(z)=\big(S^2(z)A^{-5}+(\sqrt2-1,2-\sqrt2)\big)/\kappa=
(\frac{3-\sqrt2}4,\frac{3\sqrt2-3}4)$ and
$S^4(z)=(\frac{3-\sqrt2}4,\frac{2\sqrt2-1}4)=z$.
\end{proof}

\begin{figure}
\begin{minipage}{7.3cm}
\centerline{\includegraphics[scale=.8]{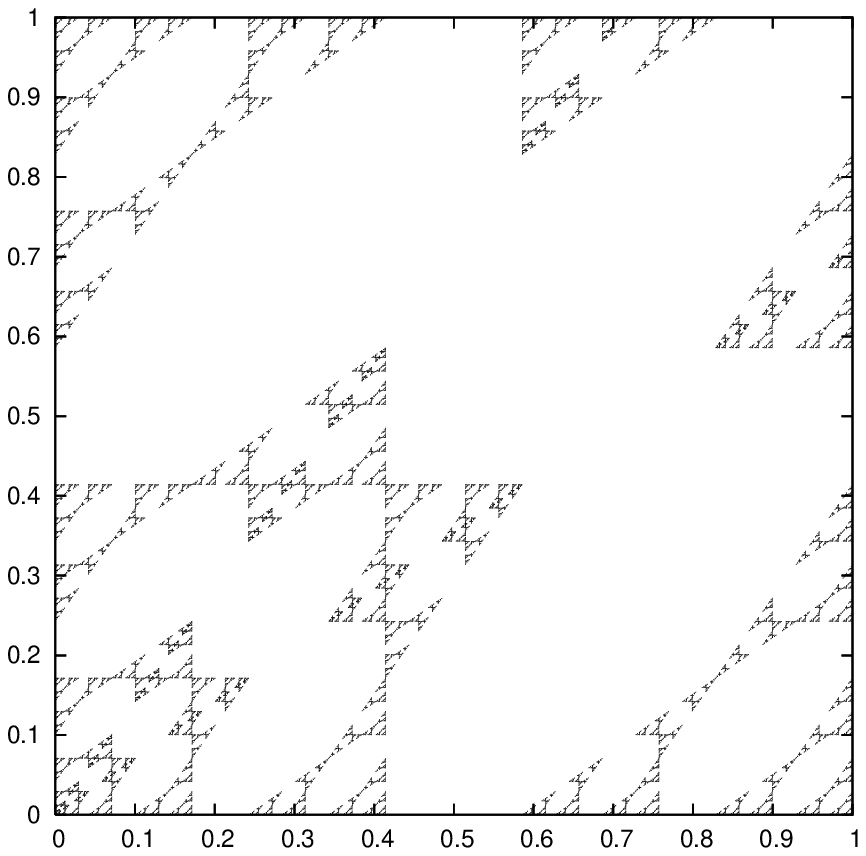}}
\caption{\mbox{Aperiodic points, $\lambda=\sqrt2$.}}\label{fig3a}
\end{minipage}
\begin{minipage}{7.3cm}
\centerline{\includegraphics[scale=.8]{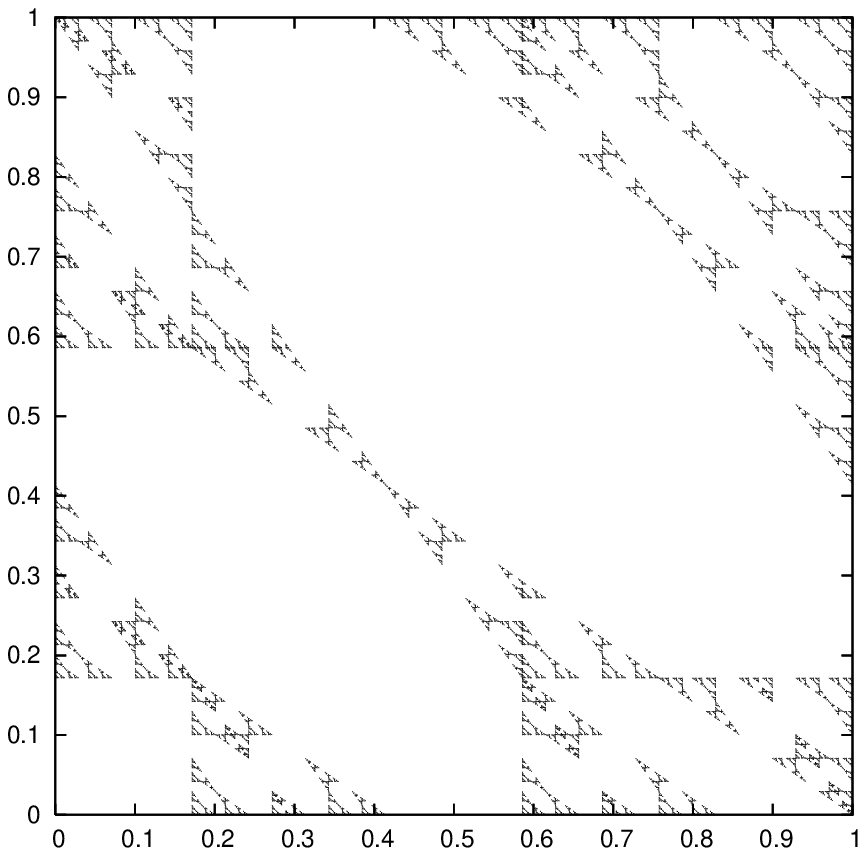}}
\caption{\mbox{Aperiodic points, $\lambda=-\sqrt2$.}}\label{fig4a}
\end{minipage}
\end{figure}

\section{The case $\lambda=-\sqrt2=-2\cos\frac\pi4$}

Let $\lambda=-\sqrt2$ ($\lambda'=\sqrt2$) and set
$$
\mathcal D=\{(x,y)\in[0,1)^2:\sqrt2x+y>2\mbox{ or }x+\sqrt2y>2\}=
\bigcup\nolimits_{\ell\in\mathcal A=\{0,1,2\}}D_\ell,
$$
with $D_0=\{(x,y)\in\mathcal D:x+\sqrt2y>2\}$ and
$D_1=\{(x,y)\in\mathcal D:x+\sqrt2y<2\}$.
Figure~\ref{figR4} shows that $\hat T(z)=T^{\tau(\ell)}(z)$ if $z\in D_\ell$,
with $\tau(0)=1$, $\tau(1)=21$, $\tau(2)=31$, and
\begin{gather*}
\mathcal R=\{(0,0)\}\cup D_A\cup D_B\cup\bigcup\nolimits_{k=0}^3 T^k(D_\Gamma)
\cup\bigcup\nolimits_{k=0}^9T^k(D_\Delta), \\
D_A=\{(x,y):0\le x,y\le3-2\sqrt2\}\setminus\{(0,0),(3-2\sqrt2,3-2\sqrt2)\}, \\
D_B=\{z\in[0,1)^2:T^{k+1}(z)=T^k(z)A+(0,1)\mbox{ for all }k\in\mathbb Z\}, \\
D_\Gamma=\{z\in[0,1)^2:T^{k+1}(z)=T^k(z)A+(0,2)\mbox{ for all }k\in\mathbb Z\},
\end{gather*}
$D_\Delta=\{(1/\sqrt2,0)\}$.
Set $\kappa=\sqrt2-1$, $V(z)=((1,1)-z)/\kappa=(\sqrt2+1)((1,1)-z)$, i.e.,
$$
U(z)=(1,1)-(\sqrt2-1)\big((1,1)-z\big)=(\sqrt2-1)z+(2-\sqrt2,2-\sqrt2).
$$
Then Figure~\ref{figP4} shows that the conditions in Section~\ref{sectgeneral}
are satisfied by
$$
\sigma:\ 0\mapsto 010 \qquad 1\mapsto 000\qquad 2\mapsto 020
$$
with $\varepsilon=-1$ and $\mathcal P=D_\alpha\cup
\bigcup_{k=0}^5\hat T^k(D_\beta)\cup\bigcup_{k=0}^2\hat T^k(D_\zeta)$ with 
$$
D_\alpha=\{z\in[0,1)^2:T^{k+1}(z)=T^k(z)A+(0,3)\mbox{ for all }
k\in\mathbb Z\},
$$ 
$D_\beta=\{(x,2-\sqrt2x):5-3\sqrt2<x<2\sqrt2-2\}$ and 
$D_\zeta=\{(8-5\sqrt2,8-5\sqrt2)\}$.
All points in $\mathcal P$ are periodic and $|\sigma^n(\ell)|\to\infty$ as
$n\to\infty$ for all $\ell\in\mathcal A$.
Therefore, all conditions of Proposition~\ref{propperiodic} and
Theorem~\ref{thmperiodic} are satisfied, and we obtain the following theorem.

\begin{figure}
\includegraphics[scale=.99]{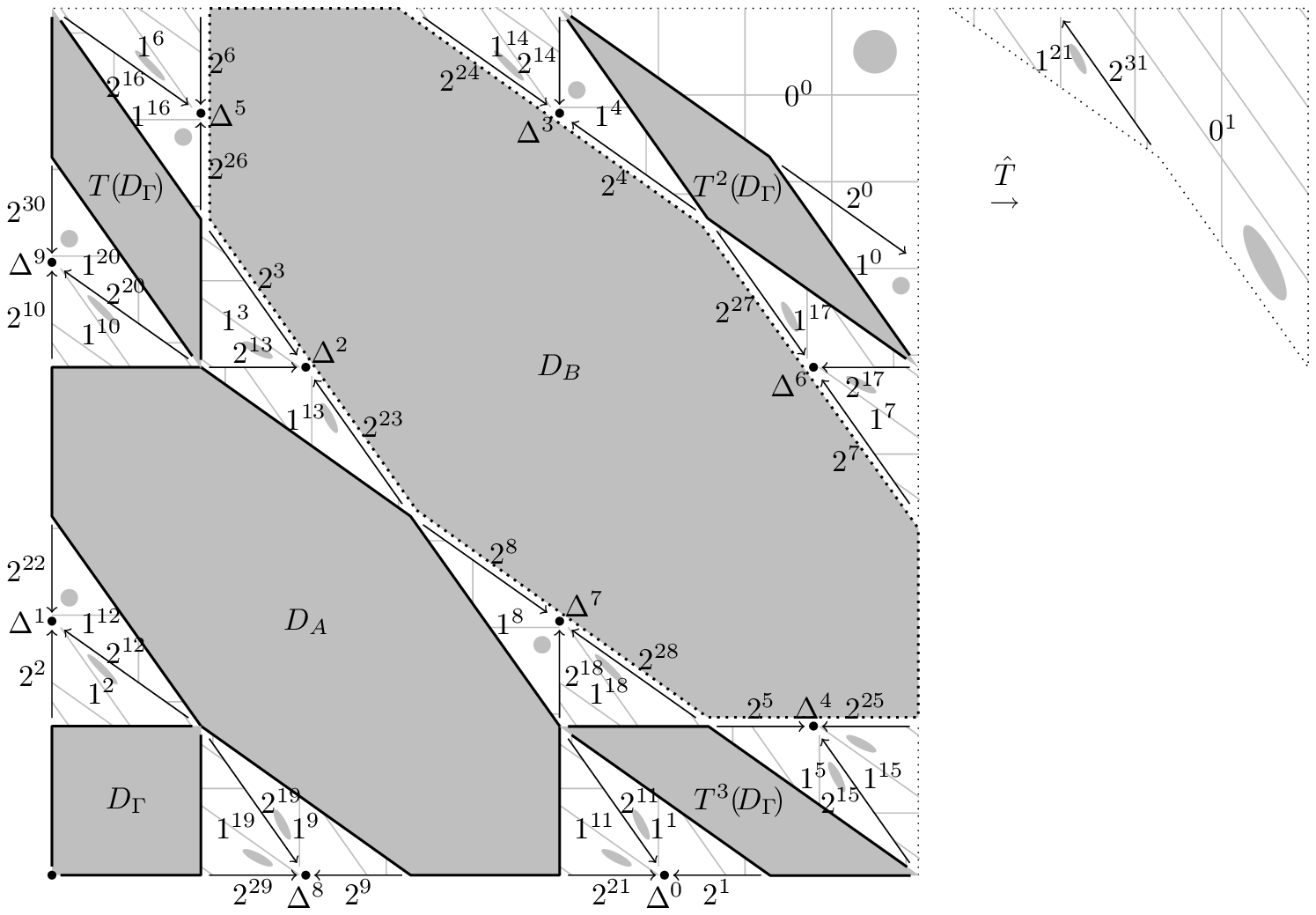}
\caption{The map $\hat T$ and the set $\mathcal R$, $\lambda=-\sqrt2$.
($\ell^k$ stands for $T^k(D_\ell)$.)} \label{figR4}
\end{figure}

\begin{figure}
\includegraphics{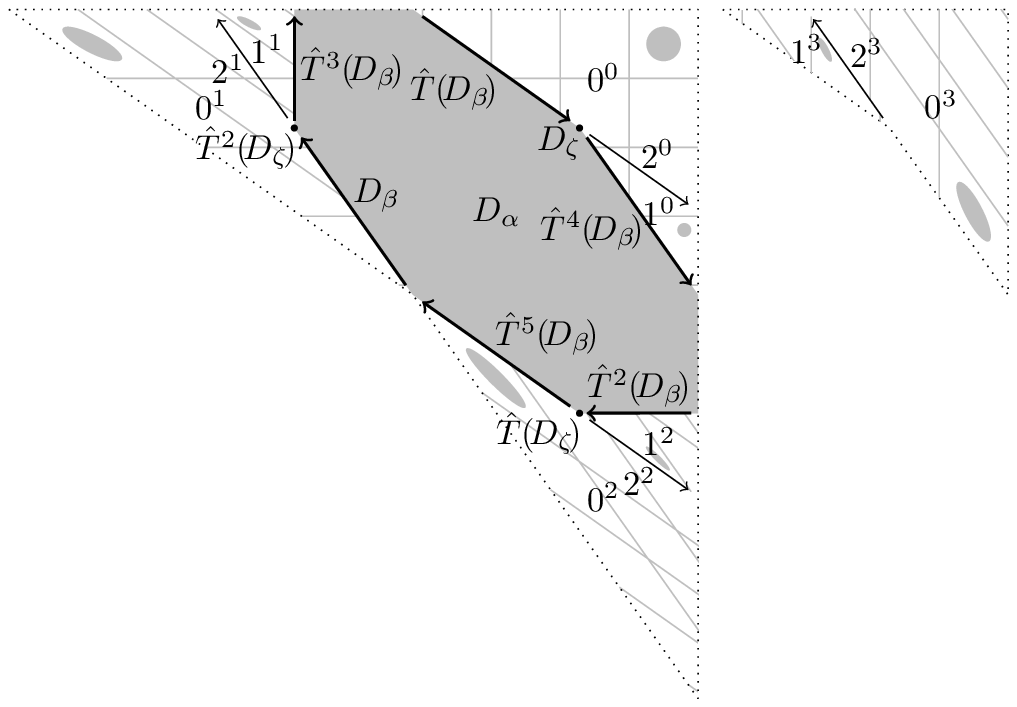}
\caption{\mbox{The trajectory of the scaled domains and $\mathcal P$,
$\lambda=-\sqrt2$. ($\ell^k$ stands for $\hat T^{-k}U(D_\ell)$.)}} \label{figP4}
\end{figure}

\begin{theorem}\label{thperiods4}
If $\lambda=-\sqrt2$, then the minimal period length $\pi(z)$ is \\
\centerline{\begin{tabular}{cl}
$1$ & if $z\in\{(0,0),(1/\sqrt2,1/\sqrt2),(2-\sqrt2,2-\sqrt2)\}$ \\
$4$ & if $z=T^m(3/2-\sqrt2,3/2-\sqrt2)$ for some $m\in\{0,1,2,3\}$  \\
$10$ & if $z=T^m(1/\sqrt2,0)$ for some $m\in\{0,1,\ldots,9\}$ \\
$8$ & for the other points in $\mathcal R$ \\
$2\cdot3^{n+1}-5(-1)^n$&if $S^nR(z)=(3-3/\sqrt2,3-3/\sqrt2)$ for some $n\ge0$\\
$8(2\cdot3^{n+1}-5(-1)^n)$&for the other points with $S^nR(z)\in D_\alpha$\\
$4(3^{n+2}+5-5(-1)^n)$ & if $S^nR(z)=\hat T^m\big((9-5\sqrt2)/2,5-3\sqrt2\big)$
for some $m\in\{0,\ldots,5\}$, $n\ge0$ \\
$8(3^{n+2}+5-5(-1)^n)$ & for the other points with $S^nR(z)\in\hat T^m(D_\beta)$
\\ $2\cdot 3^{n+2}+20-5(-1)^n$ & if $S^nR(z)=\hat T^m(8-5\sqrt2,8-5\sqrt2)$ for
some $m\in\{0,1,2\}$, $n\ge0$ \\
$\infty$ & if $S^nR(z)\in\mathcal D\setminus\mathcal P$ for all $n\ge0$.
\end{tabular}}
\end{theorem}

\begin{proof}
As for $\lambda=\sqrt2$, we have
\begin{gather*}
\begin{pmatrix}|\sigma^n(0)|_0\\|\sigma^n(0)|_1\end{pmatrix}=
3^n\begin{pmatrix}3/4\\1/4\end{pmatrix}+
(-1)^n\begin{pmatrix}1/4\\-1/4\end{pmatrix},\qquad
\begin{pmatrix}|\sigma^n(1)|_0\\|\sigma^n(1)|_1\end{pmatrix}=
3^n\begin{pmatrix}3/4\\1/4\end{pmatrix}+
(-1)^n\begin{pmatrix}-3/4\\3/4\end{pmatrix},
\end{gather*}
hence $\tau(\sigma^n(0))=2\cdot3^{n+1}-5(-1)^n$ and
$\tau(\sigma^n(2))=\tau(\sigma^{n-1}(020))=2\cdot3^{n+1}+20+5(-1)^n$.
For $S^nR(z)\in D_\alpha$, we have $\pi(z)=\tau(\sigma^n(0))$ and
$\pi(z)=8\tau(\sigma^n(0))$ respectively;
if $S^nR(z)\in T^m(D_\beta)$, then $\pi(z)=\tau(\sigma^n(002000))$
and $\pi(z)=2\tau(\sigma^n(002000))$ respectively;
if $S^nR(z)=\hat T^m(D_\zeta)$, then $\pi(z)=\tau(\sigma^n(020))$.
\end{proof}

For $z\in\mathcal D\setminus(U(\mathcal D)\cup\mathcal P)$, we choose 
$\hat s(z)$ as follows and obtain the following $s(z),t(z)$:
\begin{align*}
z\in\hat T^{-2}U(D_0\cup D_1\cup D_2):\ & \hat s(z)=-1,\,s(z)=-1,\
t(z)=V(\hat T^{-1}(z))-V(z)A^{-1}=(1,0) \\
z\in\hat T^{-1}U(D_0\cup D_1\cup D_2):\ & \hat s(z)=1,\,s(z)=1,\
t(z)=V(\hat T(z))-V(z)A=(0,1)
\end{align*}
This gives $\delta=\sqrt2/\sqrt2=1$ since
$$
\{t(z)A^h:z\in\mathcal D\setminus\mathcal P,h\in\mathbb Z\}=
\pm\{(0,0),\ (1,0),(0,1),(1,-\sqrt2),(-\sqrt2,1)\}.
$$

\begin{theorem}
$\pi(z)$ is finite for all $z\in(\mathbb Z[\sqrt2]\cap[0,1))^2$, but
$(T^k(\frac34,\frac{5-\sqrt2}4))_{k\in\mathbb Z}$ is aperiodic.
\end{theorem}

\begin{proof}
Since $V(\mathcal D)=\{(x,y):x>0,y>0,x+\sqrt2y<1\mbox{ or }\sqrt2x+y<1\}$,
there exists no $z\in\mathbb Z[\sqrt2]^2\cap\mathcal D$ with
$\|(V(z))'\|_\infty\le1$.
Therefore Conjecture~\ref{cj} holds for $\lambda=-\sqrt2$.
It can be shown that all points in $(\frac12\mathbb Z[\sqrt2]\cap[0,1))^2$ and
$(\frac13\mathbb Z[\sqrt2]\cap[0,1))^2$ are periodic as well.
For $z=(\frac34,\frac{5-\sqrt2}4)$, we have $V(z)=(\frac{\sqrt2+1}4,\frac14)$,
$$
VS(z)=(\sqrt2+1)(V(z)A+(0,1))=(\sqrt2+1)\Big(\frac14,\frac{3-2\sqrt2}4\Big)=
\Big(\frac{\sqrt2+1}4,\frac{\sqrt2-1}4\Big),
$$
$VS^2(z)=(\frac14,\frac{\sqrt2+1}4)$,
$VS^3(z)=(\frac{\sqrt2-1}4,\frac{\sqrt2+1}4)$ and
$VS^4(z)=(\frac{\sqrt2+1}4,\frac14)=V(z)$.
\end{proof}

\section{The case $\lambda=1/\gamma=-2\cos\frac{3\pi}5$}

Let $\lambda=1/\gamma$ ($\lambda'=-\gamma$) and set
$$
\mathcal D=\{(x,y)\in[0,1)^2:\gamma x-1<y<x/\gamma\}=
\bigcup\nolimits_{\ell\in\mathcal A=\{0,1,2,3\}}D_\ell,
$$
with $D_0,D_1,D_2,D_3$ satisfying the (in)equalities
$$
\begin{array}{c|c|c|c}
D_0 & D_1 & D_2 & D_3 \\ \hline
y>x-1/\gamma^2 & 0<y<x-1/\gamma^2 & y=x-1/\gamma^2 & y=0,\,1/\gamma^2<x<1/\gamma
\end{array} 
$$
Figure~\ref{figR5} shows that $\hat T(z)=T^{\tau(\ell)}(z)$ if $z\in D_\ell$,
with $\tau(0)=6$, $\tau(1)=4$, $\tau(2)=7$, $\tau(3)=5$, and
$\mathcal R=\{(0,0)\}$.
If we set $U(z)=z/\gamma^2$, $V(z)=z$, $\kappa=1/\gamma^2$, $\varepsilon=1$, 
and
$$
\sigma:\ 0\mapsto010\qquad 1\mapsto01110\qquad 2\mapsto012\qquad 3\mapsto01112
$$
then Figure~\ref{figP5} shows that $\sigma$ satisfies the conditions in
Section~\ref{sectgeneral}, and
$$
\mathcal P=D_\alpha\cup D_\beta\cup\bigcup\nolimits_{k=0}^3\hat T^k(D_\zeta)
\cup D_\vartheta\cup\bigcup\nolimits_{k=0}^1\hat T^k(D_\eta)\cup D_\mu
$$
with
$D_\alpha=\{z\in\mathcal D:\hat T^k(z)\in D_0\mbox{ for all }k\in\mathbb Z\}$,
$D_\beta=\{z\in\mathcal D:\hat T^k(z)\in D_1\mbox{ for all }k\in\mathbb Z\}$,
$D_\zeta=\{(x,0):1/\gamma^3<x<1/\gamma^2\}$, $D_\eta=D_3$,
$D_\vartheta=\{(1/\gamma^3,0)\}$ and $D_\mu=\{(1/\gamma^2,0)\}$.
All points in $\mathcal P$ are periodic and $|\sigma^n(\ell)|\to\infty$ as
$n\to\infty$ for all $\ell\in\mathcal A$.
Therefore, all conditions of Proposition~\ref{propperiodic} and
Theorem~\ref{thmperiodic} are satisfied, and we obtain the following theorem.

\begin{figure}
\includegraphics{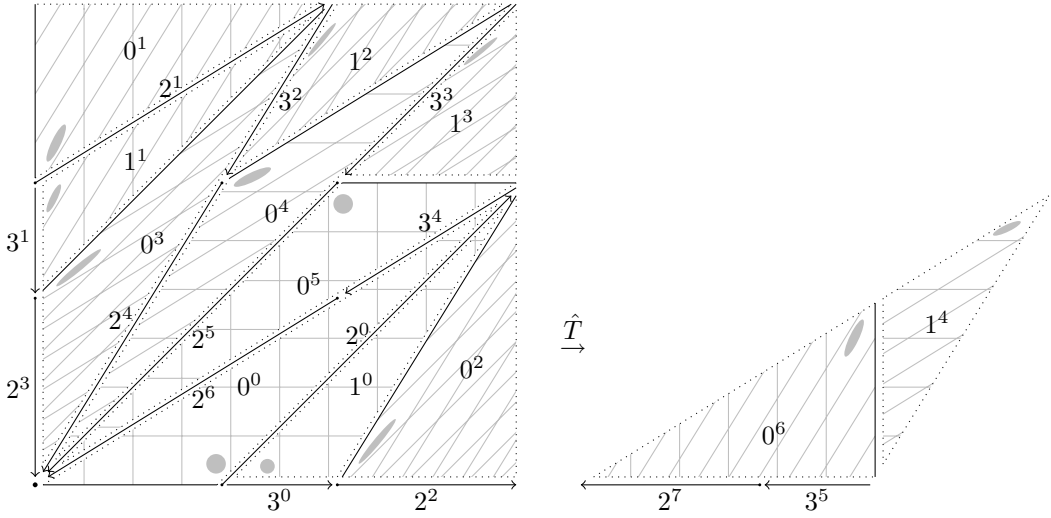}
\caption{The map $\hat T$, $\lambda=1/\gamma$. ($\ell^k$ stands for
$T^k(D_\ell)$.)} \label{figR5}
\end{figure}

\begin{figure}
\includegraphics{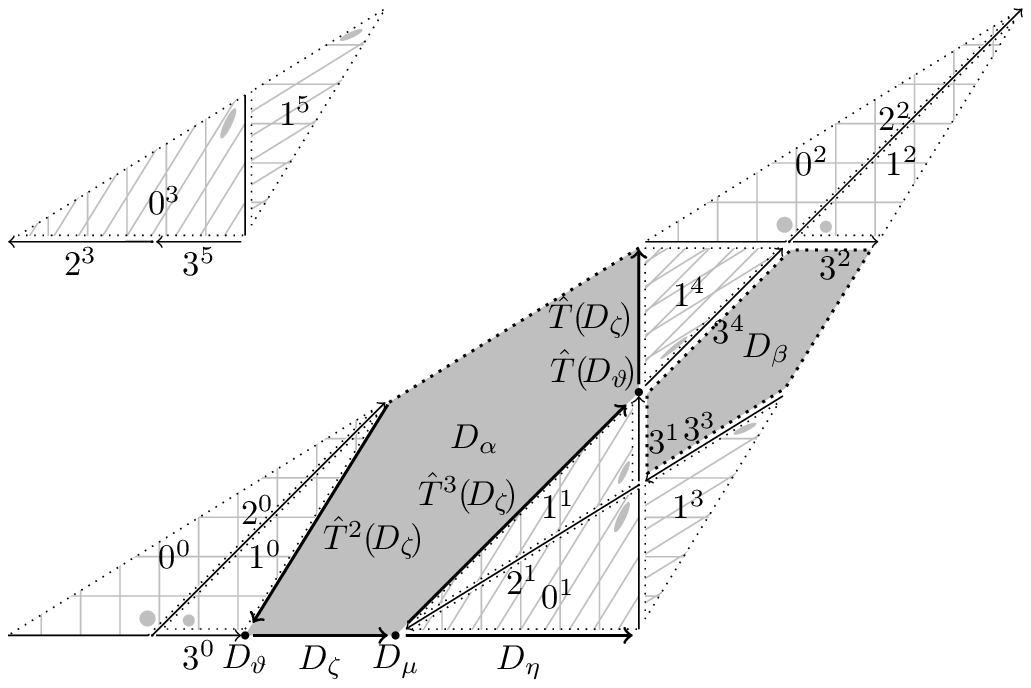}
\caption{The trajectory of the scaled domains and $\mathcal P$,
$\lambda=1/\gamma$. ($\ell^k$ stands for $\hat T^kU(D_\ell)$.)} \label{figP5}
\end{figure}

\begin{theorem}\label{thperiods5}
If $\lambda=1/\gamma$, then the minimal period length $\pi(z)$ is \\
\centerline{ \begin{tabular}{cl}
$1$ & if $z=(0,0)$ \\
$2(5\cdot 4^n+4)/3$ & if
$S^nR(z)=\big(\frac\gamma{\gamma^2+1},\frac{1/\gamma}{\gamma^2+1}\big)$ for
some $n\ge0$ \\
$10(5\cdot 4^n+4)/3$ & for the other points with $S^nR(z)\in D_\alpha$ \\
$4(5\cdot 4^n-2)/3$ & if
$S^nR(z)=\big(\frac{\gamma^2}{\gamma^2+1},\frac1{\gamma^2+1}\big)$ for some
$n\ge0$ \\
$20(5\cdot 4^n-2)/3$ & for the other points with $S^nR(z)\in D_\beta$ \\
$5(4^{n+1}-1)/3$ & if $S^nR(z)=(0,1/2)$ for some $n\ge0$ \\
$10(4^{n+1}-1)/3$ & for the other points with $S^nR(z)\in D_\vartheta$ \\
$5(2\cdot 4^{n+1}+7)/3$ & if $S^nR(z)=\hat T^m(1/(2\gamma),0)$ for some
$m\in\{0,1,2,3\}$ and $n\ge0$ \\
$10(2\cdot 4^{n+1}+7)/3$ & for the other points with
$S^nR(z)\in \hat T^m(D_\zeta)$ \\
$(10\cdot 4^n+11)/3$ & if $S^nR(z)=(1/\gamma^2,0)$ for some $n\ge0$ \\
$(5\cdot 4^{n+1}+19)/3$ & if $S^nR(z)=\hat T^m(1/\gamma^3,0)$ for some
$m\in\{0,1\}$ and $n\ge0$ \\
$\infty$ & if $S^nR(z)\in\mathcal D\setminus\mathcal P$ for all $n\ge0$.
\end{tabular} }
\end{theorem}

\begin{proof}
As for $\lambda=-1/\gamma$, we have
\begin{gather*}
\begin{pmatrix}|\sigma^n(0)|_0\\|\sigma^n(0)|_1\end{pmatrix}=
4^n\begin{pmatrix}1/3\\1/3\end{pmatrix}+\begin{pmatrix}2/3\\-1/3\end{pmatrix},
\qquad \begin{pmatrix}|\sigma^n(1)|_0\\|\sigma^n(1)|_1\end{pmatrix}=
4^n\begin{pmatrix}2/3\\2/3\end{pmatrix}+\begin{pmatrix}-2/3\\1/3\end{pmatrix},
\end{gather*}
hence $\tau(\sigma^n(0))=\frac{10}3 4^n+\frac83$,
$\tau(\sigma^n(1))=\frac{20}3 4^n-\frac83$,
$\tau(\sigma^n(2))=\frac{10}3 4^n+\frac{11}3$,
$\tau(\sigma^n(3))=\frac{20}3 4^n-\frac53$.
For $S^nR(z)\in D_\alpha$, we have $\pi(z)=\tau(\sigma^n(0))$ and
$\pi(z)=5\tau(\sigma^n(0))$ respectively; if $S^nR(z)\in D_\beta$, then
$\pi(z)=\tau(\sigma^n(1))$ and $5\tau(\sigma^n(1))$ respectively; if
$S^nR(z)\in D_\eta$, then $\pi(z)=\tau(\sigma^n(3))$ and $2\tau(\sigma^n(3))$ 
respectively; if $S^nR(z)\in D_\zeta$, then $\pi(z)=\tau(\sigma^n(0002))$ and 
$2\tau(\sigma^n(0002))$ respectively; if $S^nR(z)=\hat T^m(1/\gamma^3,0)$, then 
$\pi(z)=\tau(\sigma^n(02))$; if $S^nR(z)=(1/\gamma^2,0)$, then 
$\pi(z)=\tau(\sigma^n(2))$.
\end{proof}

Note that $\hat T^mU(D_3)$ plays no role in the calculation of $\delta$ since 
$U(D_3)\subset U(\mathcal P)$ and thus $\pi(z)<\infty$ for all 
$z\in\hat T^mU(D_3)$.
For the other $z\in\mathcal D\setminus(\mathcal P\cup U(\mathcal D))$,
we choose $\hat s(z)$ as follows:
\begin{align*}
z\in\hat T^2U(D_0\cup D_1\cup D_2):\ & \hat s(z)=-2,\,s(z)=-10,\
t(z)=\hat T^{-2}(z)-z=(-1/\gamma,-1/\gamma^2) \\
z\in\hat TU(D_1\cup D_2):\ & \hat s(z)=-1,\,s(z)=-6,\
t(z)=\hat T^{-1}(z)+zA^{-1}=(1,1/\gamma) \\
z\in\hat T^4U(D_1):\ & \hat s(z)=1,\,s(z)=6,\ t(z)=\hat T(z)+zA=(1/\gamma,0) \\
z\in\hat TU(D_0)\cup\hat T^3U(D_1):\ & \hat s(z)=2,\,s(z)=10,\
t(z)=\hat T^2(z)-z=(-1/\gamma^2,0) 
\end{align*}
This gives again $\delta=\gamma^2/\gamma=\gamma$ since
$$
\{(1/\gamma,0)A^h:\,h\in\mathbb Z\}=
\pm\{(1/\gamma,0),\,(0,1/\gamma),\,(1/\gamma,1),\,(1,1),\,(1,1/\gamma)\}.
$$

\begin{theorem}
$\pi(z)$ is finite for all $z\in(\mathbb Z[\gamma]\cap[0,1))^2$, but
$\big(T^k\big(1/4,1/(4\gamma^3)\big)\big)_{k\in\mathbb Z}$ is aperiodic.
\end{theorem}

\begin{proof}
Conjecture~\ref{cj} holds for $\lambda=1/\gamma$ since no 
$z\in\mathbb Z[\gamma]^2\cap\mathcal D$ satisfies $\|z'\|_\infty\le\gamma$.
It can be shown that all points in $(\frac12\mathbb Z[\gamma]\cap[0,1))^2$ and
$(\frac13\mathbb Z[\gamma]\cap[0,1))^2$ are periodic as well.
If $z=\big(1/4,1/(4\gamma^3)\big)$, then we have
$S(z)=\big(\gamma^2/4,1/(4\gamma)\big)$,
$S^2(z)=\gamma^2\big(S(z)-(1/\gamma^2,0)\big)=\big((3\gamma-2)/4,\gamma/4\big)$
and $S^3(z)=\gamma^2\big(S^2(z)-(1/\gamma,1/\gamma^2)\big)=
\big(1/4,1/(4\gamma^3)\big)=z$.
\end{proof}

\begin{figure}
\begin{minipage}{7.3cm}
\centerline{\includegraphics[scale=.8]{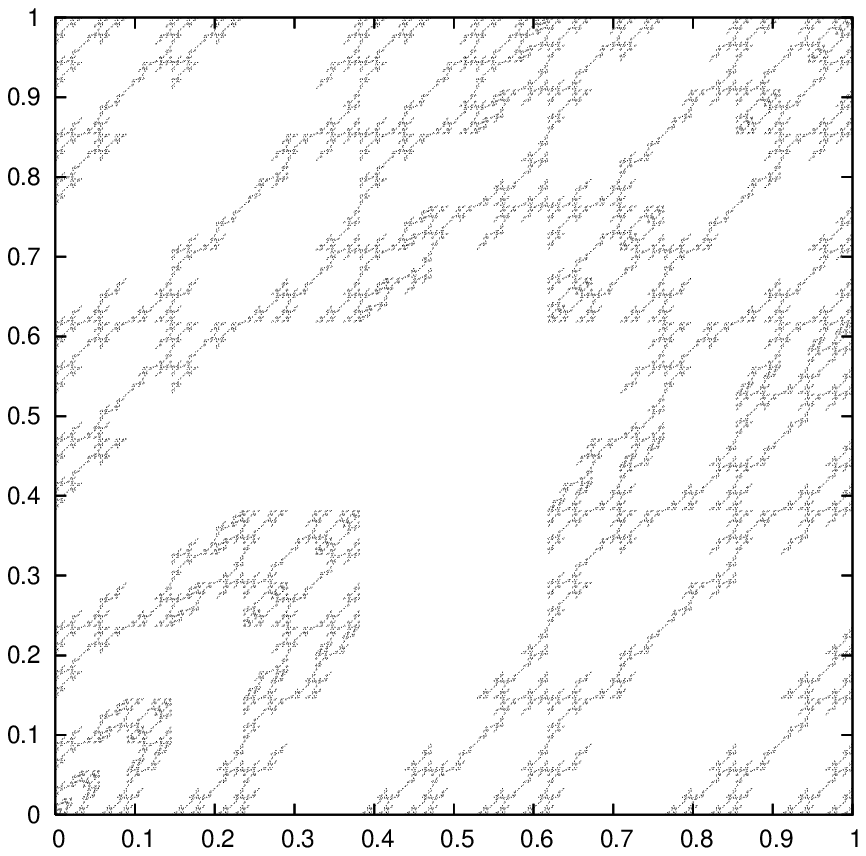}}
\caption{\mbox{Aperiodic points, $\lambda=1/\gamma$.}}\label{fig5a}
\end{minipage}
\begin{minipage}{7.3cm}
\centerline{\includegraphics[scale=.8]{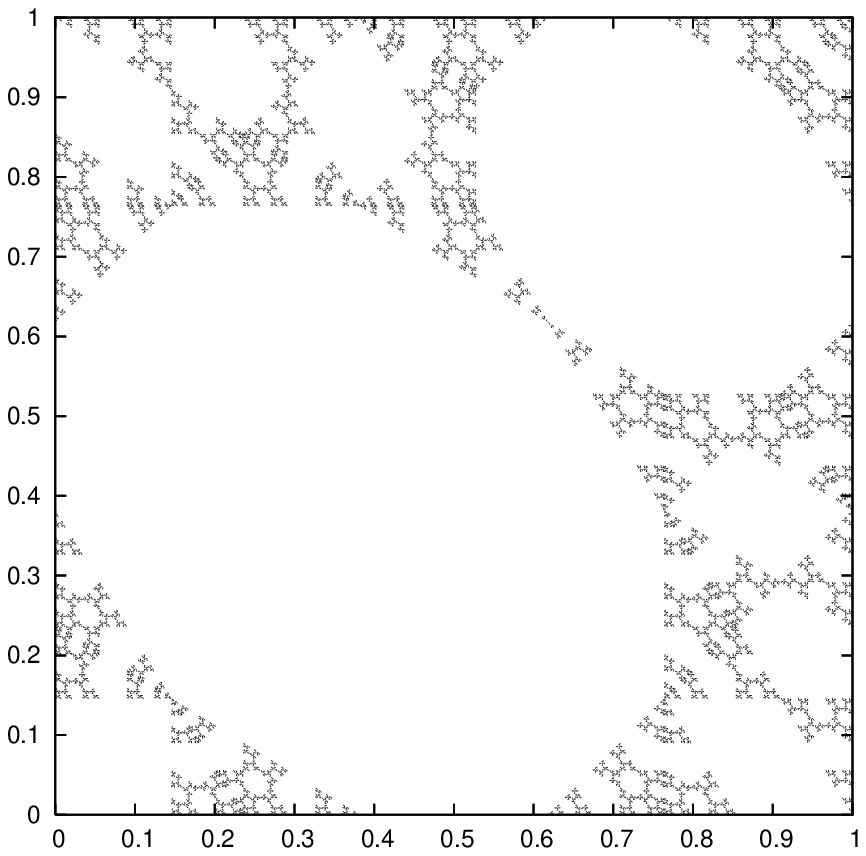}}
\caption{\mbox{Aperiodic points, $\lambda=-\gamma$.}}\label{fig6a}
\end{minipage}
\end{figure}

\section{The case $\lambda=-\gamma=-2\cos\frac\pi5$}

Let $\lambda=-\gamma$ ($\lambda'=1/\gamma$) and set
$$
\mathcal D=\{(x,y)\in[0,1)^2:\,x<y,\,\gamma x+y\ge4-\gamma\}=D_0\cup D_1
$$
with $D_0=\{(x,y)\in\mathcal D:x>1-1/\gamma^5\}$,
$D_1=\{(x,y)\in\mathcal D:x\le1-1/\gamma^5\}$.
Figure~\ref{figR6} shows that $\hat T(z)=T^{\tau(\ell)}(z)$ if $z\in D_\ell$,
with $\tau(0)=42$, $\tau(1)=28$, and
$$
\mathcal R=\{(0,0)\}\cup D_A\cup D_B\cup\bigcup\nolimits_{k=0}^4T^k(D_\Gamma)
\cup\bigcup\nolimits_{k=0}^1T^k(D_\Delta)\cup
\bigcup\nolimits_{k=0}^{24}T^k(D_E)\cup\bigcup\nolimits_{k=0}^{10}T^k(D_Z)
$$
with $D_A=\{z:T^{k+1}(z)=T^k(z)A+(0,1)\mbox{ for all }k\in\mathbb Z\}$,
$D_B=\{z:T^{k+1}(z)=T^k(z)A+(0,2)\}$, $D_\Delta=\{z\in[0,1)^2:T^{2k+1}(z)=
T^{2k}(z)A+(0,2),\,T^{2k}(z)=T^{2k-1}(z)A+(0,1)\mbox{ for all }k\in\mathbb Z\}$,
$D_\Gamma=\{(x,y):0\le x,y\le1/\gamma^4\}\setminus
\{(0,0),(1/\gamma^4,1/\gamma^4)\}$, $D_E=\{(x,x):1-1/\gamma^5<x<1\}$,
$D_Z=\{(1-1/\gamma^5,1-1/\gamma^5)\}$.
Set $\kappa=1/\gamma^2$, $V(z)=\gamma^4\big((1,1)-z)$, i.e.
$$
U(z)=(1,1)-\big((1,1)-z\big)/\gamma^2=z/\gamma^2+(1/\gamma,1/\gamma).
$$
Then Figure~\ref{figP6} shows that the conditions in Section~\ref{sectgeneral}
are satisfied by $\varepsilon=1$ and
$$
\sigma:\ 0\mapsto010\qquad 1\mapsto01110.
$$
All points in $\mathcal P=D_\alpha\cup D_\beta$ are periodic, with
$D_\alpha=\{z\in\mathcal D:\hat T^k(z)\in D_0\mbox{ for all }k\in\mathbb Z\}$,
$D_\beta=\{z\in\mathcal D:\hat T^k(z)\in D_1\mbox{ for all }k\in\mathbb Z\}$.
Since $|\sigma^n(\ell)|\to\infty$ as $n\to\infty$ for all $\ell\in\mathcal A$,
all conditions of Proposition~\ref{propperiodic} and Theorem~\ref{thmperiodic}
are satisfied, and we obtain the following theorem.

\begin{figure}
\includegraphics[scale=.92]{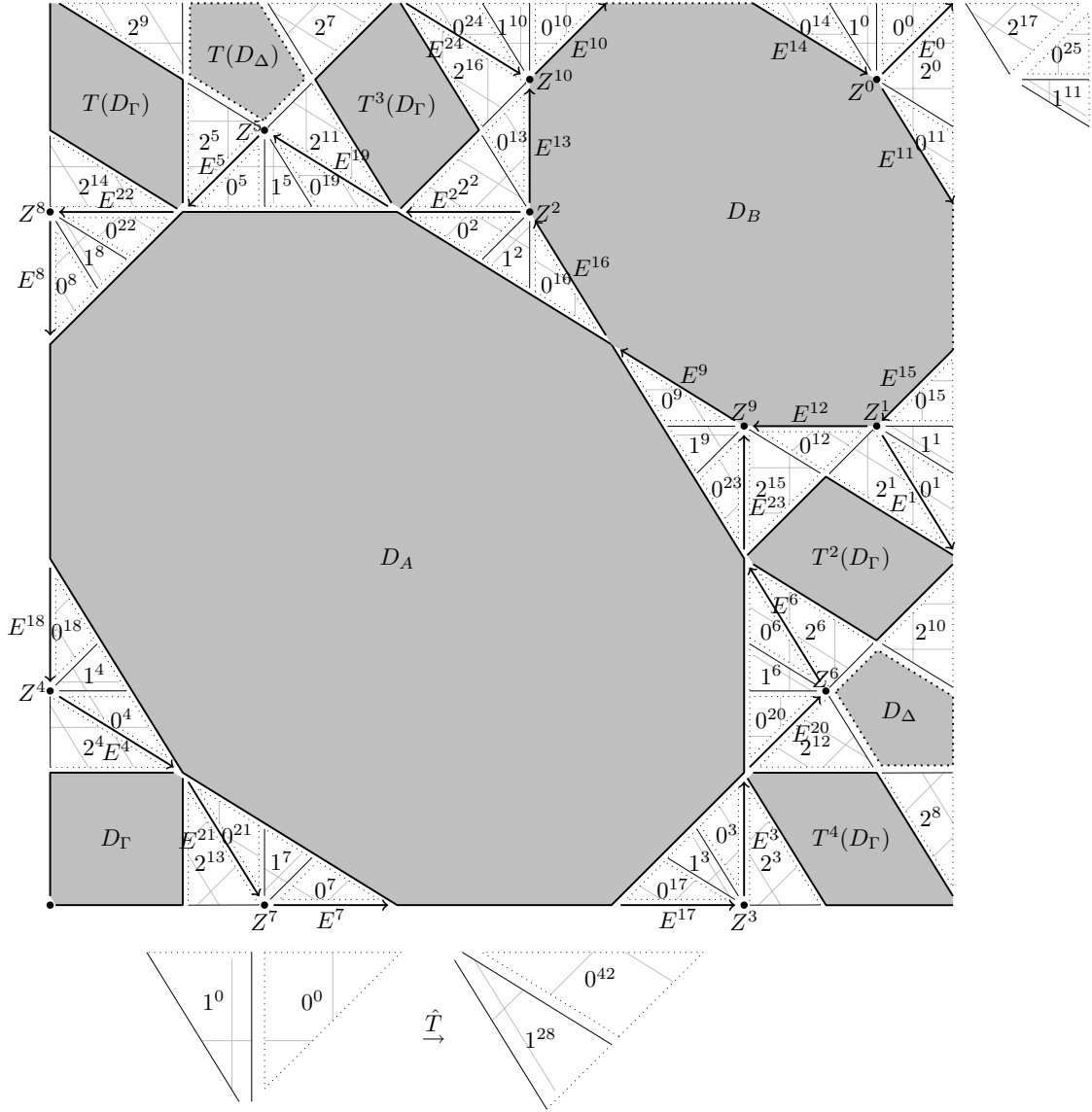}
\caption{The map $\hat T$ and the set $\mathcal R$, $\lambda=-\gamma$.
($\ell^k$ stands for $T^k(D_\ell)$.)} \label{figR6}
\end{figure}

\begin{figure}
\includegraphics{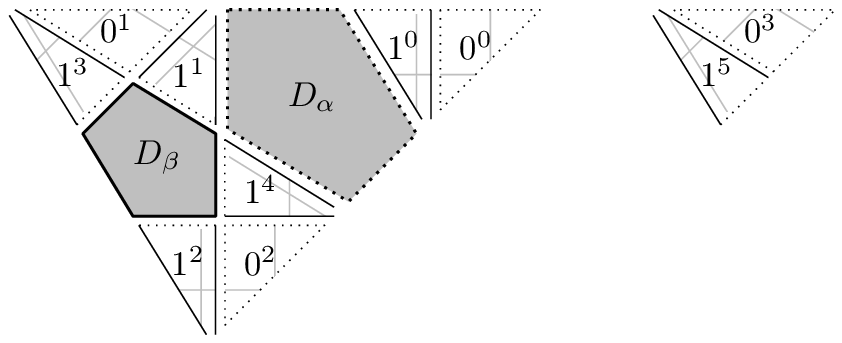}
\caption{The trajectory of the scaled domains and $\mathcal P$,
$\lambda=-\gamma$. ($\ell^k$ stands for $\hat T^kU(D_\ell)$.)} \label{figP6}
\end{figure}

\begin{theorem}\label{thperiods6}
If $\lambda=-\gamma$, then the minimal period length $\pi(z)$ is \\
\centerline{ \begin{tabular}{cl}
$1$ & if $z\in\{(0,0),(1/\gamma^2,1/\gamma^2),(2/\gamma^2,2/\gamma^2)\}$ \\
$2$ & if $z\in\{(\frac{5-\gamma}{\gamma^2+1},\frac{2/\gamma^2}{\gamma^2+1}),
(\frac{2/\gamma^2}{\gamma^2+1},\frac{5-\gamma}{\gamma^2+1})\}$ \\
$5$ & if $z=T^m(1/(2\gamma^4),1/(2\gamma^4))$ for some $m\in\{0,1,2,3,4\}$ \\
$10$ & for the other points of $D_A$, $D_B$, $T^m(D_\Gamma)$, $T^m(D_\Delta)$\\
$11$ & if $z=T^m(1-1/\gamma^5,1-1/\gamma^5)$ for some $m\in\{0,1,\ldots,10\}$\\
$25$ & if $z=T^m(1-1/(2\gamma^5),1-1/(2\gamma^5))$ for some
$m\in\{0,1,\ldots,24\}$ \\
$50$ & for the other points of $T^m(D_E)$ \\
$2(35\cdot 4^n+28)/3$ & if $S^nR(z)$ is the center of $D_\alpha$ \\
$10(35\cdot 4^n+28)/3$ & for the other points of $D_\alpha$ \\
$4(35\cdot 4^n-14)/3$ & if $S^nR(z)$ is the center of $D_\beta$ \\
$20(35\cdot 4^n-14)/3$ & for the other points of $D_\beta$ \\
$\infty$ & if $S^nR(z)\in\mathcal D\setminus\mathcal P$ for all $n\ge0$
\end{tabular} }
\end{theorem}

\begin{proof}
As for $\lambda=-1/\gamma$ and $\lambda=1/\gamma$, we have
\begin{gather*}
\begin{pmatrix}|\sigma^n(0)|_0\\|\sigma^n(0)|_1\end{pmatrix}=
4^n\begin{pmatrix}1/3\\1/3\end{pmatrix}+\begin{pmatrix}2/3\\-1/3\end{pmatrix},
\qquad
\begin{pmatrix}|\sigma^n(1)|_0\\|\sigma^n(1)|_1\end{pmatrix}=
4^n\begin{pmatrix}2/3\\2/3\end{pmatrix}+\begin{pmatrix}-2/3\\1/3\end{pmatrix},
\end{gather*}
hence $\tau(\sigma^n(0))=(70\cdot4^n+56)/3$,
$\tau(\sigma^n(1))=(140\cdot 4^n-56)/3$.
For $S^nR(z)\in D_\alpha$, we have $\pi(z)=\tau(\sigma^n(0))$ and
$5\tau(\sigma^n(0))$ respectively; if $S^nR(z)\in D_\beta$, then
$\pi(z)=\tau(\sigma^n(1))$ and $5\tau(\sigma^n(1))$ respectively.
\end{proof}

We choose $\hat s(z)$ as follows and obtain the following $s(z),t(z)$:
\begin{align*}
z\in\hat T^2U(D_0\cup D_1):\ & \hat s(z)=-2,\,s(z)=-70,\
t(z)=V(\hat T^{-2}(z))-V(z)=(-1/\gamma^2,-1/\gamma^2) \\
z\in\hat TU(D_1):\ & \hat s(z)=-1,\,s(z)=-42,\
t(z)=V(\hat T^{-1}(z))-V(z)A^{-2}=(1/\gamma,1/\gamma) \\
z\in\hat T^4U(D_1):\ & \hat s(z)=1,\,s(z)=42,\
t(z)=V(\hat T(z))-V(z)A^2=(1,0) \\
z\in\hat TU(D_0)\cup\hat T^3U(D_1):\ & \hat s(z)=2,\,s(z)=70,\
t(z)=V(\hat T^2(z))-V(z)=(-1/\gamma,0)
\end{align*}
This gives again $\delta=\gamma^2/\gamma=\gamma$ since
$$
\{(1,0)A^h:\,h\in\mathbb Z\}=
\pm\{(1,0),\,(0,1),\,(1,-1/\gamma),\,(1/\gamma,1/\gamma),\,(1/\gamma,-1)\}.
$$

\begin{theorem}
$\pi(z)$ is finite for all $z\in(\mathbb Z[\gamma]\cap[0,1))^2$, but
$\pi\big(1-1/(3\gamma^2),1-1/(3\gamma^5)\big)=\infty$.
\end{theorem}

\begin{proof}
Since $V(\mathcal D)=\{(x,y):x>y>0,\gamma x+y\le\gamma\}$, we have no point
$z\in\mathbb Z[\gamma]^2\cap\mathcal D$ with $\|V(z)'\|_\infty\le\gamma$, and
Conjecture~\ref{cj} holds for $\lambda=-\gamma$.
If $V(z)=\big(\gamma^2/3,1/(3\gamma)\big)$, then we have
$$
VS(z)=\gamma^2\Big(V(z)-\Big(\frac1\gamma,0\Big)\Big)=
\Big(\frac23,\frac\gamma3\Big),\
VS^2(z)=\gamma^2\Big(VS(z)-\Big(\frac1{\gamma^2},\frac1{\gamma^2}\Big)\Big)=
\Big(\frac{\gamma^2+1}{3\gamma},\frac2{3\gamma}\Big), 
$$
$VS^3(z)=\gamma^2\big(VS^2(z)-\big(\frac1{\gamma^2},\frac1{\gamma^2}\big)\big)=
\big(\frac{3\gamma-2}3,\frac1{3\gamma^3}\big)$ and
$VS^4(z)=\gamma^2\big(VS^3(z)-\big(\frac1\gamma,0\big)\big)=V(z)$.
\end{proof}

\section{The case $\lambda=\sqrt3=-2\cos\frac{5\pi}6$}\label{sectsqrt3}

\begin{figure}
\includegraphics[scale=.95]{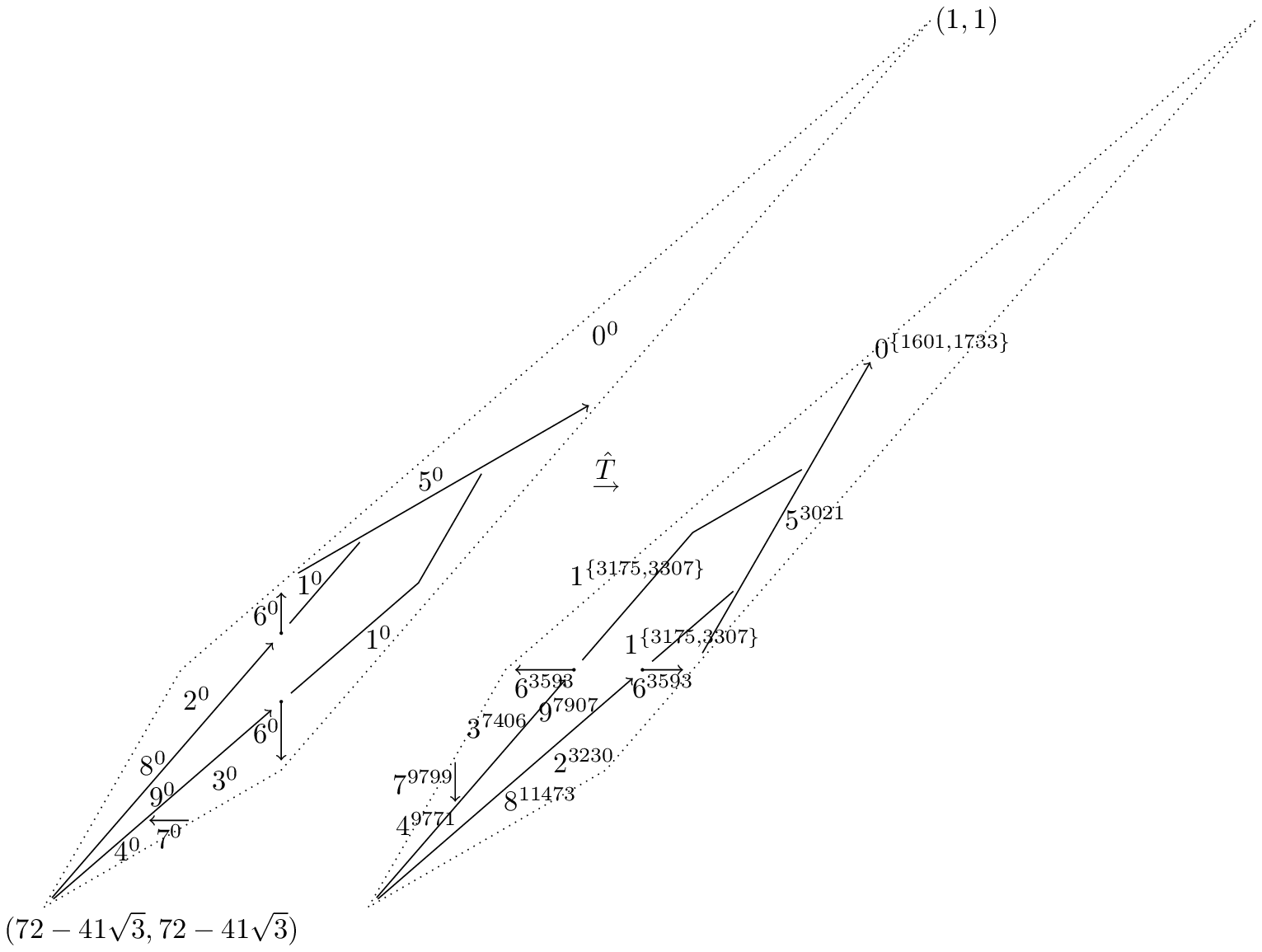} 
\includegraphics[scale=.95]{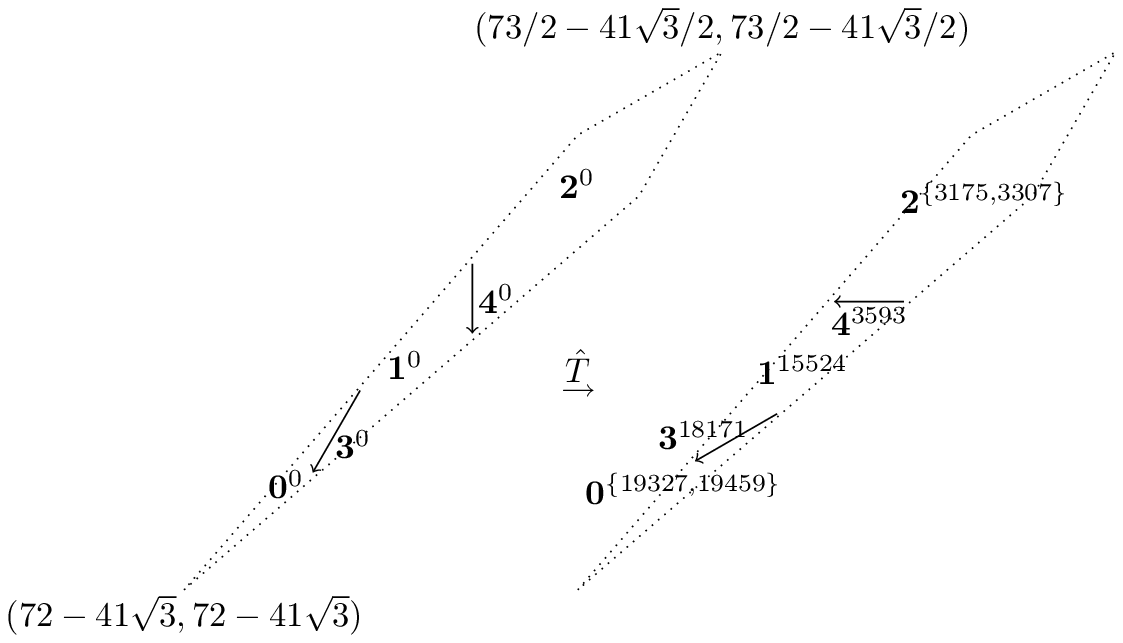}
\caption{\mbox{The first return map on $\mathcal D_1$ and $\mathcal D_2$ 
respectively, $\lambda=\sqrt3$.  ($\ell^k$ stands for $T^k(D_\ell)$.)}}
\label{figT7}
\end{figure}

The case $\lambda=\sqrt3$ is much more involved than the previous cases.
Therefore we show only that all points in $(\mathbb Z[\sqrt3]\cap[0,1))^2$ are 
periodic and refrain from calculating the period lengths. 
Furthermore we postpone the determination of $\hat T$ and $\mathcal R$ to 
Appendix~\ref{app7}.
Let 
$$
\mathcal D = \{(x,y):2x-\sqrt3y<2-\sqrt3,2y-\sqrt3x<2-\sqrt3,
y-\sqrt3x<195-113\sqrt3,x-\sqrt3y<195-113\sqrt3\}
$$
and $\mathcal D_1=\mathcal D\setminus\mathcal D_2$, where $\mathcal D_2$ is 
defined by the inequalities
$$
2x-\sqrt3y>267-154\sqrt3,\,2y-\sqrt3x>267-154\sqrt3,\,
y-\sqrt3x>98-57\sqrt3,\,x-\sqrt3y>98-57\sqrt3.
$$
The sets $\mathcal D_1$ and $\mathcal D_2$ have to be treated separately
because their trajectories are disjoint, and both sets contain aperiodic points.
The trajectories of aperiodic points in $\mathcal D_1$ come arbitrarily close 
to $(1,1)$, whereas $(72-41\sqrt3,72-41\sqrt3)$ is a limit point in 
$\mathcal D_2$.
(Note that $72-41\sqrt3=1-(\sqrt3+1)(2-\sqrt3)^4\approx0.9859$.)
The scaling maps are 
\begin{align*}
U_1(z) & =(2-\sqrt3)z+(\sqrt3-1,\sqrt3-1)=V_1^{-1}(\kappa V_1(z))\qquad
\mbox{for }z\in\mathcal D_1, \\ 
U_2(z) & =(2-\sqrt3)z+(113\sqrt3-95,113\sqrt3-195)=V_2^{-1}(\kappa V_2(z)) 
\qquad \mbox{for }z\in\mathcal D_2,
\end{align*}
with $\kappa=2-\sqrt3$, $V_1(z)=\big((1,1)-z\big)/\kappa^4$, 
$V_2(z)=\big(z-(72-41\sqrt3,72-41\sqrt3)\big)/\kappa^5$.
Then we have
\begin{align*}
V_1(\mathcal D)&=\{(x,y):2x>\sqrt3y,\,2y>\sqrt3x,\,x>\sqrt3y-2,\,y>\sqrt3x-2\},
\\ V_2(\mathcal D_2)&=
\{(x,y):2x>\sqrt3y,\,2y>\sqrt3x,\,x>\sqrt3y-2-\sqrt3,\,y>\sqrt3x-2-\sqrt3\}.
\end{align*}
The first return map $\hat T$ induces a partition of $\mathcal D_1$ into sets 
$D_0,\ldots,D_9$ and a partition of $\mathcal D_2$ into sets 
$D_{\mathbf 0},\ldots,D_{\mathbf 4}$, as in Figure~\ref{figT7}.
These sets are defined by the following (in)equalities:
\begin{gather*}
\begin{array}{c|c|c|c|c}
V_1(D_0) & V_1(D_1) & V_1(D_2) & V_1(D_3) & V_1(D_4) \\ \hline
x>\sqrt3y-1 & x>\sqrt3y-1 & 2x>\sqrt3y+\sqrt3-1 & 2y>\sqrt3x+\sqrt3-1 &
2y>\sqrt3x+\sqrt3-1 \\ & x<2 & x>2 & x>2,\ y<2\sqrt3-1 & y>2\sqrt3-1 
\end{array} \\
\begin{array}{c|c|c|c|c}
V_1(D_5) & V_1(D_6) & V_1(D_7) & V_1(D_8) & V_1(D_9) \\ \hline
x=\sqrt3y-1 & x=2 & y=2\sqrt3-1 & 2x=\sqrt3y+\sqrt3-1 & 2y=\sqrt3x+\sqrt3-1 \\ 
& & x<3-1/\sqrt3 & x>2 & x>2
\end{array} \\
\begin{array}{c|c|c|c|c}
V_2(D_{\mathbf 0}) & V_2(D_{\mathbf 1}) & V_2(D_{\mathbf 2}) & 
V_2(D_{\mathbf 3}) & V_2(D_{\mathbf 4}) \\ \hline
y>\sqrt3x-1 & y<\sqrt3x-1,\,x<\sqrt3+1 & x>\sqrt3+1 & y=\sqrt3x-1 & x=\sqrt3+1 
\end{array} 
\end{gather*}

The return times of $z\in D_\ell$ to $\mathcal D$ are given by the following 
tables. 
\begin{gather*}
\begin{array}{c|c|c|c|c|c|c|c|c|c}
D_0 & D_1 & D_2 & D_3 & D_4 & D_5 & D_6 & D_7 & D_8 & D_9 \\ \hline 1601,\,1733 
& 3175,\,3307 & 3230 & 7406 & 9771 & 3021 & 3593 & 9799 & 11473 & 7907
\end{array} \\
\begin{array}{c|c|c|c|c}
D_{\mathbf 0} & D_{\mathbf 1} & D_{\mathbf 2} & D_{\mathbf 3} & D_{\mathbf 4}\\
\hline 19459 & 15524 & 3175,\,3307 & 18171 & 3593 
\end{array} 
\end{gather*}
Note that the return times are not constant on all $D_\ell$.
E.g., the return time for $z\in D_0$ is $1601$ if $V_1(z)=(1,y)$ and $1733$ 
else, see Appendix~\ref{app7} for details.
Since we do not calculate the period lengths, it is not necessary to 
distinguish between the parts of $D_\ell$ with different period lengths. 

\begin{figure}
\includegraphics[scale=.95]{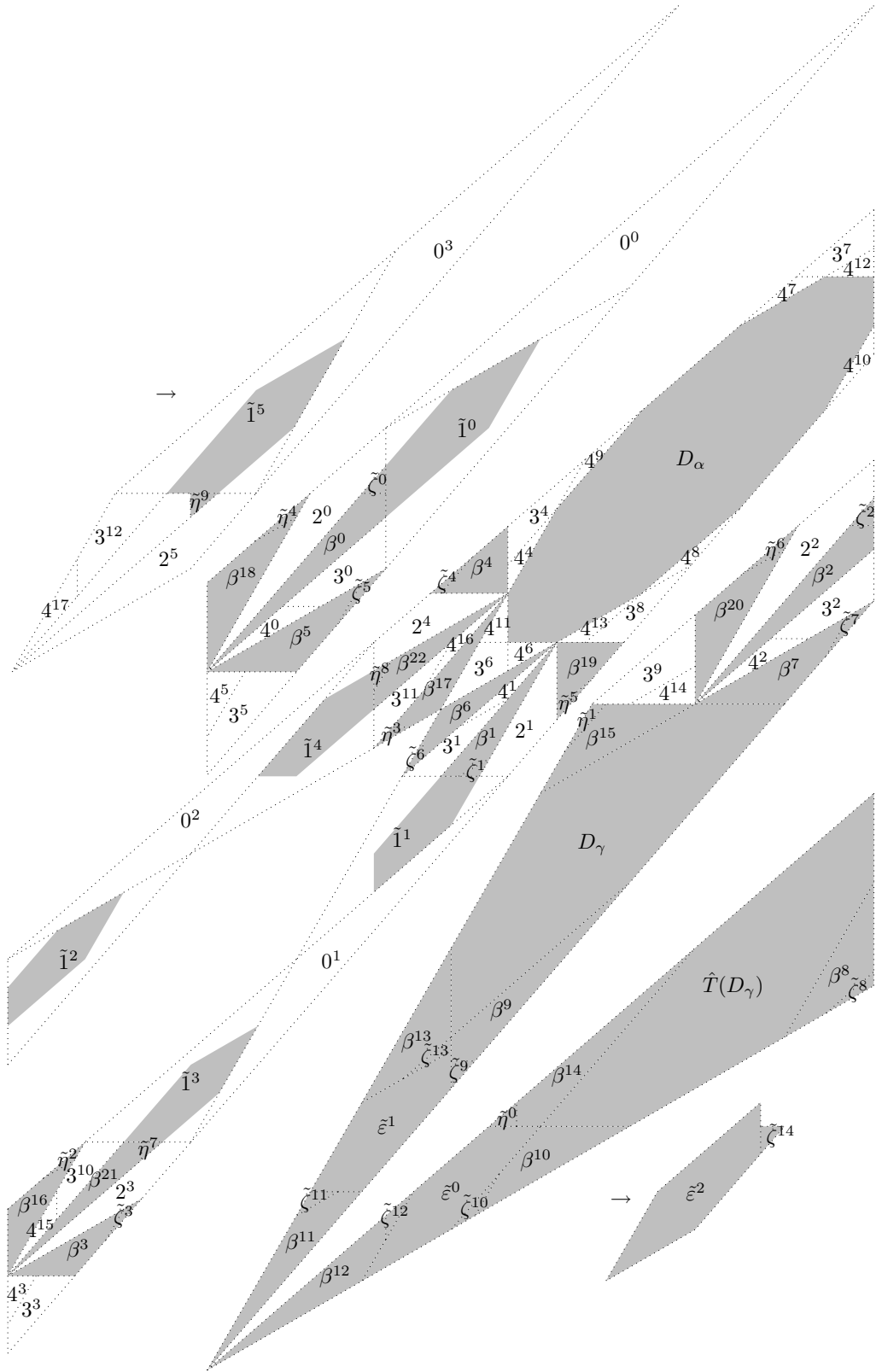}
\caption{The trajectory of the open scaled sets in $\mathcal D_1$ and the set
$\mathcal P_1$, $\lambda=\sqrt3$. \newline ($\ell^k$ stands for 
$\hat T^kU_1(D_\ell)$ if $\ell\in\{0,\tilde1,2,3,4\}$, for $\hat T^k(D_\ell)$ 
else.)} \label{figP71}
\end{figure}

\begin{figure}
\includegraphics[scale=.95]{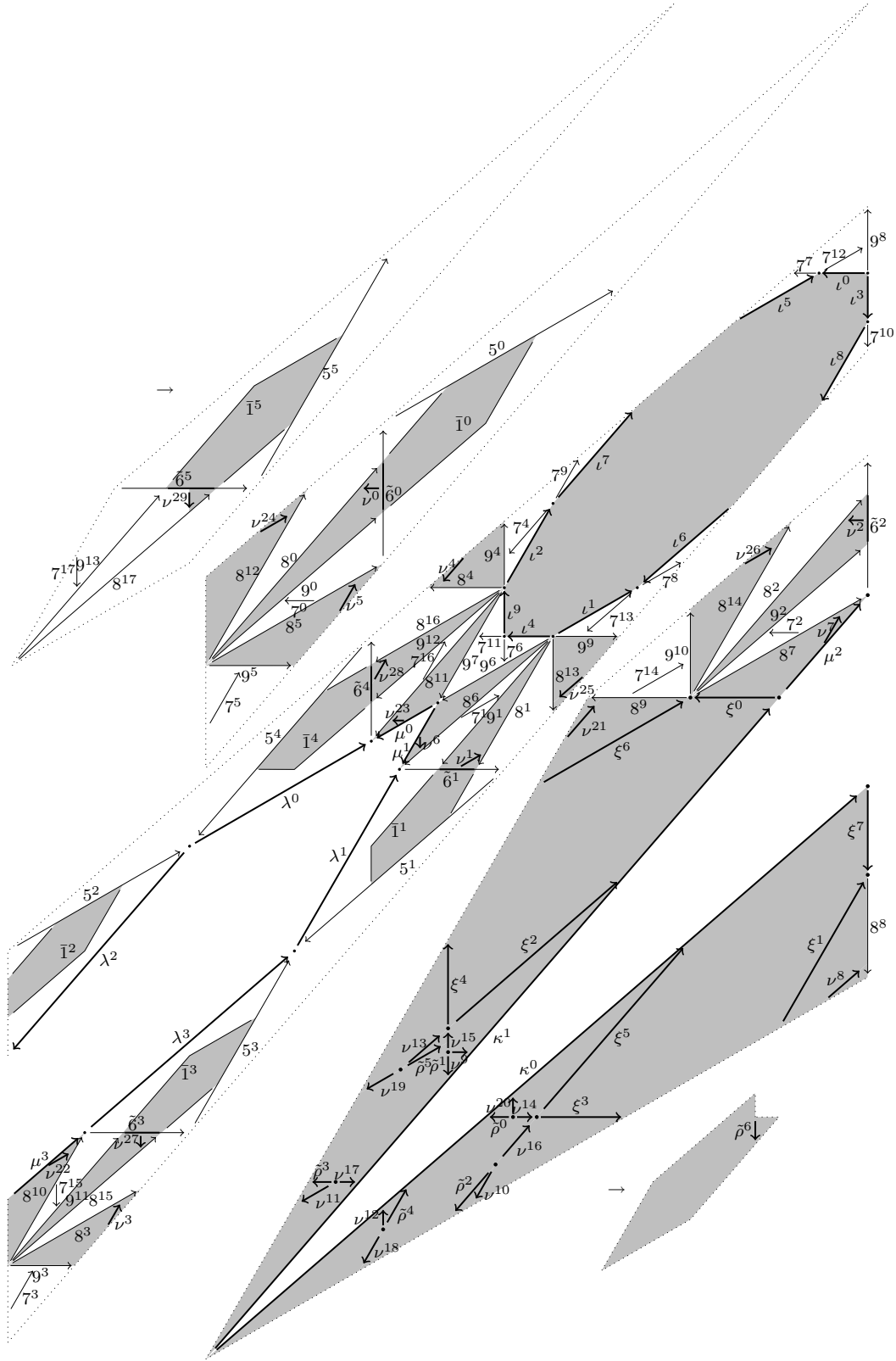}
\caption{The trajectory of the scaled lines and the set $\mathcal P_1$, 
$\lambda=\sqrt3$. \newline ($\ell^k$ stands for $\hat T^kU_1(D_\ell)$ if 
$\ell\in\{\bar1,5,\tilde6,7,8,9\}$, for $\hat T^k(D_\ell)$ else.)} 
\label{figP72}
\end{figure}

\begin{figure}
\includegraphics{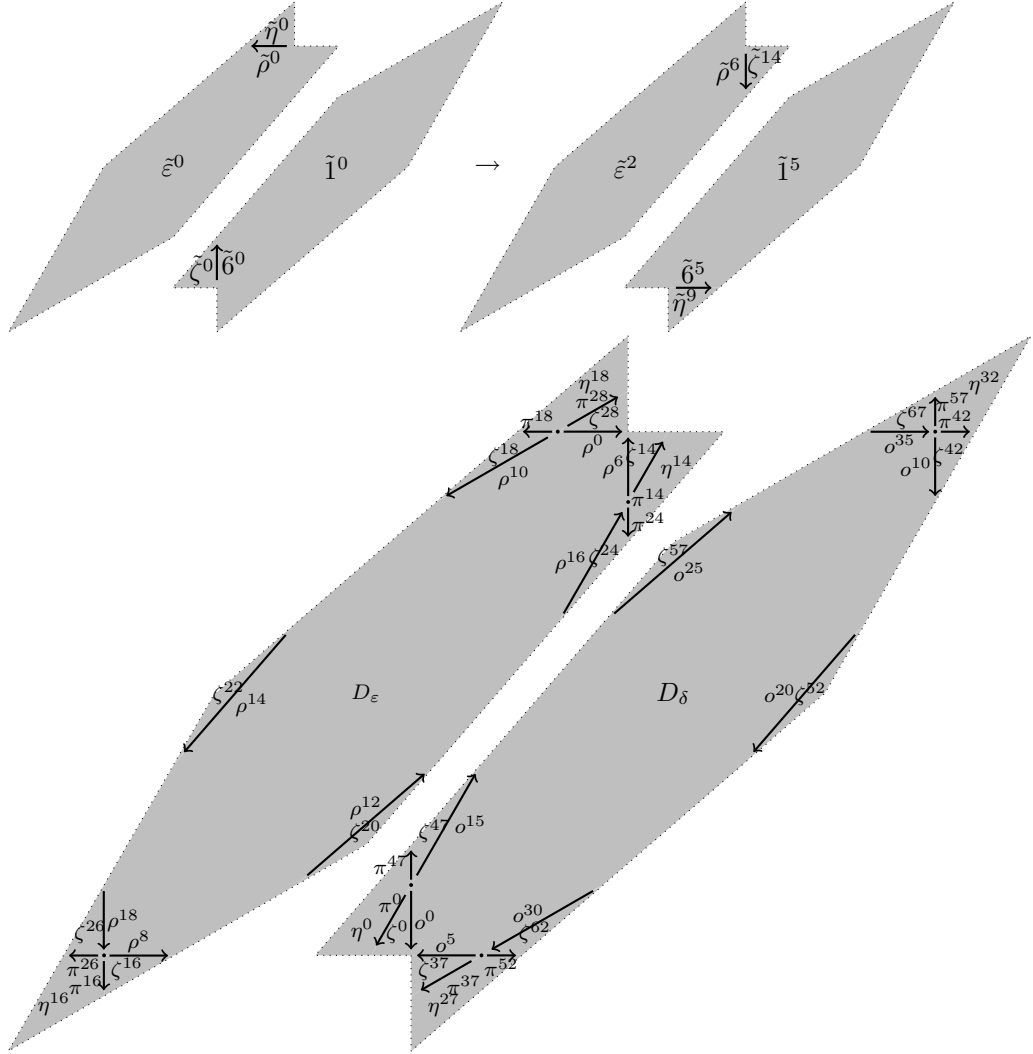}
\caption{Small parts of $\mathcal P_1$, $\lambda=\sqrt3$. ($\ell^k$ stands for 
$\hat T^k(D_\ell)$ for $\ell\not\in\{\tilde 1,\tilde 6\}$.)} \label{figP73}
\end{figure}

\subsection{The scaling domain $\mathcal D_1$.}
Figure~\ref{figP71} shows the trajectory of the open scaled sets in 
$\mathcal D_1$.
Here, $V_1(\mathcal D_1)$ is split up into the three stripes $x<\sqrt3-1$,
$\sqrt3-1<x<2$ and $x>2$, and $D_{\tilde1}$ denotes the set given by 
$V_1(D_{\tilde1})=\{(x,y)\in V_1(\mathcal D):x>\sqrt3y-1,x<2\}$.
We see that 
$$
\sigma_1: 
\begin{array}{rclrclrclrcl}
0 & \mapsto & 010 & 3 & \mapsto & 012100001210 & 5 & \mapsto & 01510 &
7 & \mapsto & 01210000500001210 \\
1 & \mapsto & 01110 & 4 & \mapsto & 01210000000001210 & 6 & \mapsto & 01610 &
8 & \mapsto & 01210012621001210 \\
2 & \mapsto & 01210 & & & & & & & 9 & \mapsto & 0121005001210
\end{array}
$$
codes the trajectory of $U_1(D_\ell)$, $\ell\in\{0,1,2,3,4\}$, with
$\hat T^{|\sigma_1(\ell)|}U_1(z)=U_1\hat T(z)$ for $z\in D_\ell$.
All points in $D_\alpha$, $D_\beta$ and $D_\gamma$ are periodic.
Figure~\ref{figP73} shows that 
$D_{\tilde\varepsilon},D_{\tilde\zeta},D_{\tilde\eta}$ and the grey part of
$U_1(\mathcal D_{\tilde1})$ split up further, but all their points are periodic 
as well.

The trajectory of the scaled lines is depicted in Figure~\ref{figP72}, where 
again $V_1(\mathcal D_1)$ is split up into the stripes $x<\sqrt3-1$,
$\sqrt3-1\le x<2$ and $x\ge2$.
Here, $D_{\bar1}$ denotes boundary lines of $D_1$, and $D_{\tilde6}$ is given 
by $V_1(D_{\tilde6})=\{(2,y)\in V_1(\mathcal D)\}$.
We see that $\sigma_1$ codes the trajectory of $U_1(D_\ell)$,
$\ell\in\{5,6,7,8,9\}$, as well and satisfies the conditions in 
Section~\ref{sectgeneral} (with respect to $\mathcal D_1$). 
All points in $D_\iota,D_\kappa,D_\lambda,D_\mu,D_\nu,D_\xi,D_o,D_\pi,D_\rho$ 
(and their orbits) are periodic.
The finitely many remaining points in $\mathcal P_1=\{z\in\mathcal D_1:
\hat T^m(z)\not\in U_1(\mathcal D_1)\mbox{ for all }m\in\mathbb Z\}$ are 
clearly periodic as well.
Since $|\sigma_1^n(\ell)|\to\infty$ for all $\ell\in\{0,\ldots,9\}$, we can
use Proposition~\ref{propaper} to show the following proposition.

\begin{proposition}\label{propper71}
$\pi(z)$ is finite for all $z\in\mathbb Z[\sqrt3]^2\cap\mathcal D_1$, but
$\pi(V_1^{-1}(\sqrt3+1/4,7/4))=\infty$.
\end{proposition}

\begin{proof}
First we show that only $D_0$ and $D_1$ contain aperiodic points: 
$D_3,D_4,D_7,D_8,D_9$ lie in~$\mathcal P_1$.
The only part of $D_2$ which is not in $\mathcal P_1$ or 
$\hat T^mU_1(\mathcal P_1)$, lies in $\hat T^2U_1(D_2)$. 
By iterating this argument on $\hat T^2U_1(D_2)$, the possible set of aperiodic 
points in $D_2$ becomes smaller and smaller, and converges to 
$V_1^{-1}(2,\sqrt3)\not\in D_2$.
A similar reasoning shows that all points in $D_5$ and $D_6$ are periodic.
Therefore it is sufficient to determine $t(z)$ for points in the trajectories 
of $U_1(D_0\cup D_1)$.
\begin{align*}
z\in\hat TU_1(D_0)\cup\hat T^3U_1(D_1):\ & \hat s(z)=2,\, 
s(z)\equiv 0\bmod 12,\, t(z)=(1-\sqrt3)(\sqrt3,2) \\
z\in\hat T^4U_1(D_1):\ & \hat s(z)=1,\,s(z)\equiv 5\bmod 12,\,
t(z)=V_1(\hat T(z))-V_1(z)A^5=(\sqrt3,2) \\
z\in\hat TU_1(D_1):\ & \hat s(z)=-1,\,s(z)\equiv-5\bmod 12,\, t(z)=(2,\sqrt3)\\
z\in\hat T^2U_1(D_0)\cup\hat T^2U_1(D_1):\ & \hat s(z)=-2,\,
s(z)\equiv 0\bmod 12,\, t(z)=(1-\sqrt3)(2,\sqrt3)
\end{align*}
We have $\delta_1=(\sqrt3+1)2/(\sqrt3+1)=2$ since
$$
\{(\sqrt3,2)A^h:\,h\in\mathbb Z\}=\pm
\{(\sqrt3,2),\,(2,\sqrt3),\,(\sqrt3,1),\,(1,0),\,(0,1),\,(1,\sqrt3)\}.
$$
The only point $z\in V_1(\mathbb Z[\sqrt3]^2\cap\mathcal D_1)$ with 
$\|z'\|_\infty\le2$ is $(1,1)\in V_1(D_\alpha)$.

If $V_1(z)=(\sqrt3+1/4,7/4)$, then we have
\begin{align*}
V_1S(z) & = (2+\sqrt3)\big(V_1(z)+(1-\sqrt3)(2,\sqrt3)\big)=
(3/2+\sqrt3/4,3\sqrt3/4+1/2), \\
V_1S^2(z) & = (2+\sqrt3)\big(V_1S(z)+(1-\sqrt3)(2,\sqrt3)\big)=
(7/4,\sqrt3+1/4),
\end{align*}
$V_1S^3(z)=(2+\sqrt3)\big(V_1S^2(z)+(1-\sqrt3)(\sqrt3,2)\big)=
(3\sqrt3/4+1/2,3/2+\sqrt3/4)$, $V_1S^4(z)=V_1(z)$.
\end{proof}

\noindent{\it Remark.}
The primitive part of $\sigma_1$ is again $0\mapsto 010$, $1\mapsto 01110$.

\begin{figure}
\includegraphics{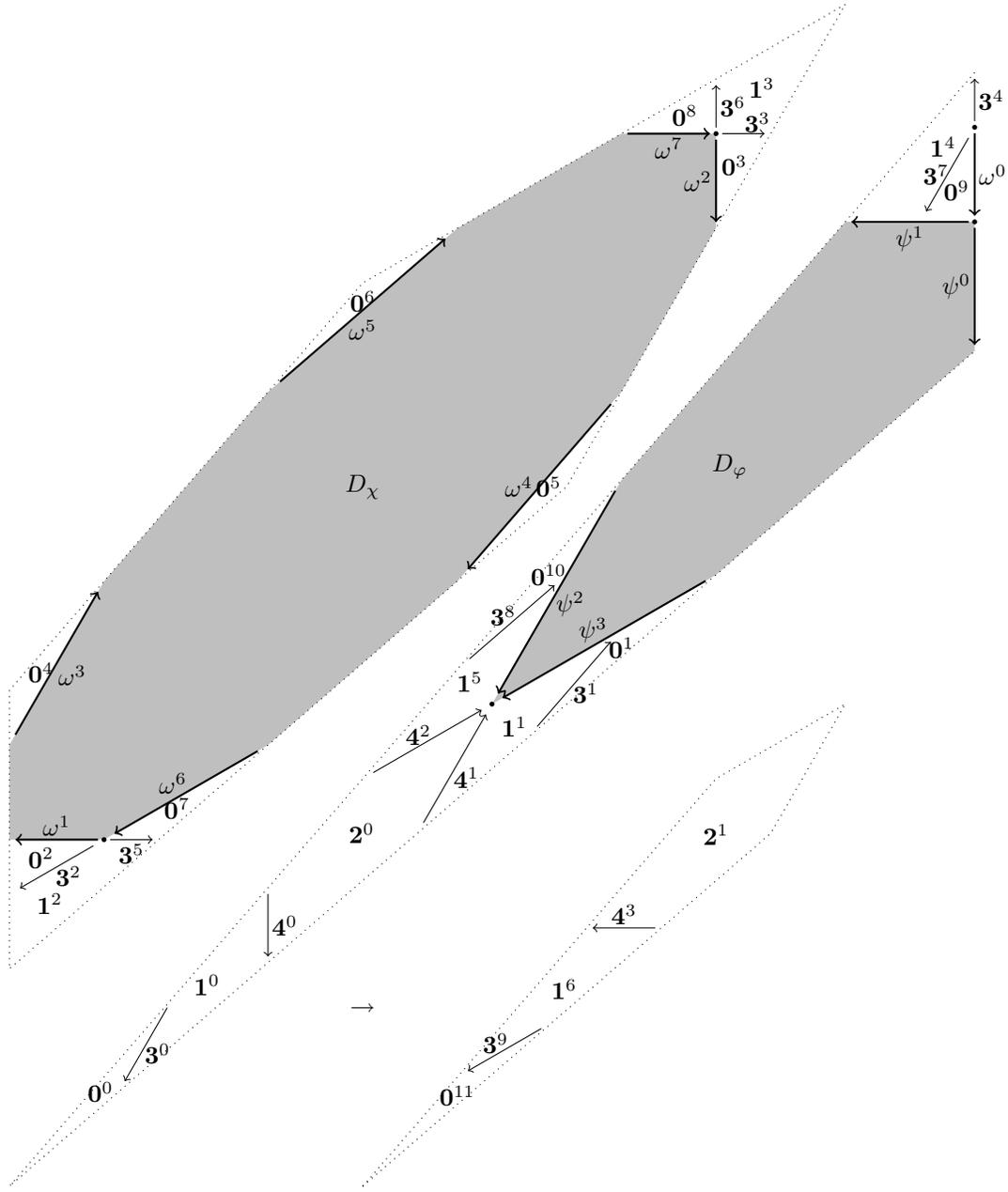}
\caption{The trajectory of the scaled domains in $\mathcal D_2$ and the set
$\mathcal P_2$, $\lambda=\sqrt3$. \newline ($\ell^k$ stands for 
$\hat T^k(D_\ell)$ if $\ell\in\{\psi,\omega\}$, for $\hat T^kU_2(D_\ell)$ 
else.)} \label{figP74}
\end{figure}

\subsection{The scaling domain $\mathcal D_2$}
Figure~\ref{figP74} shows the trajectory of the the scaled domains in 
$\mathcal D_2$.
Here, $V_2(\mathcal D_2)$ is split up into $x\le\sqrt3+1$ and $x>\sqrt3+1$.
With $\varepsilon_2=1$ and 
$$
\sigma_2: 
\begin{array}{rclrcl}
\mathbf 0 & \mapsto & \mathbf 0\mathbf 1\mathbf 2\mathbf 2 \mathbf 2\mathbf 2
\mathbf 2\mathbf 2\mathbf 2\mathbf 1\mathbf 0 & \mathbf 3 & \mapsto & \mathbf 0
\mathbf 1\mathbf 2 \mathbf 2\mathbf 4\mathbf 2\mathbf 2\mathbf 1\mathbf 0 \\
\mathbf 1 & \mapsto & \mathbf 0\mathbf 1\mathbf 2\mathbf 2\mathbf 1\mathbf 0 &
\mathbf 4 & \mapsto & \mathbf 0\mathbf 3\mathbf 0 \\ 
\mathbf 2 & \mapsto & \mathbf 0
\end{array}
$$
the conditions in Section~\ref{sectgeneral} are satisfied.
The set $\mathcal P_2=\{z\in\mathcal D_2:\hat T^m(z)\not\in U_2(\mathcal D_2)
\mbox{ for all }m\in\mathbb Z\}$ consists of the orbits of 
$D_\varphi,D_\chi,D_\psi,D_\omega$ and several isolated (periodic) points. 
Since $|\sigma_2^n(\ell)|\to\infty$ for all 
$\ell\in\{\mathbf 0,\mathbf 1,\mathbf 2,\mathbf 3,\mathbf 4\}$, we can
use Proposition~\ref{propaper} to show the following proposition.

\begin{proposition}\label{propper72}
$\pi(z)$ is finite for all $z\in\mathbb Z[\sqrt3]^2\cap\mathcal D_2$, but
$\pi(V_2^{-1}(5/7,3\sqrt3/7))=\infty$.
\end{proposition}

\begin{proof}
Similarly to $\mathcal D_1$, we see that all points in $D_{\mathbf 3}$ and
$D_{\mathbf 4}$ are periodic.
Choose $\hat s(z)$ as follows:
\begin{align*}
z\in\hat T^{10}U_2(D_{\mathbf 0})\cup\hat T^5U_2(D_{\mathbf 1}):\ & 
\hat s(z)=1,\, s(z)\equiv 7\bmod 12,\, t(z)=(2,\sqrt3) \\
z\in\hat T^9U_2(D_{\mathbf 0})\cup\hat T^4U_2(D_{\mathbf 1}):\ & \hat s(z)=2,\, 
s(z)\equiv 3\bmod 12,\, t(z)=(1-\sqrt3,\sqrt3-1) \\
z\in\hat T^8U_2(D_{\mathbf 0})\cup\hat T^3U_2(D_{\mathbf 1}):\ & \hat s(z)=3,\, 
s(z)\equiv 10\bmod 12,\, t(z)=(1-\sqrt3,-3) \\
z\in\hat T^7U_2(D_{\mathbf 0}):\ & \hat s(z)=4,\, s(z)\equiv 5\bmod 12,\, 
t(z)=\sqrt3(\sqrt3,2) \\
z\in\hat T^6U_2(D_{\mathbf 0}):\ & \hat s(z)=5,\, s(z)\equiv 0\bmod 12,\, 
t(z)=-2(\sqrt3,2) 
\end{align*}
For the remaining $z\in\hat T^mU_2(D_{\mathbf 0}\cup D_{\mathbf 1})$, 
$\hat s(z),s(z)$ and $t(z)$ are obtained by symmetry.
The sets $\{(1-\sqrt3,\sqrt3-1)A^h:h\in\mathbb Z\}$ and 
$\{(\sqrt3-1,3)A^h:h\in\mathbb Z\}$ are
\begin{gather*}
\pm\{(1-\sqrt3,\sqrt3-1),(\sqrt3-1,2),(2,\sqrt3+1),(\sqrt3+1,\sqrt3+1),
(\sqrt3+1,2),(2,\sqrt3-1)\}, \\
\pm\{(\sqrt3-1,3),(3,2\sqrt3+1),
(2\sqrt3+1,3+\sqrt3),(3+\sqrt3,2+\sqrt3),(2+\sqrt3,\sqrt3),(\sqrt3,1-\sqrt3)\},
\end{gather*}
hence $\delta_2=4/(\sqrt3+1)=2(\sqrt3-1)$.
The only $x\in\mathbb Z[\sqrt3]$ with $0<x<5$ and $|x'|\le2(\sqrt3-1)$ are 
$1,1+\sqrt3,2+\sqrt3$ and $3+\sqrt3$.
Therefore the only $z\in V_2(\mathbb Z[\sqrt3]^2\cap\mathcal D_2)$ with 
$\|z'\|_\infty\le2(\sqrt3-1)$ are $(1,1)$, the center of 
$V_2U_2(D_\chi)$, $(1+\sqrt3,1+\sqrt3)$, the center of $D_{\mathbf 4}$, 
$(2+\sqrt3,2+\sqrt3)$, the center of $D_\chi$, and $(3+\sqrt3,3+\sqrt3)$, a 
fixed point of $\hat T^3$.

If $V_2(z)=(5/7,3\sqrt3/7)$, then we have $V_2S(z)=(2+\sqrt3)V_2(z)$ and \\
$V_2S^2(z)=(2+\sqrt3)\big(V_2S(z)A^3+(1-\sqrt3,\sqrt3-1)\big)=
(5/7,3\sqrt3/7)=V_2(z)$.
\end{proof}

By combining Propositions~\ref{propper71} and~\ref{propper72} and the fact that
all points in $\mathcal R$ are periodic (see Appendix~\ref{app7}), we obtain 
the following theorem.

\begin{theorem}
Conjecture~\ref{cj} holds for $\lambda=\sqrt3$.
\end{theorem}

\noindent{\it Remark.}
The eigenvalues corresponding to the primitive part of $\sigma_2$ 
($\ell\in\{\mathbf 0,\mathbf 1,\mathbf 2\}$) are $5,-2$ and $1$.

\begin{figure}
\begin{minipage}{6.6cm}
\centerline{\includegraphics[scale=.8]{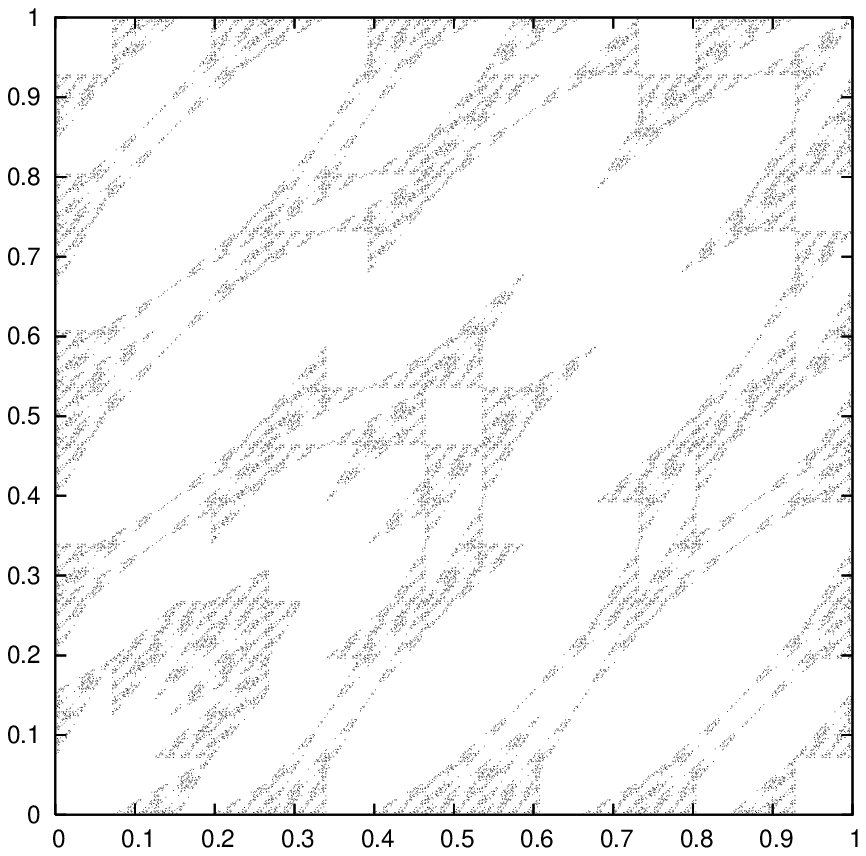}}
\caption{\mbox{Aperiodic points, $\lambda=\sqrt3$.}}\label{fig71a}
\end{minipage}
\begin{minipage}{8cm}
\centerline{\includegraphics[scale=.8]{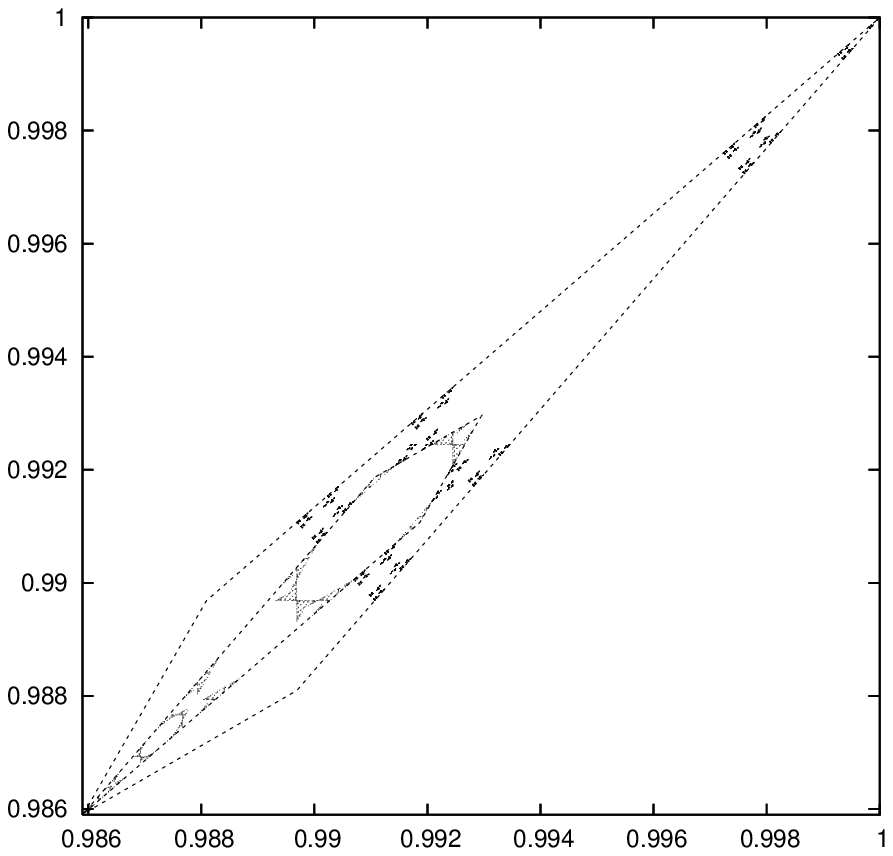}}
\caption{\mbox{Aperiodic points in $\mathcal D_1\cup\mathcal D_2$, 
$\lambda=\sqrt3$.}}\label{fig72a}
\end{minipage}
\end{figure}

\section{The case $\lambda=-\sqrt3=-2\cos\frac\pi6$} 

\begin{figure}
\includegraphics{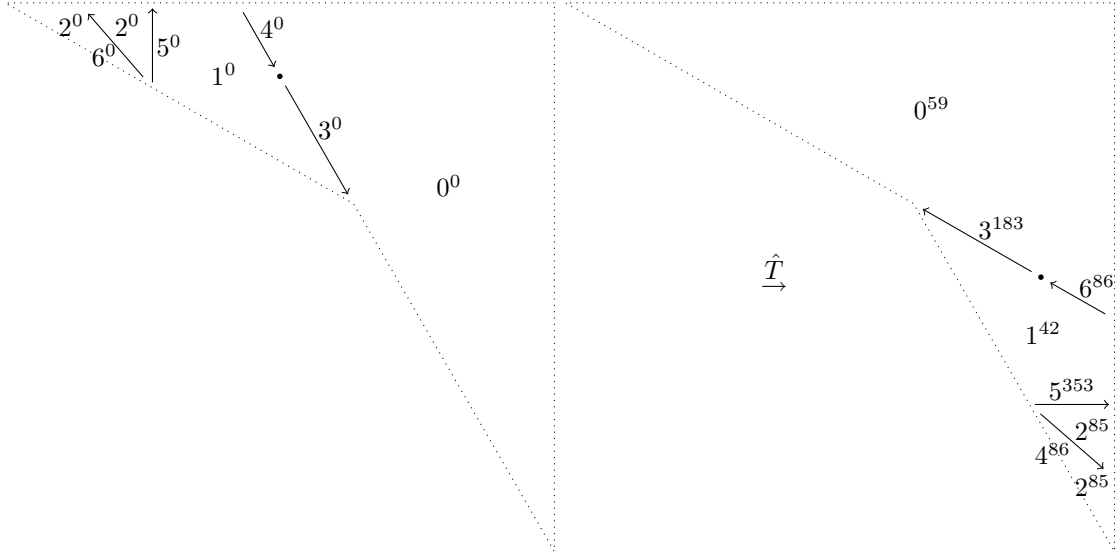}
\caption{The map $\hat T$ on $\mathcal D$, $\lambda=-\sqrt3$. 
($\ell^k$ stands for $T^k(D_\ell)$.)} \label{figThat8}
\end{figure}

Let $\mathcal D=\{(x,y)\in[0,1)^2:x+\sqrt3y>5\sqrt3-6\mbox{ or }
y+\sqrt3x>5\sqrt3-6\}$, $U_1$ as in Section~\ref{sectsqrt3} and 
$$
U(z)=U_1^2(z)=(2-\sqrt3)^2z+(4\sqrt3-6,4\sqrt3-6)=V^{-1}(\kappa V(z)),
$$
$\kappa=(2-\sqrt3)^2$, $V(z)=\big((1,1)-z\big)/\kappa$.
Then we have
$$
V(\mathcal D)=\{(x,y):\ x>0,\ y>0,\ x+\sqrt3y<1\mbox{ or }y+\sqrt3x<1\}.
$$
Figure~\ref{figThat8} shows the first return map $\hat T$ on $\mathcal D$, 
which is determined in Appendix~\ref{app8}.
The sets $D_0,\ldots,D_6$ satisfy the (in)equalities
\begin{gather*}
\begin{array}{c|c|c}
V(D_0) & V(D_1) & V(D_2) \\ \hline
\sqrt3x+y<1 & \sqrt3x+y>1,\,x<\sqrt3-1 & x>\sqrt3-1,\,2x+\sqrt3y\ne\sqrt3 \\
\end{array} \\
\begin{array}{c|c|c|c}
V(D_3) & V(D_4) & V(D_5) & V(D_6) \\ \hline
\sqrt3x+y=1,\,x<1/2 & \sqrt3x+y=1,\,x>1/2 & x=\sqrt3-1 & 2x+\sqrt3y=\sqrt3
\end{array} 
\end{gather*}
The remaining point $z=V^{-1}(1/2,1-\sqrt3/2)$ has return time $183$ and
satisfies $\hat T^{10}(z)=z$.

\begin{figure}
\includegraphics[scale=.99]{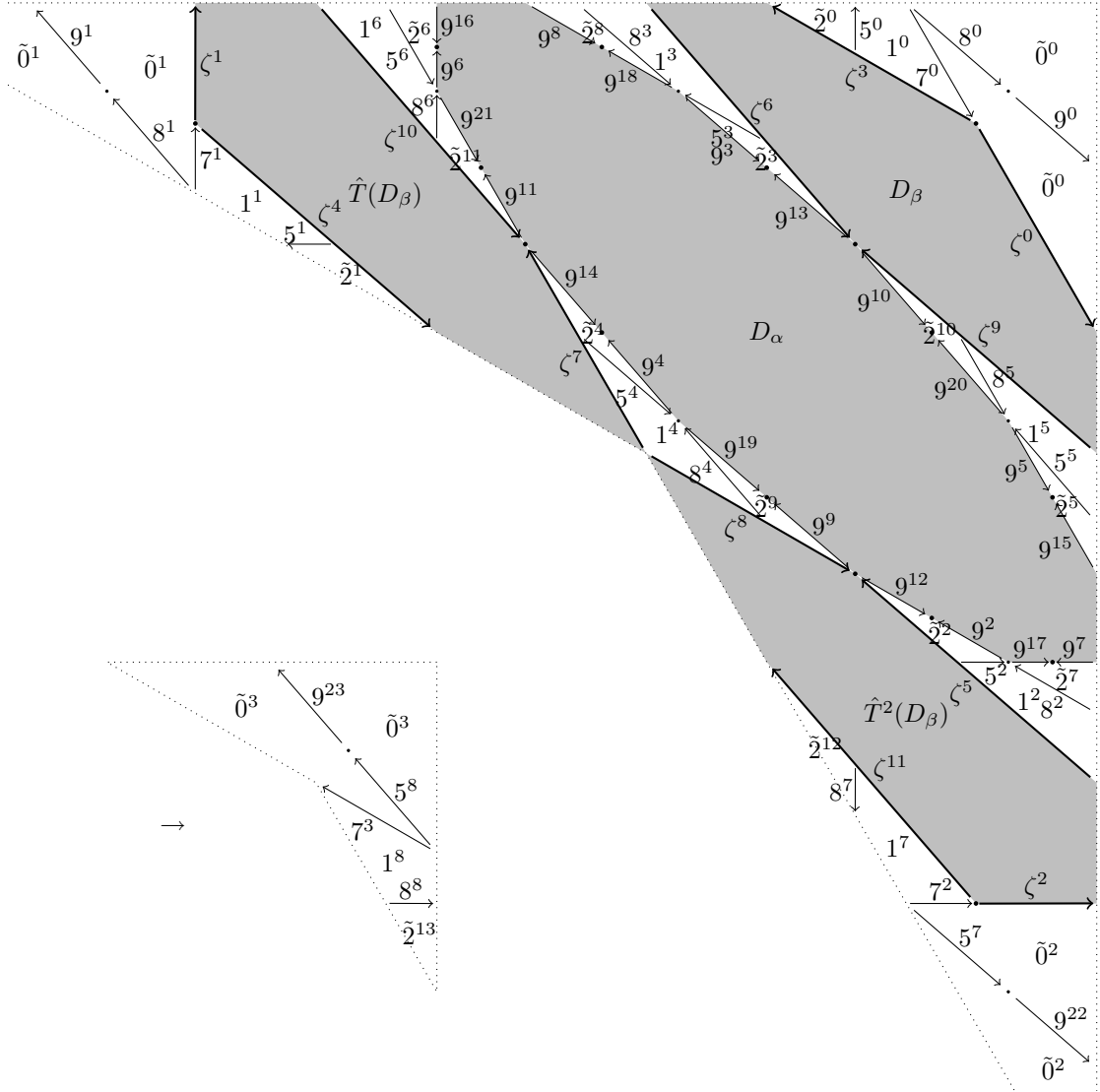}
\caption{Trajectory of $U_1(\mathcal D)$ and large parts of $\mathcal P$,
$\lambda=-\sqrt3$. ($\ell^k$ stands for $\hat T^kU_1(D_\ell)$.)} \label{figP81}
\end{figure}

\begin{figure}
\includegraphics[scale=.99]{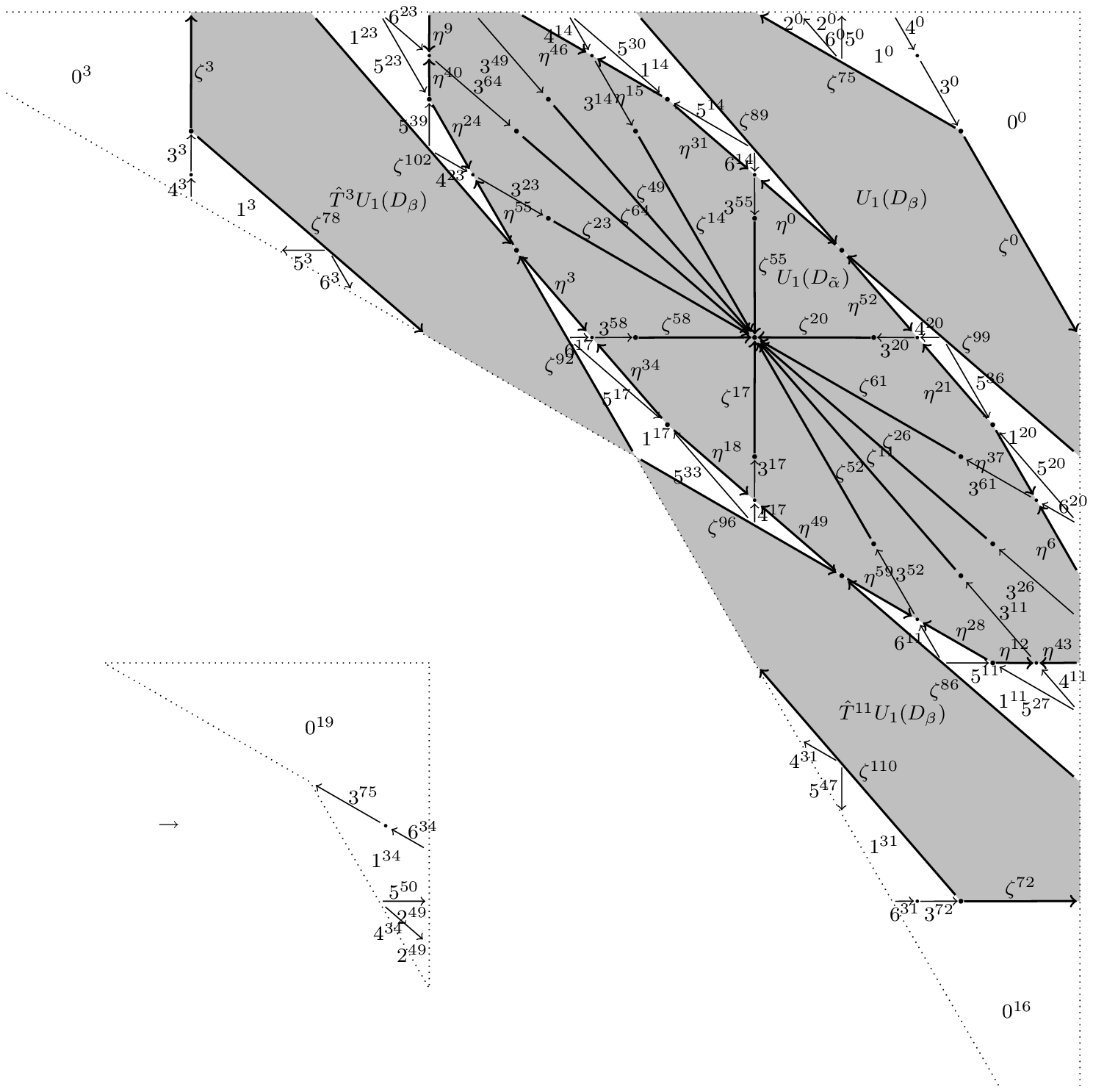}
\caption{Trajectory of $U(\mathcal D)$ and small parts of $\mathcal P$,
$\lambda=-\sqrt3$. ($\ell^k$ stands for $\hat T^kU(D_\ell)$.)} \label{figP82}
\end{figure}

Figure~\ref{figP81} shows that the first return map on $U_1(\mathcal D)$ 
differs from $U_1\hat TU_1^{-1}$ on several lines.
Therefore we add the lines $D_7,D_8,D_9$ satisfying the following 
(in)equalities
$$
\begin{array}{c|c|c}
V(D_7) & V(D_8) & V(D_9) \\ \hline
\sqrt3x+y=1 & \sqrt3x+2y=1,\,x>2-\sqrt3 & \sqrt3x+2y=1,\,x<2-\sqrt3
\end{array}
$$
and define $D_{\tilde0}=D_0\setminus V^{-1}(\{(x,y):\sqrt3x+2y=1\})$,
$D_{\tilde2}=D_2\cup D_6$.
For $z\in D_{\tilde\ell}$, $\ell\in\{0,2\}$ and $z\in D_\ell$, $\ell=1$, we 
have $\hat T^{|\sigma_1(\ell)|}U_1(z)=U_1\hat T(z)$ with
$$
\sigma_1: 
\begin{array}{rclrclrclrcl} 
0 & \mapsto & 020 & 1 & \mapsto & 010^410 & 2 & \mapsto & 010^910 & & & \\ 
5 & \mapsto & 010^440 & 7 & \mapsto & 050 & 8 & \mapsto & 060^410 & 9 & \mapsto 
& 060^930^940 
\end{array}
$$
Figure~\ref{figP82} shows that the substitution $\sigma$ given by 
$\sigma(\ell)=\sigma_1\sigma_2(\ell)$ with
$$
\sigma_2: 
\begin{array}{rclrclrclrcl} 
0 & \mapsto & 020 & 1 & \mapsto & 010^410 & 2 & \mapsto & 010^910 & & & \\ 
3 & \mapsto & 050^590^580 & 4 & \mapsto & 050^410 & 5 & \mapsto & 010^470^410 & 
6 & \mapsto & 010^480 
\end{array}
$$
satisfies the conditions in Section~\ref{sectgeneral} (with $\varepsilon=1$).
The coding of the return path of the remaining point is $\sigma_1(050^470^480)$.

\begin{theorem}
$\pi(z)$ is finite for all $z\in(\mathbb Z[\sqrt3]\cap[0,1))^2$, but
$\pi(V^{-1}(2/7,\sqrt3/7+1/7)=\infty$.
\end{theorem}

\begin{proof}
First we show that all points on the lines $U_1^n(D_\ell)$, 
$\ell\in\{3,\ldots,9\}$, $n\ge0$, are periodic.
The only possibly aperiodic part of $D_5$ is $\hat TU_1(D_7)$, and the only
possibly aperiodic part of $U_1(D_7)$ is $\hat T^{23}U_1^2(D_5)$.
Inductively, the set of aperiodic points in $D_5$ converges to 
$V^{-1}(\sqrt3-1,1-1/\sqrt3)\not\in D_5$ and is therefore empty.
Therefore, all points in $U^n(D_5)$ and $U^nU_1(D_7)$ are periodic.
Similar arguments show that all points in $U^n(D_3)$ in $U^nU_1(D_9)$ are 
periodic, then the same holds for $U^n(D_4)$ and $U^nU_1(D_5)$, for $U^n(D_6)$ 
and $U^nU_1(D_8)$, and finally for $U^n(D_8)$ and $U^nU_1(D_6)$.
Then it is clear that all points in $U^nU_1(D_3\cup D_4)$ and 
$U^n(D_7\cup D_9)$ are periodic as well.

Therefore we can limit our considerations to 
$U_1^n(D_{\tilde 0}\cup D_1\cup D_2)$, and consider the scaling map $U_1$ 
instead of $U$.
If we define $\hat s_1(z),s_1(z)$ and $t_1(z)$ accordingly, we obtain:
\begin{align*}
z\in\hat T^{-1}U_1(\mathcal D):\ & \hat s_1(z)=1,\,
s_1(z)\equiv 11\bmod 12,\, t_1(z)=V(\hat T(z))-V(z)A^{-1}=(1,0) \\
z\in\hat T^6U_1(D_1)\cup\hat T^{11}U_1(D_2):\ & \hat s_1(z)=2,\, 
s_1(z)\equiv 5\bmod 12,\, t_1(z)=(-1,\sqrt3-1) \\
z\in\hat T^5U_1(D_1)\cup\hat T^{10}U_1(D_2):\ & \hat s_1(z)=3,\, 
s_1(z)\equiv 4\bmod 12,\, t_1(z)=(\sqrt3-1,\sqrt3-2) \\
z\in\hat T^4U_1(D_1)\cup\hat T^9U_1(D_2):\ & \hat s_1(z)=4,\, 
s_1(z)\equiv 3\bmod 12,\, t_1(z)=(\sqrt3-1)(-\sqrt3,2) \\
z\in\hat T^8U_1(D_2):\ & \hat s_1(z)=5,\, s_1(z)\equiv 2\bmod 12,\, 
t_1(z)=(2-\sqrt3)(\sqrt3,-2) \\
z\in\hat T^7U_1(D_2):\ & \hat s_1(z)=6,\, s_1(z)\equiv 1\bmod 12,\, 
t_1(z)=(2\sqrt3-4,3\sqrt3-4) 
\end{align*}
For the remaining $z$, $\hat s_1(z),s_1(z)$ and $t_1(z)$ are given 
symmetrically.
By looking at the following sets $\{t_1(z)A^h:h\in\mathbb Z\}$, 
we obtain $\delta_1=(3\sqrt3+4)/(\sqrt3+1)=(5+\sqrt3)/2$:
\begin{gather*}
\pm\{(1,0),\,(0,1),\,(1,-\sqrt3),\,(-\sqrt3,2),\,(2,-\sqrt3),\,(-\sqrt3,1)\},\\
\pm\{(1,1-\sqrt3),\,(1-\sqrt3,2-\sqrt3),\,(2-\sqrt3,2-\sqrt3),\,
(2-\sqrt3,1-\sqrt3),\,(1-\sqrt3,1),\,(1,-1)\}, \\
\pm\{(2\sqrt3-4,3\sqrt3-4),\,(3\sqrt3-4,2\sqrt3-5),\,(2\sqrt3-5,2\sqrt3-2),\\
(2\sqrt3-2,-1),\,(-1,\sqrt3-2),\,(2-\sqrt3,4-2\sqrt3)\}.
\end{gather*}

The only $x\in\mathbb Z[\sqrt3]$ with $0<x<1$ and $|x'|\le(5+\sqrt3)/2$ is 
$\sqrt3-1$.
Therefore no point $z\in V(\mathbb Z[\sqrt3]^2\cap\mathcal D)$ satisfies
$\|z'\|_\infty\le\delta_1$, and Conjecture~\ref{cj} holds for $\lambda=-\sqrt3$.

If $V(z)=(2/7,\sqrt3/7+1/7)$, then we have
\begin{align*}
VS_1(z) & = (2+\sqrt3)\big(V(z)A^3+(\sqrt3-1)(-\sqrt3,2)=
(3\sqrt3/7-5/7,5\sqrt3/7-3/7), \\
VS_1^2(z)&=(2+\sqrt3)\big(VS_1(z)A^{11}+(1,0)=(\sqrt3/7+2/7,\sqrt3/7-1/7), \\
VS_1^3(z)&=(2+\sqrt3)\big(VS_1^2(z)A^5+(-1,\sqrt3-1)=(\sqrt3/7-1/7,3\sqrt3/7),
\end{align*}
$VS_1^4(z)=(2+\sqrt3)\big(VS_1^3(z)A^{11}+(1,0)=(2/7,\sqrt3/7+1/7)=V(z)$.
\end{proof}

\noindent{\it Remark.}
The eigenvalues corresponding to the primitive part of $\sigma_1$ 
($\ell\in\{0,1,2\}$) are $5,-2$ and $1$.

\begin{figure}
\begin{minipage}{6.6cm}
\centerline{\includegraphics[scale=.8]{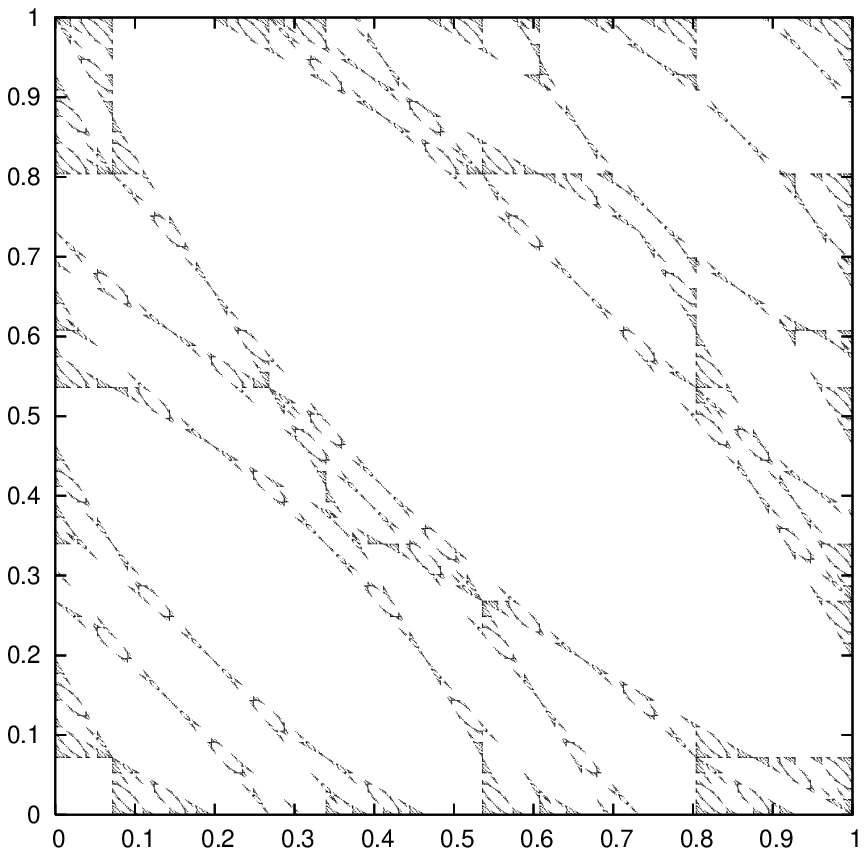}}
\caption{\mbox{Aperiodic points, $\lambda=-\sqrt3$.}}\label{fig81a}
\end{minipage}
\begin{minipage}{8cm}
\centerline{\includegraphics[scale=.8]{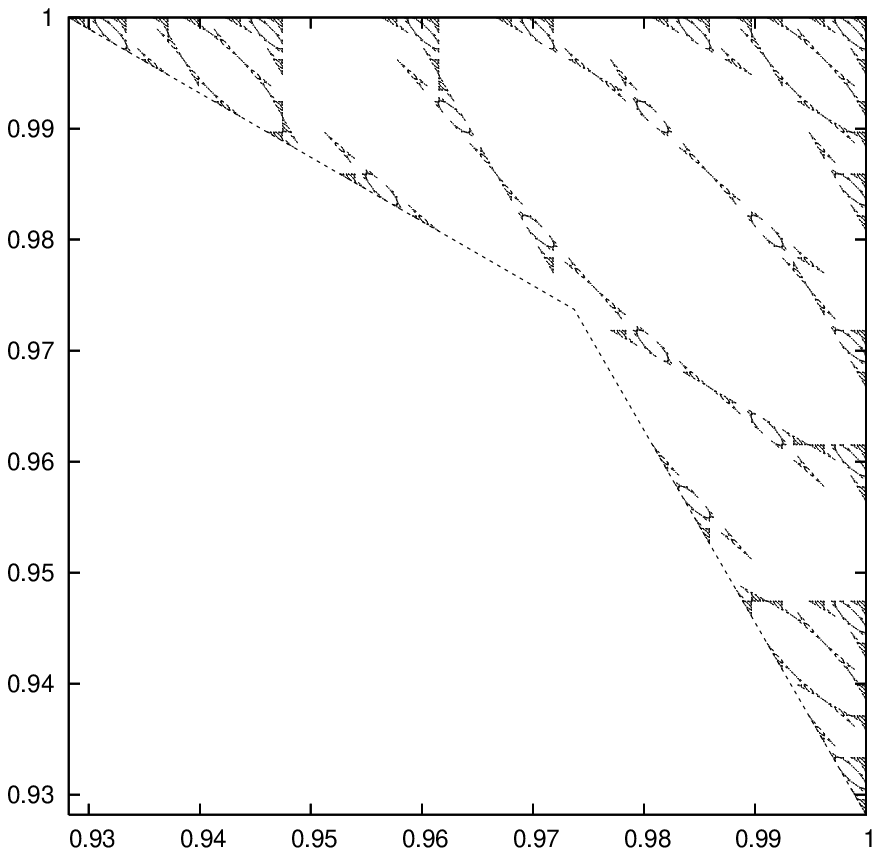}}
\caption{\mbox{Aperiodic points in $\mathcal D$, $\lambda=-\sqrt3$.}}
\label{fig82a}
\end{minipage}
\end{figure}

\section{The Thue-Morse sequence, the golden mean and $\sqrt3$}

We conclude by exhibiting a relation between the Thue-Morse sequence
and substitutions we used in golden mean cases (see~\cite{AS} for a
survey on links between fractal objects and automatic sequences).
The Thue-Morse sequence is a fixed point of the substitution
$0\mapsto01$, $1\mapsto10$:
$$
0\,1\,10\,1001\,10010110\,1001011001101001\,
10010110011010010110100110010110\cdots
$$
It can be written as
$$
0^11^20^11^10^21^20^21^10^11^20^11^10^21^10^11^20^21^20^11^10^21^20^2
1^10^11^20^21^20^11^10^21^10^11^20^11^10^21^20^21^10^11^20^1\cdots
$$
By subtracting $1$ from each term of the sequence of exponents (the
run-lengths of $0$'s and $1$'s) we obtain the sequence
$$
0\,10\,01110010\,01001110011100111001001001110010\cdots
$$
which is easily shown to be the fixed point of the substitution $0\mapsto010$, $1\mapsto01110$ (see~\cite{ASh}), which is equal to 
$\sigma$ in the cases
$\lambda=-1/\gamma$, $\lambda=1/\gamma$, $\lambda=-\gamma$, and to $\sigma_1$
in the case $\lambda=\sqrt3$.
In case $\lambda=\gamma$, we have that $\sigma^\infty(1)$ is the image of
this word by the morphism $0\mapsto10$, $1\mapsto110$ since
$\sigma(10)=(10)(110)(10)$ and $\sigma(110)=(10)(110)(110)(110)(10)$.

\bigskip\noindent
{\bf Acknowledgments}. We thank Professors Nikita Sidorov and
Franco Vivaldi for valuable hints and for drawing our attention to several 
references.
The second author wishes to express his heartfelt thanks to the members of the 
LIAFA for their hospitality in December 2006.
The third author was supported partially by the Hungarian National Foundation 
for Scientific Research Grant No.~T67580.
The fourth author was supported by the grant ANR-06-JCJC-0073 of the
French Agence Nationale de la Recherche.

\medskip\noindent
{\footnotesize
\sc Dep. of Mathematics, Faculty of Science Niigata University, Ikarashi 
2-8050, Niigata 950-2181, Japan \\
\tt akiyama@math.sc.niigata-u.ac.jp 

\smallskip\noindent
\sc Haus-Endt-Stra{\ss}e 88, D-40593 D\"usseldorf, Germany \\
\tt brunoth@web.de 

\smallskip\noindent
\sc Department of Computer Science, University of Debrecen, P.O. Box 12, H-4010 
Debrecen, Hungary \\
\tt pethoe@inf.unideb.hu 

\smallskip\noindent
\sc LIAFA, CNRS, Universit\'e Paris Diderot -- Paris 7, Case 7014, 75205 Paris Cedex 13, France \\
\tt steiner@liafa.jussieu.fr}

\newpage
\appendix
\section{The map $\hat T$ for $\lambda=\sqrt3$.}\label{app7}

\begin{figure}
\includegraphics[scale=.99]{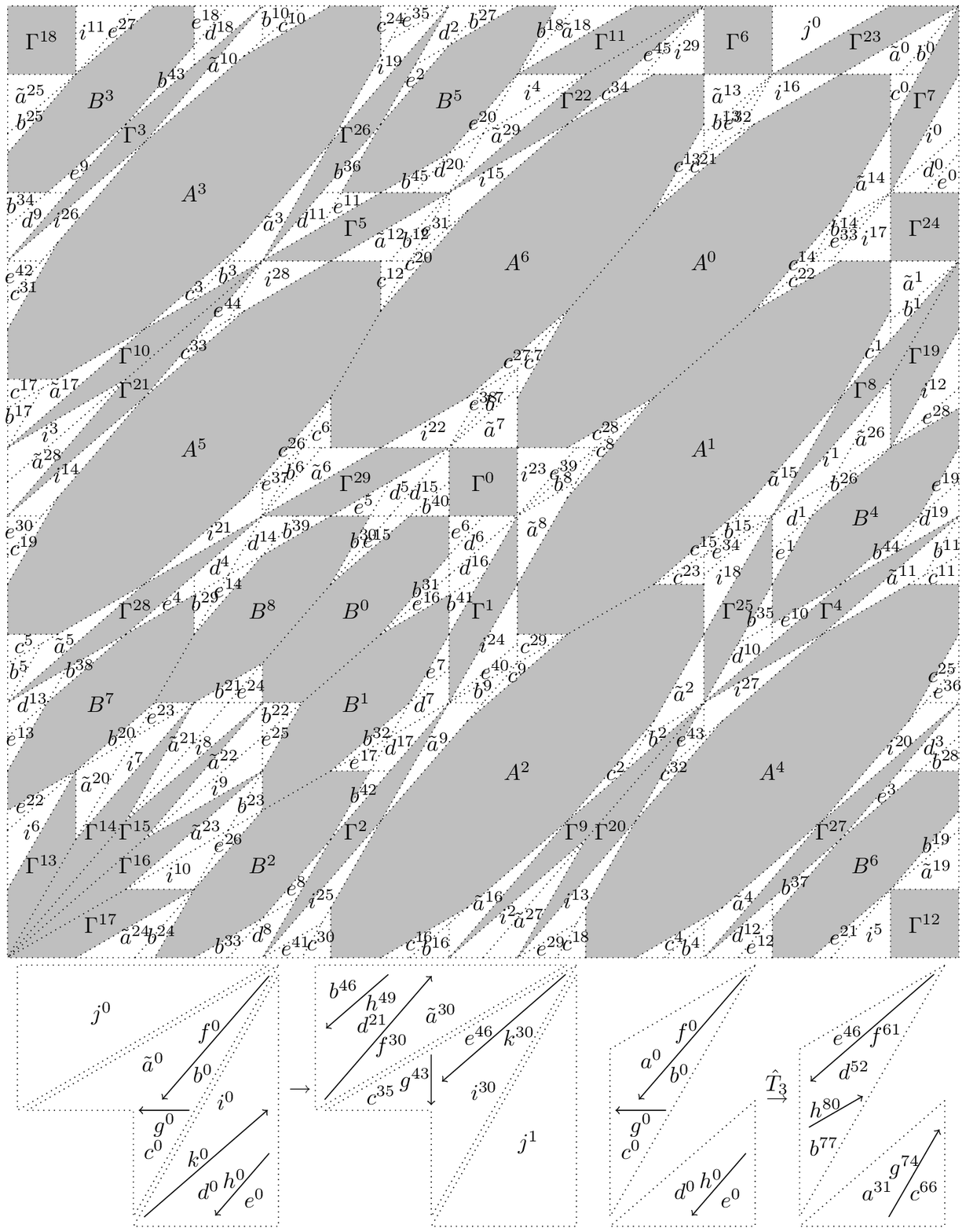}
\caption{The first return map $\hat T_3$ and large parts of $\mathcal R$,
$\lambda=\sqrt3$. 
($\ell^k$ stands for $T^k(D_\ell)$.)} \label{fig77}
\end{figure}

\begin{figure}
\includegraphics[scale=.95]{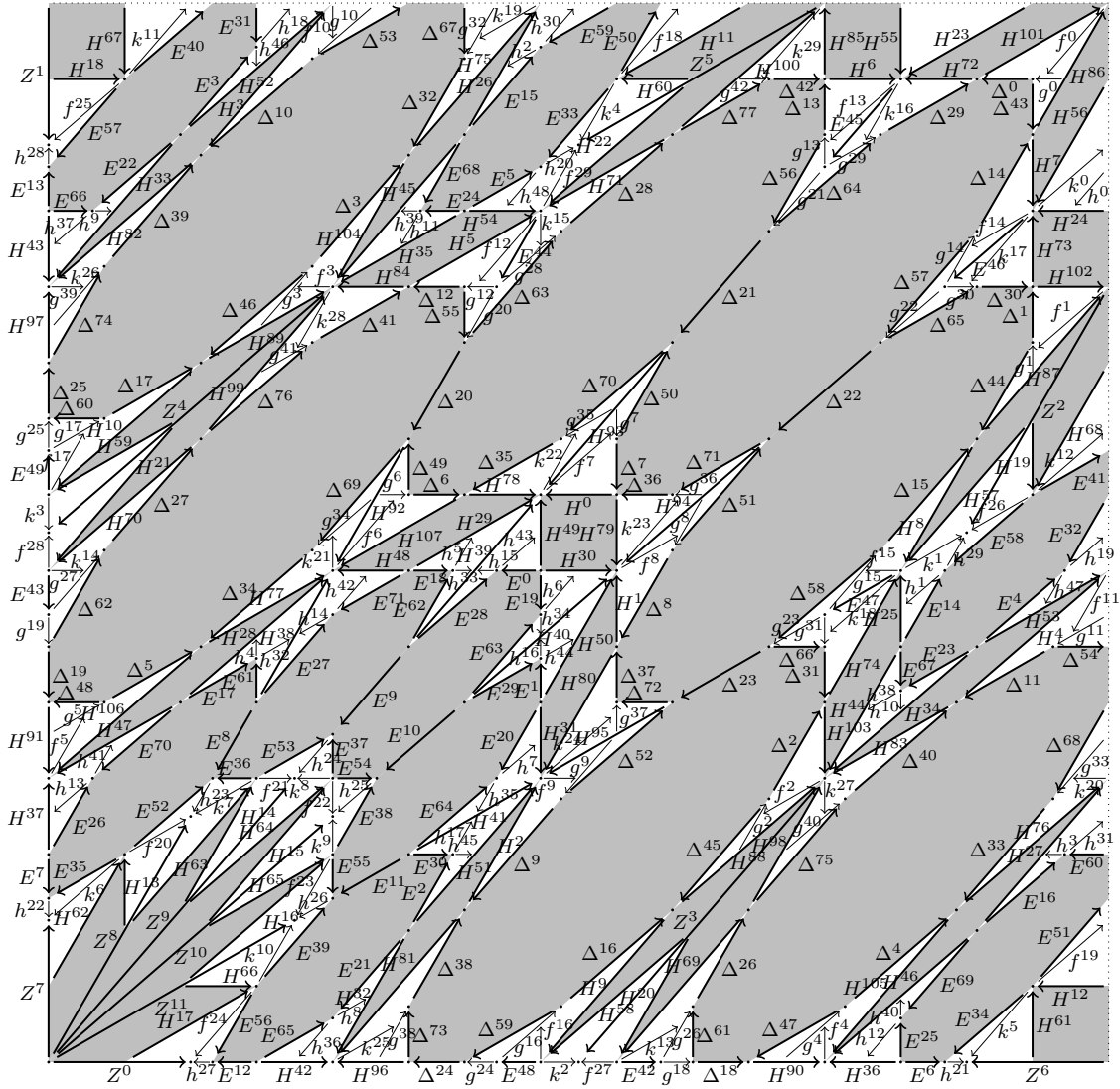}
\caption{Trajectories of long lines in $\mathcal R$ and $D_f,D_g,D_h,D_k$,
$\lambda=\sqrt3$. 
($\ell^k$ stands for $T^k(D_\ell)$.)} \label{fig78}
\end{figure}

\begin{figure}
\includegraphics[scale=.95]{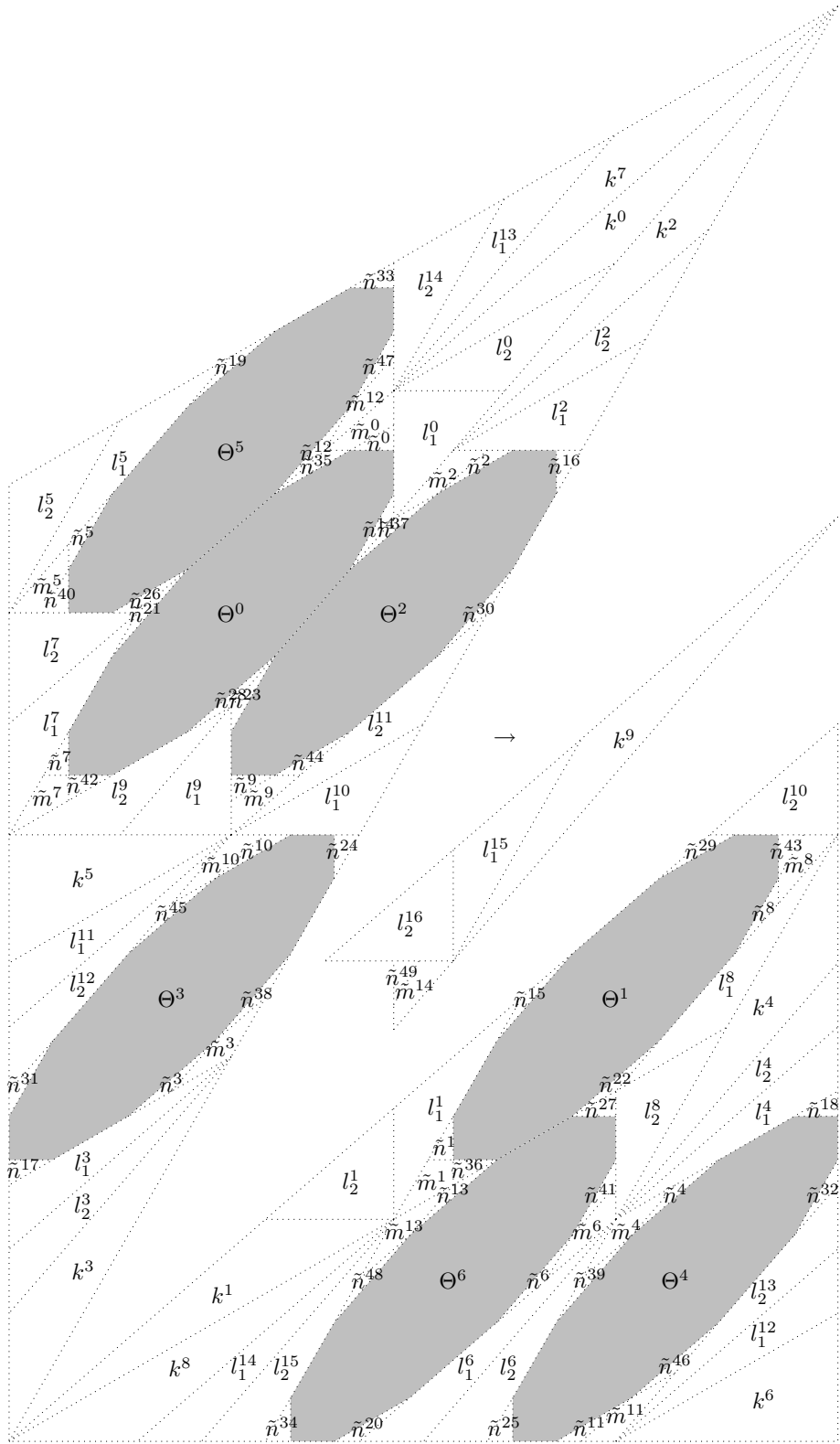}
\caption{An intermediate first return map, $\lambda=\sqrt3$. 
($\ell^k$ stands for $\hat T_3^k(D_\ell)$.)} \label{fig79}
\end{figure}

\begin{figure}
\includegraphics[scale=.95]{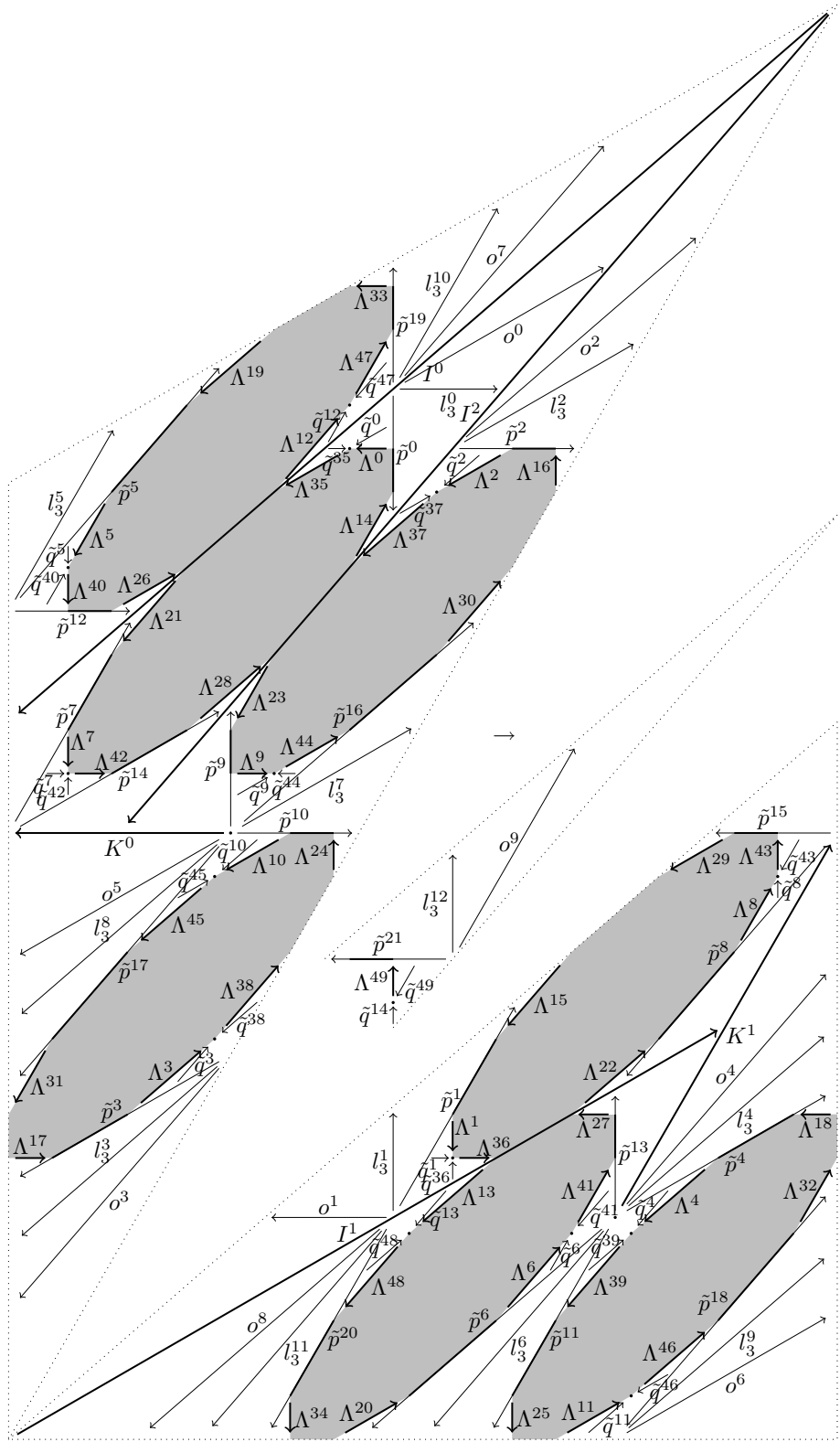} 
\caption{The trajectory of the lines, $\lambda=\sqrt3$. 
($\ell^k$ stands for $\hat T_3^k(D_\ell)$.)} \label{fig710}
\end{figure}

\begin{figure}
\includegraphics[scale=.9]{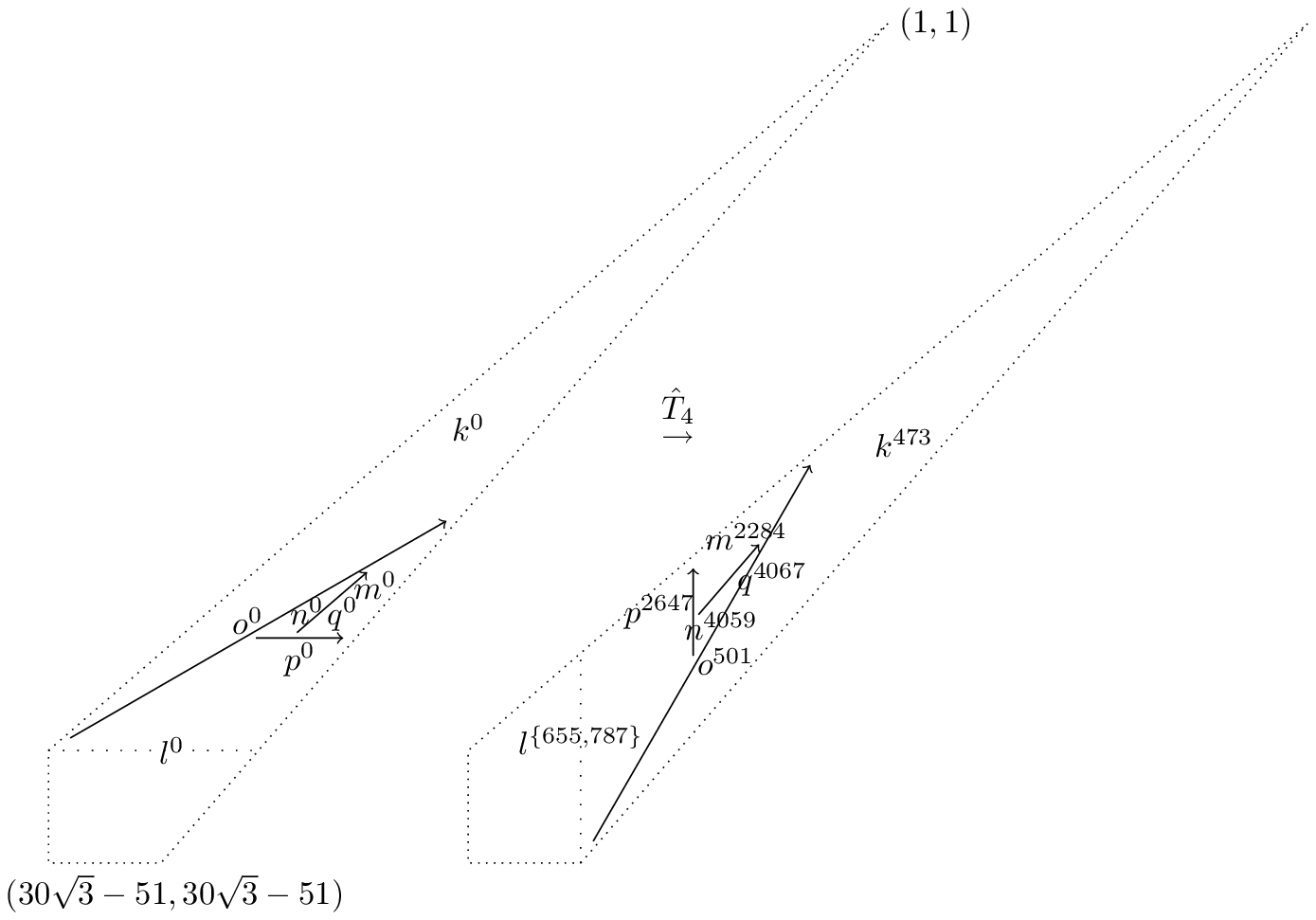}
\caption{The map $\hat T_4$, $\lambda=\sqrt3$. 
($\ell^k$ stands for $T^k(D_\ell)$.)} \label{fig711}
\end{figure}

\begin{figure}
\includegraphics{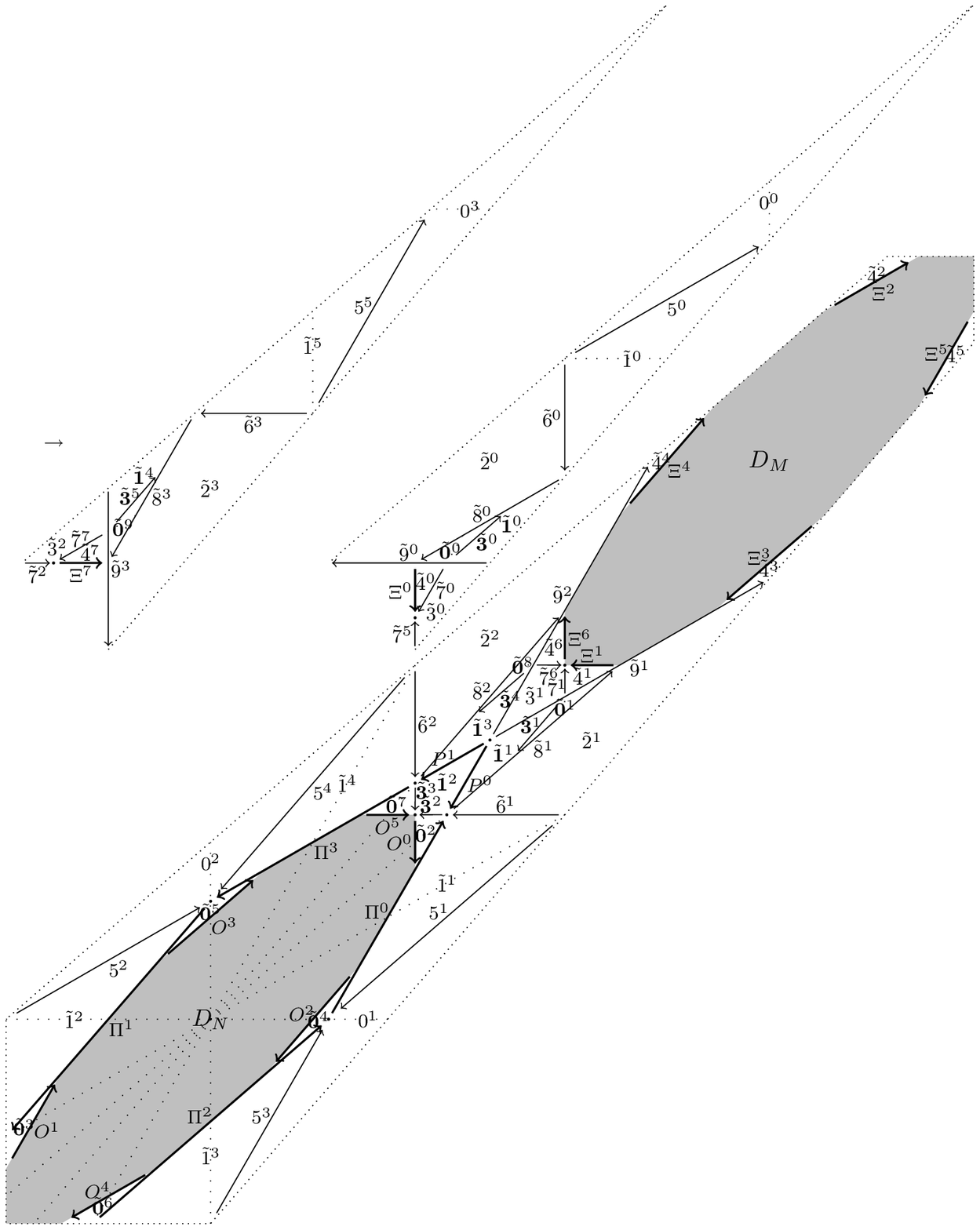}
\caption{Almost the map $\hat T$, $\lambda=\sqrt3$. 
($\ell^k$ stands for $\hat T_4^k(D_\ell)$.)} \label{fig712}
\end{figure}

As the scaling domain $\mathcal D$ is very small in case $\lambda=\sqrt3$, the
determination of $\hat T$ is done in several steps.
Figure~\ref{fig77} shows the action of $\hat T_3$, which is the first return 
map on the domain
$$
\{x>4\sqrt3-6,2y<\sqrt3x+2-\sqrt3,y>\sqrt3x+1-\sqrt3\}\cup
\{x<1,y>12\sqrt3x-20,y>\sqrt3x+6\sqrt3-11\},
$$
on sets $D_a,\ldots,D_h$.
To this end, we first determine the trajectory of sets 
$D_{\tilde a},D_b,\ldots,D_k$, which partition a symmetric version of this 
domain.
Figure~\ref{fig77} shows the trajectory of the open sets 
$D_{\tilde a},D_b,D_c,D_d,D_e,D_i,D_j$, Figure~\ref{fig78} completes the 
picture with the trajectories of the lines $D_f,D_g,D_h,D_k$. 
All points which are not on these trajectories are periodic.
>From the symmetric first return map, it is easy to determine $\hat T_3$.

Next, we consider the first return map on 
$$
\{(x,y):\,2y<\sqrt3x+2-\sqrt3,\,2x<\sqrt3y+2-\sqrt3,\,x\ge30\sqrt3-51\mbox{ or }
y\ge30\sqrt3-51\}
$$
in Figures~\ref{fig79} and~\ref{fig710}, partitioned into open sets 
$D_k,D_{l_1},D_{l_2},D_{\tilde m},D_{\tilde n}$ and lines 
$D_{l_3},D_o,D_{\tilde p},D_{\tilde q}$. 
>From this map, we easily obtain the first return map $\hat T_4$ on $\{(x,y):\,
2y<\sqrt3x+2-\sqrt3,\,2x<\sqrt3y+2-\sqrt3,\,x\ge30\sqrt3-51\mbox{ and }
y\ge30\sqrt3-51\}$, which is partitioned into the sets $D_l,\ldots,D_q$.
Observe that the return time on $D_l$ is not constant since the trajectories of 
the three parts $D_{l_1},D_{l_2},D_{l_3}$ are different. 
This implies that the return times on $D_0,D_1$ and $D_{\mathbf 0}$ are
not constant. 

Finally, we consider in Figure~\ref{fig712} the first return map on
$$
\{(x,y):\,2y<\sqrt3x+2-\sqrt3,\,2x<\sqrt3y+2-\sqrt3,\,x\ge72-41\sqrt3\mbox{ or }
y\ge72-41\sqrt3\},
$$
partitioned into sets $D_0,D_{\tilde 1},\ldots,D_{\tilde 9},
D_{\tilde{\mathbf 0}},D_{\tilde{\mathbf 1}},D_{\tilde{\mathbf 3}}$, from which 
it is easy to deduce $\hat T$ on $\mathcal D$.

\section{The map $\hat T$ for $\lambda=-\sqrt3$.}\label{app8}

\begin{figure}
\includegraphics[scale=.99]{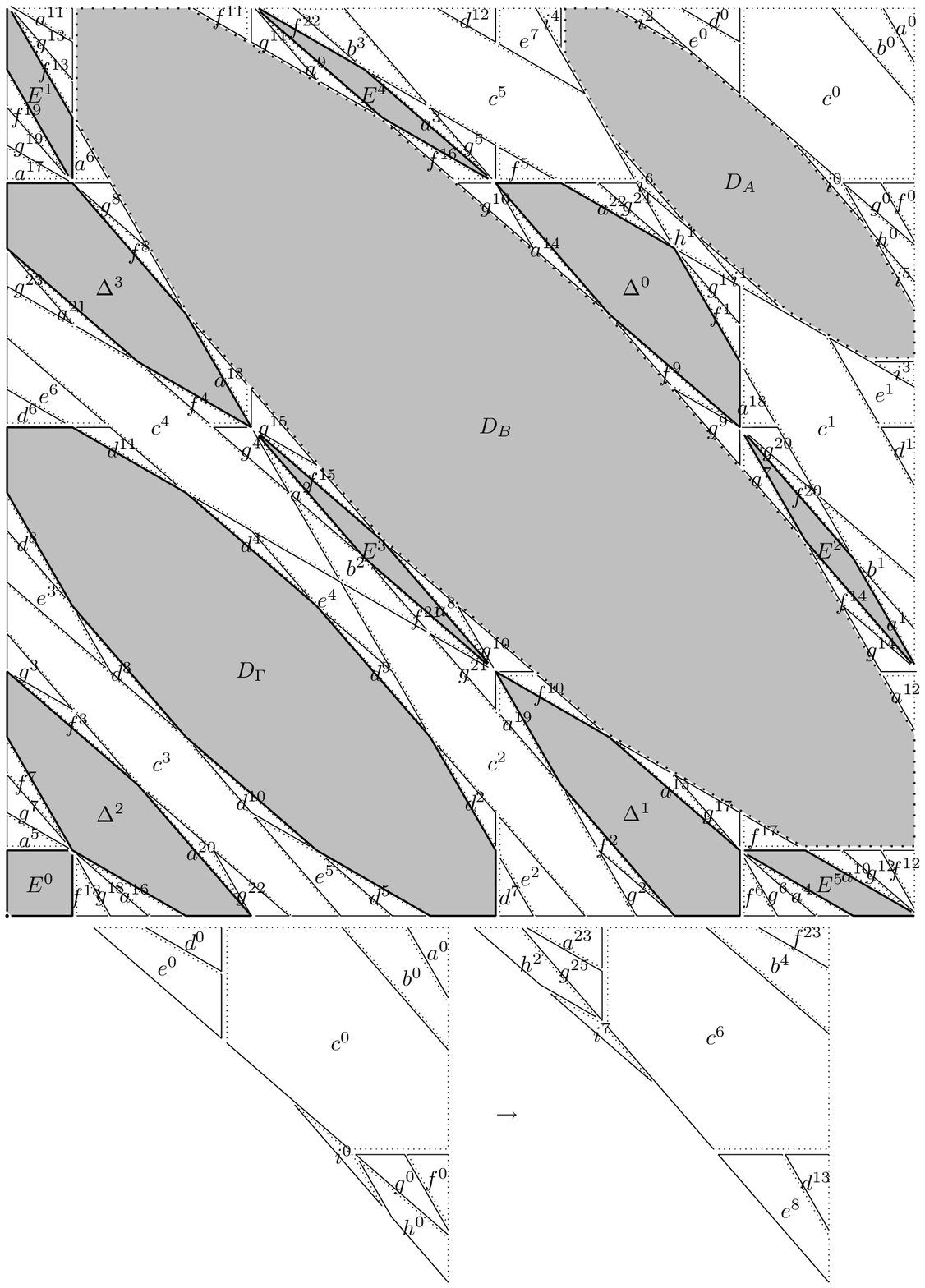}
\caption{A first return map and large parts of $\mathcal R$, $\lambda=-\sqrt3$. 
($\ell^k$ stands for $T^k(D_\ell)$.)} \label{fig84}
\end{figure}

\begin{figure}
\includegraphics[scale=.99]{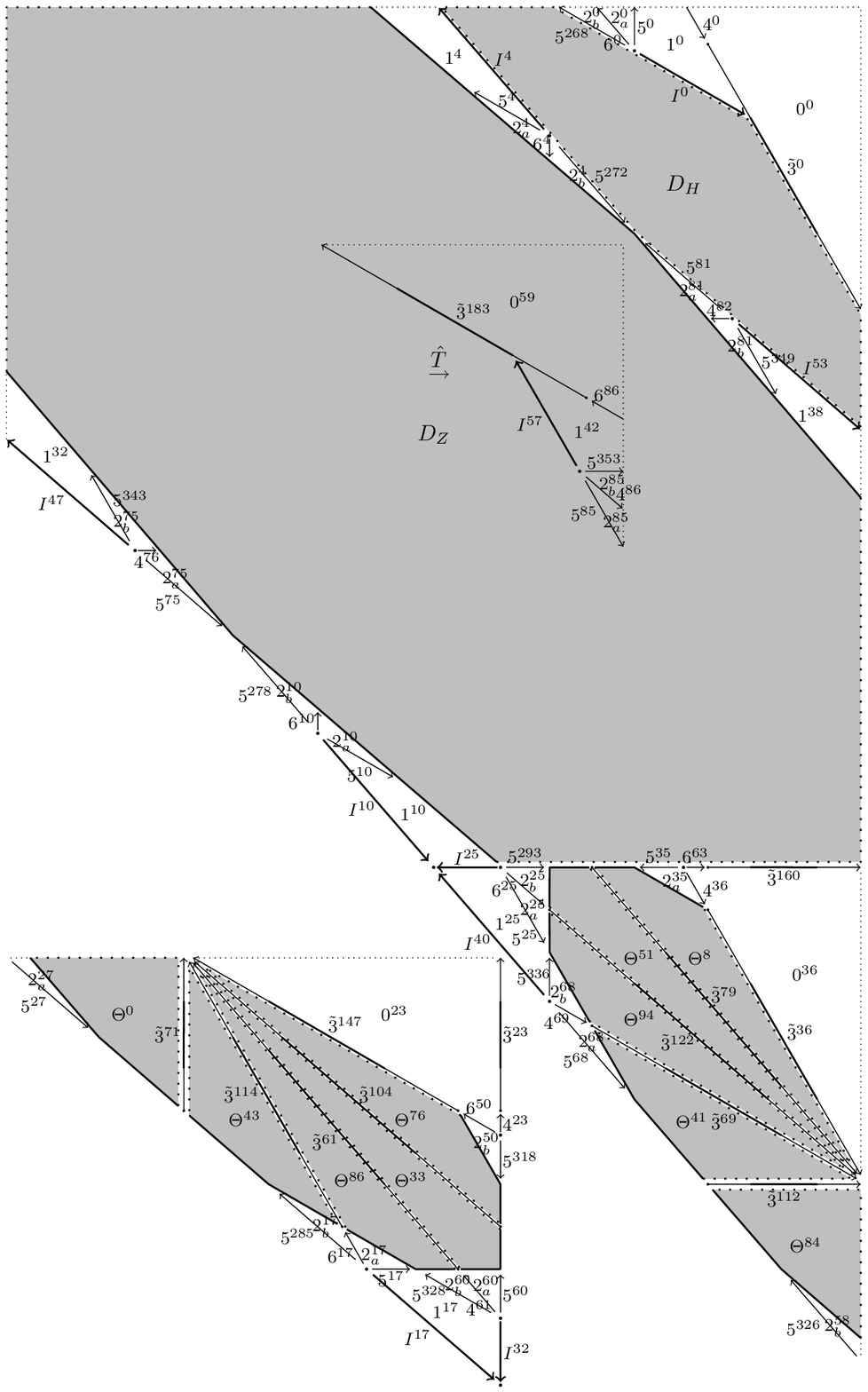}
\caption{The map $\hat T$ and small parts of $\mathcal R$, $\lambda=-\sqrt3$. 
($\ell^k$ stands for $T^k(D_\ell)$.)} \label{fig85}
\end{figure}

For $\lambda=-\sqrt3$, we consider in Figure~\ref{fig84} the first return map 
on
$$
\{(x,y)\in[0,1)^2:\,2x+\sqrt3y>3\sqrt3-2\mbox{ or }\sqrt3x+2y>3\sqrt3-2\},
$$
partitioned into sets $D_a,\ldots,D_i$.
Figure~\ref{fig85} provides the first return map $\hat T$ on $\mathcal D$.
Again, all points in $\mathcal R$ are periodic.

\end{document}